%% file: journal_main.tex
\documentclass{article}
\usepackage[margin=0.9in]{geometry}
\usepackage{setspace}
\usepackage{amssymb}% http://ctan.org/pkg/amssymb
\usepackage{pifont}% http://ctan.org/pkg/pifont
\newcommand{\xmark}{\ding{55}}%
\usepackage{multirow}
\usepackage{colortbl}
\usepackage{threeparttable}
\usepackage{booktabs}
\usepackage{makecell}
\usepackage{microtype}
\usepackage{graphicx}
\usepackage{subfigure}
\usepackage{booktabs}
\usepackage{hyperref}

\usepackage{natbib}
\usepackage{amsmath}
\usepackage{amssymb}
\usepackage{mathtools}
\usepackage{amsthm}
\usepackage{amsfonts}

\allowdisplaybreaks[1]

\usepackage[capitalize,noabbrev]{cleveref}

\theoremstyle{plain}
\newtheorem{theorem}{Theorem}[section]
\newtheorem{proposition}[theorem]{Proposition}
\newtheorem{lemma}[theorem]{Lemma}
\newtheorem{corollary}{Corollary}[theorem]
\theoremstyle{definition}

\newtheorem{assumption}{Assumption}
\theoremstyle{remark}
\newtheorem{remark}{Remark}[section]

\usepackage{todonotes}

\usepackage[algo2e,lined,boxed,commentsnumbered,ruled]{algorithm2e}
\usepackage{setspace}
\RestyleAlgo{ruled} 

\usepackage{enumitem}
\usepackage{soul}
\usepackage{nicefrac}       % compact symbols for 1/2, etc.
\usepackage{textcomp}

\usepackage{comment}

\definecolor{mygray}{gray}{0.8}

\renewcommand{\arraystretch}{1.2} % Default value: 1.2

\newcommand{\numberthis}{\addtocounter{equation}{1}\tag{\theequation}}

\newcommand{\ind}{{\bf 1}}
\newcommand{\lrset}[1]{\left\{ #1 \right\}}
\newcommand{\lrp}[1]{\left( #1 \right)}
\newcommand{\lrs}[1]{\left[ #1 \right]}
\newcommand{\abs}[1]{\left|{#1}\right|}

\newcommand{\E}[1]{\mathbb{E}\!\lrs{#1}}
\newcommand{\Exp}[2]{\mathbb{E}_{#1\!}\!\lrs{#2}}

\newcommand{\V}{\operatorname{Var}}

\newcommand{\R}{\mathbb{R}}
\newcommand{\N}{\mathbb{N}}

\usepackage{xspace}

\newcommand{\G}{\mathcal N}
\newcommand{\calA}{{\mathcal A}}
\newcommand{\calX}{{\mathcal X}}
\newcommand{\calS}{{\mathcal S}}
\newcommand{\calM}{{\mathcal M}}
\newcommand{\calY}{{\mathcal Y}}

\newcommand{\frakd}{{\mathfrak d}}
\newcommand{\frakC}{{\mathfrak C}}

\usepackage{yfonts}
\newcommand{\A}{\textsf{A}}
\newcommand{\vb}{\textgoth{b}}
\newcommand{\upb}{\textsf{H \!}}
\newcommand{\iter}{\Gamma}

\title{Markov Chain Variance Estimation: A Stochastic Approximation Approach}

\usepackage{authblk}

\author[1]{Shubhada Agrawal\footnote{This work was done when the author was a postdoctoral researcher at Georgia Institute of Technology, USA.}}
\author[2]{Prashanth L.A.}
\author[3]{Siva Theja Maguluri}
\affil[1]{Carnegie Mellon University, USA}
\affil[2]{Indian Institute of Technology, Madras, India}
\affil[3]{Georgia Institute of Technology, USA}
\date{}

\begin{document}
\maketitle

\begin{abstract}
We consider the problem of estimating the asymptotic variance of a function defined on a Markov chain, an important step for statistical inference of the stationary mean. We design a novel recursive estimator that requires $O(1)$ computation at each step, does not require storing any historical samples or any prior knowledge of run-length, and has optimal $O(\frac{1}{n})$ rate of convergence for the mean-squared error (MSE) with provable finite sample guarantees. Here, $n$ refers to the total number of samples generated. Our estimator is based on linear stochastic approximation of an equivalent formulation of the asymptotic variance in terms of the solution of the Poisson equation. 

We generalize our estimator in several directions, including estimating the covariance matrix for vector-valued functions, estimating the stationary variance of a Markov chain, and approximately estimating the asymptotic variance in settings where the state space of the underlying Markov chain is large. We also show applications of our estimator in average reward reinforcement learning (RL), where we work with asymptotic variance as a risk measure to model safety-critical applications. We design a temporal-difference type algorithm tailored for policy evaluation in this context. We consider both the tabular and linear function approximation settings. Our work paves the way for developing actor-critic style algorithms for variance-constrained RL. 
\end{abstract}

\section{Introduction}\label{sec:intro}
Consider an irreducible and aperiodic discrete time Markov chain (DTMC) on a finite state space. Given a function $f$, consider the partial sum 
$S_n(f) := \sum\nolimits_{k=0}^{n-1}f(X_k).$ Under appropriate assumptions, it is well known that $\frac{S_n}{n}$ converges almost surely to $\bar{f}$, the stationary expectation of $f$. Moreover, Central Limit Theorem (CLT) holds \cite[Chapter 21]{douc2018markov}, and implies that for a positive constant $\kappa(f)$,
\[\frac{S_n(f) - n \bar{f}}{\sqrt{n}} \Rightarrow N(0, \kappa(f)) \quad \text{ as } \quad n\rightarrow\infty,  \]
where the above convergence is in distributions, and $N(0,\kappa(f))$ denotes the Gaussian distribution with mean $0$ and variance $\kappa(f)$. Since $\kappa(f)$ corresponds to the variance in CLT, it is commonly referred to as asymptotic variance of the Markov chain. 

The focus of the paper is to estimate $\kappa(f)$, which has been a long-standing area of research. Firstly, estimates of $\kappa(f)$ play an important role in the statistical analysis of simulation outputs obtained from a Markov chain Monte Carlo experiment, for example, in statistical inference for the stationary mean, $\bar{f}$ (see, \cite{schmeiser1982batch,schruben1983confidence, carlstein1986use,kunsch1989jackknife, glynn1992asymptotic,chien1997large, flegal2010batch,wu2009recursive,nordman2009note,atchade2011kernel,dingecc2015jackknifed}). 

A second motivation comes from reinforcement learning (RL). In RL framework, the transition dynamics of the underlying MDP are unknown. The algorithm can obtain a sample of the MDP under any given policy, which specifies how actions are chosen in a given state. The traditional goal in an average reward RL problem is to find a policy that maximizes $\frac{S_n}{n}$. While the need to optimize over the average reward is well motivated, for applications in safety-critical domains, for example, healthcare or finance, it is also crucial to control adverse outcomes. {As an example, one can consider the portfolio optimization problem, where the objective is to find an investment strategy that maximizes the return (expected value), while keeping the variability under control.} As we will see later in Section~\ref{sec:motivation_var}, a natural measure of variability in this setup is the asymptotic variance $\kappa(f)$. Estimating $\kappa(f)$  is a vital sub-problem in mean-variance policy optimization, for instance, as a critic in an actor-critic framework, cf. \cite{prashanth2016variance}. 

Estimating $\kappa(f)$, while practically relevant, is particularly challenging compared to estimating variance using independent and identically-distributed (i.i.d.) samples because of temporal correlations introduced due to the underlying Markov chain structure. The problem of estimating $\kappa(f)$ has been extensively studied in literature. However, to the best of our knowledge, {most of the existing estimators have limitations of being either non-recursive in that the estimator at step $n$ cannot be expressed as a function of the various statistics computed at step $n-1$, necessitating re-computing the estimator from scratch at each step, or requiring to store a large number of historical observations, or having a sub-optimal  rate of convergence of mean-squared estimation error to $0$, or lack finite sample guarantees, or some combination of these. We discuss the related literature in details in Section~\ref{sec:lit}. In this work, our goal is to design an estimator for  $\kappa(\cdot)$ that utilizes samples from the given Markov chain along a \emph{single} trajectory, is recursive and memory-efficient, easy to compute, and has an optimal rate of convergence in mean-squared sense, with provable finite sample guarantees.

\subsection{Contributions} 
We now briefly describe the main contributions of this work. 

\begin{enumerate}
\item We propose a novel recursive estimator (or algorithm) for estimating $\kappa(f)$ that requires $O(1)$ computation at each step, does not require storing any historical samples, and enjoys optimal $O(\frac{1}{n})$ rate of convergence for the mean-squared error (MSE). Here, $n$ refers to the total number of samples generated. 

The proposed estimator does not require any prior knowledge of the run length. In addition,  we establish finite sample bounds on MSE of the proposed estimator that hold for any run length. In an earlier version of this paper \cite{pmlr-v235-agrawal24a}, we  proposed an algorithm that required generating samples along two independent trajectories of the underlying Markov chain. But this may not always be feasible. In this version, we develop a novel \emph{single-trajectory} algorithm.

\item Our algorithm can be specialized to estimate the stationary variance of a Markov chain. It has similar convergence guarantees, as well as all the other desirable properties, even in this setting. Additionally, it can be used to design a recursive estimator for the variance in i.i.d. setting, which requires $O(1)$ memory and computation at each step, and has an optimal $O(\frac{1}{n})$ rate of convergence for MSE.  

\item In contrast to prior work on DTMC variance estimation, a key idea behind our approach is to use stochastic approximation (SA) for statistical estimation problems to develop estimators with provably-good finite sample guarantees. In particular, we work with an expression for the DTMC asymptotic variance in terms of the solution of the so-called Poisson equation (these solutions are called value functions). We use linear SA on this expression to develop an algorithm for variance estimation. Our algorithm uses Temporal Difference (TD) learning from prior work for estimating the value function \cite{zhang2021finite}. {We also obtain finite sample bounds for TD learning, improving the MSE bound of \cite{zhang2021finite} from $O(\frac{\log n}{n})$ to $O(\frac{1}{n})$, which is of independent interest.}

The mean square guarantees are obtained by first establishing that the underlying matrix of our linear SA is Hurwitz, and then builds upon recent work \cite{zhang2021finite} on finite sample mean square error bounds for linear SA. A novel step in our proof is the use of Poisson equation (again) to handle Markov noise, which enables us to shave off the multiplicative $\log n$ factor in the error. 

\item We generalize our estimator in several directions. 

\begin{enumerate}
    \item First, we develop extensions of it to estimate the covariance matrix of a vector-valued function $f$. 

    \item Second, we extend our estimator to the settings where the underlying Markov chain has a large state space. Here, we consider linear  approximations of the value function using a small number of feature vectors. We then use this approximate value function to estimate the asymptotic variance.

    \item Finally, we consider the average reward RL framework in the risk-sensitive setting. We propose asymptotic variance as a natural measure of risk in contrast to the stationary variance studied in prior work. We then focus on evaluating the asymptotic variance of a given policy. We adapt our variance estimator to this setting, and obtain a novel TD like algorithm and establish its finite sample error bounds. 
\end{enumerate}
    
\end{enumerate}

\subsection{Related Literature}\label{sec:lit}
We now discuss the existing estimators for $\kappa(f)$ in details, which have been extensively studied in literature. We remark that most of these previous estimators have been studied under more general setups than considered in this work. However, in this section, we present their properties for the specific setting of finite state DTMC. Furthermore, we point out that asymptotic variance has also been referred to as time average variance constant (TAVC) in several previous works. 

%\begin{comment}

\begin{table*}[t]
\begin{center}
\caption{Summary of the results on estimators for $\kappa(f)^{[a]}$}
\label{table:results2}
\renewcommand{\arraystretch}{1}
\centering
%\vskip 0.1in
\begin{small}
\begin{tabular}{ |c|c|c|c|c|}
\toprule
\multirow{2}{*}{\centering Estimator} &  \multirow{2}{12em}{\centering Representative references (non exhaustive)} & \multirow{2}{*}{\centering MSE}& \multirow{2}{4.5 em}{\centering Recursive$^{[b]}$~} & \multirow{2}{5.5 em}{\centering Memory requirement$^{[c]}$~} \\
 & &    &  & \\
\midrule

{Batched Means (BM)} & \cite{schmeiser1982batch,chien1997large} & \multirow{3}{*}{ $O\lrp{\frac{1}{n^{2/3}}}$}  & \multirow{3}{*}{\xmark} & \multirow{3}{*}{$O(n)$}\\\cline{1-1}
\multirow{2}{*}{Overlapping BM (OBM)} & {\cite{flegal2010batch}} &   &  & \\
 & {\cite{meketon1984overlapping}} &   &  & \\

\midrule
Block bootstrap & \cite{lahiri1999theoretical,nordman2009note} & $O\lrp{\frac{1}{n^{2/3}}}$    & \xmark   & $O(n)$   \\

\midrule
Standardized & \cite{goldsman1990properties} & \multirow{2}{*}{$O\lrp{\frac{1}{n^{4/5}}}$}  & \multirow{2}{*}{\xmark}  & \multirow{2}{*}{$O(n)$} \\
time series (STS) &  \cite{aktaran2011mean} & &  &  \\

\midrule
Recursive BM &\cite{wu2009recursive,yau2016new}  & $O\lrp{\frac{1}{n^{2/3}}}$ & \checkmark & $O(1)$ \\

\midrule
Jackknifed-(O)BM/STS & \cite{dingecc2015jackknifed} & $O\lrp{\frac{\log n}{n}}$ & \xmark & $O(n)$  \\

\midrule
Regenerative& \cite{crane1974simulating} & \multirow{2}{*}{$O\lrp{\frac{1}{n}}^{[d]}$} & \multirow{2}{*}{\checkmark} & \multirow{2}{*}{$O(1)$}  \\
simulation & \cite{glynn1987joint} &  & & \\

\midrule
\midrule
%\rowcolor{green!20} 
{ Linear SA} & { This work} &$ O\lrp{\frac{1}{n}}^{[e]}$  & \checkmark & $O(1)$ \\
\bottomrule

\end{tabular}
\begin{tablenotes}
    \noindent[a]~ Different works in literature estimate $\kappa(f)$ under different conditions on the underlying process. However, in this table, we present their rates and memory requirements for the setting of finite state DTMC.\newline
    \noindent[b]~ Estimator at step $n$ can be expressed as a function of various estimates at $n-1$, and can be computed in a constant time (independent of $n$). A non recursive estimator requires re-computing the estimator from scratch at each step. \newline
    \noindent[c]~  Memory required for computing the estimator at step $n$ (as a function of $n$). \newline
    \noindent[d]~ 
     Asymptotic rates of convergence have been studied. However, finite-sample guarantees are not known.\newline
    \noindent[e]~ Finite sample bound on MSE.
\end{tablenotes}
\end{small}
\end{center}
\end{table*}

%\end{comment}

\paragraph*{Batch Means (BM)} A popular estimator for $\kappa(f)$ is the BM estimator \cite{conway1959some,schmeiser1982batch}, which we informally describe next. Consider $n$ observations of the function $f$ evaluated on the $n$ consecutive states of the underlying Markov chain, and divide these observations into $b$ non-overlapping batches, each of size $m$ (for simplicity, assume that $n = mb$). For $i\in (b)$, let $m_i$ denote the sample mean of the observations from batch $i$. Then, the BM estimator corresponds to the sample variance of $m_i$ scaled by the batch size $m$. BM and its variants have been extensively studied (see, \cite{schmeiser1982batch, meketon1984overlapping,carlstein1986use, welch1987relationship,kunsch1989jackknife,alexopoulos2007overlapping,meketon2007overlapping,alexopoulos2011overlapping}). See \cite{politis1999subsampling} for an introduction to the different estimators and their relations. Roughly speaking, here, the batch size $m$ governs the bias, while the number of batches $b$ governs the variance of the estimator.

\cite{glynn1991estimating} establish limitations of BM in the settings where the number of batches $b$ does not increase with the total number of samples $n$. \cite{glynn1992asymptotic} establish conditions that ensure the asymptotic validity of BM. 
\cite{chien1997large} show that BM with both $b$ and $m$ increasing with $n$, is asymptotically unbiased (as $n$ increases), and converges to $\kappa(f)$ in mean square. In addition, they also show an asymptotic equality for MSE of BM, establishing it to be $O(\frac{1}{n^{2/3}})$.  

\paragraph*{Overlapping BM} While BM may look wasteful, overlapping BM \cite{meketon1984overlapping, meketon2007overlapping} uses $b$ overlapping batches, each of size $m$, where the consecutive batches use just $1$ additional observation, along with $m-1$ previous (consecutive) ones. Here, total number of samples satisfy $n=m+b-1$. When $n$ is known ahead of time, 
\cite{flegal2010batch} establish $O(\frac{1}{n^{2/3}})$ MSE for overlapping BM estimator. Notably, overlapping BM has the same MSE as BM (order wise) since the bias of the two estimators is the same, and the variance of the overlapping BM is $\frac{2}{3}$ times that of BM (improvement only by a constant factor). Thus, both BM and overlapping BM have a sub-optimal rate of convergence for MSE. See \cite{alexopoulos2011overlapping} for a survey of results on BM and overlapping BM estimator.  

In addition to being sub-optimal, another drawback of the classical and overlapping BM is that either their error guarantees are only satisfied at time $n$ and are not anytime, or they are not recursive (sequential), i.e., one needs to re-compute the estimator at each time using computation that is linear in the number of observations. This is because the batch sizes are predetermined, and depend on $n$. Moreover, one requires storing all the previous observations in order to re-compute the estimator at each step, hence, $O(n)$ computation at time $n$. \cite{wu2009recursive} develop a recursive BM  estimator ($O(1)$ computation at each time) with increasing batch sizes that are independent of $n$. Their estimator does not require storing the past observations, improving the memory requirement to $O(1)$, and is asymptotically unbiased. However, \cite{wu2009recursive} only give asymptotic guarantees, and establish a sub-optimal rate of $O(\frac{1}{n^{2/3}})$ for MSE. \cite{yau2016new} also propose a different recursive BM estimator with similar guarantees.  

%\cite{yau2016new} propose two other recursive estimators named, Trapezoidal Selection Rule (TSR) and the Parallelogram Selection Rule (PSR), based on the shapes of the batches, and study their asymptotic behaviors, including consistency and convergence rates. They show that PSR is uniformly better than \cite{wu2009recursive} at a cost of larger memory requirement of $O(n^\alpha)$, for some $\alpha\in (0,\frac{1}{2})$. In contrast, TSR preserves the memory requirement at $O(1)$ and has a smaller MSE compared to \cite{wu2009recursive} in some cases. Their estimator too does not require knowing $n$ in advance, and requires $O(1)$ computation at each step. However, it too has a sub-optimal rate of convergence for MSE of $O(\frac{1}{n^{2/3}})$. \shubhada{Compress this paragraph to 1 sentance. Remove details.}

\paragraph*{Other estimators} Several bootstrap-based estimators have also been proposed in literature  \cite{kunsch1989jackknife, liu1992moving,carlstein1986use,romano1992circular,politis1994stationary}. However, these are computationally intensive, and hence, we do not discuss them in details in this work. \cite{nordman2009note, lahiri1999theoretical} study the asymptotic rates for bias and variance of the different bootstrap-based estimators. %{\color{red}; results summarized in Table~\ref{table:results2}}.  
 We refer the reader to \cite{lahiri2013resampling} for an introduction to this approach. \cite{schruben1983confidence}  develop standardized time series based estimators, extensively studied in \cite{goldsman1984asymptotic,glynn1988new,goldsman1990properties,goldsman1990note,foley1999confidence, aktaran2011mean}. \cite{trevezas2009variance} propose maximum likelihood estimator for $\kappa(f)$, and prove that it is strongly consistent and asymptotically normal. 

\paragraph*{Jackknifed versions of estimators} Since a significant contribution to MSE of BM and its variants is due to bias of the estimator, \cite{dingecc2015jackknifed} propose using jackknife \cite{tukey1958bias} to reduce it at a cost of modest increase in variance of the estimator. They develop jackknifed versions of BM and overlapping BM estimators with a significantly reduced MSE of $O(\frac{\log n }{n})$. However, these are not recursive, and require $O(n)$ computation and $O(n)$ memory at time step $n$, rendering them less practical.  

\paragraph{Regenerative simulation} Another class of estimators that have been extensively studied at least since 1970s is the regeneration-based estimators. These depend on existence of regeneration points in the underlying process at which the process can be thought to probabilistically start afresh and evolve independent of the past. Note that in a finite state DTMC, the (random) times of successive visits of the Markov chain to a particular state can serve as the regeneration points. Hence, loosely speaking, a DTMC path can be thought of as several independent cycles pasted together. A number of past works have exploited this property of the underlying  process to design efficient recursive estimators for $\kappa(f)$ (see, \cite{crane1974simulating,glynn1987joint, glynn1993conditions,henderson2001regenerative, glynn2002necessary} and \cite{glynn2006simulation} for a nice exposition), which can be implemented in an online fashion using $O(1)$ computation. Moreover, central limit theorems (CLTs) are established for these estimators, which establish an optimal asymptotic error rate of $O(\frac{1}{\sqrt{n}})$ (that is comparable to 
$O(\frac{1}{n})$ MSE). However, to the best of our knowledge, finite time bounds of these estimators have not been established. While the CLT strongly suggests that one should be able to obtain a finite time bound of 
$O(\frac{1}{n})$ MSE, the explicit dependence of the error on the parameters of the underlying chain is not known. In contrast, we propose an alternate approach to Markov chain variance estimation based on SA, and we present explicit finite time bounds on the MSE.

%also be shown to have optimal $O(\frac{1}{n})$ MSE. Moreover, these estimators can also  We refer the reader to \cite{glynn2006simulation} for an overview of the basic ideas underlying the use of regenerative structure of a Markov chain in designing efficient estimators. Our work presents an alternate approach to Markov chain variance estimation, and we present explicit finite time bounds on the MSE.}

\paragraph*{Poisson equations} Finally, \cite{douc2022solving} resort to an alternative formulation for $\kappa(f)$ involving solution to a certain Poisson equation, propose a coupled Markov chains-based estimator for $\kappa(f)$, and provide its asymptotic convergence guarantees. The work that is closest to ours is \cite{BentonThesis} which proposed a SA-based estimator for $\kappa(f
)$, and established its asymptotic convergence. While their estimator is recursive, and requires $O(1)$ memory and computation at each step, they do not provide either the rates of convergence, nor finite-sample bounds on MSE of the estimator. In the current work, we propose a different SA-based estimator, along with finite-sample guarantees on its MSE. 
 
Note that the focus of this work is to estimate the asymptotic variance  $\kappa(f)$. To this end, as we will see later, we design a TD-like linear SA algorithm. This problem of estimating $\kappa(f)$ is very different from studying the variance of iterates of any particular  algorithm like TD, or designing algorithms with minimum variance of the iterates themselves. These latter objectives too are well-studied in literature; cf. \cite{devraj2017fastest,yin2020asymptotically,chen2020explicit,ChenSGD2020,hao2021bootstrapping,zhu2023online}. In contrast to these orthogonal prior works, which focus on studying the variance of TD learning and its variants, we do \textit{not} study the variance of a linear SA algorithm itself. We use linear SA simply as a tool for solving the variance estimation problem. To further elaborate on this distinction, for concreteness, consider the setting of Markov Decision Processes (MDPs), where the transition probabilities are known. Here, for evaluating a given policy, one can use value iteration (VI). Since the MDP is known, there is no variance associated with the iterates of VI. However, the asymptotic variance that we consider in this work is still well-defined, and exists due to inherent randomness associated with the probabilistic transitions of the MDP. This work focuses on estimating this asymptotic variance.

%\paragraph{Estimators for Iterates of SGD. } \shubhada{Include to highlight that we are using SA for Markov chain variance estimation. Don't explain these estimators. Include ICML rebuttal statements. Dump all reference around that in this paragraph.}  \cite{ChenSGD2020} estimate asymptotic variance of iterates of SGD with decaying step sizes -- these can be viewed as time non-homogeneous Markov chain. They propose a sequential BMs estimator in a batched setting. In particular, their batch sizes are also increasing, but depend on $n$,    the total number of samples. Their MSE is $O(\frac{1}{n^{1/2    }})$. \cite{zhu2023online} also estimate asymptotic variance of iterates of SGD with decaying step sizes, improving over \cite{ChenSGD2020}. They propose a fully-sequential batch means estimator that is similar to that in \cite{wu2009recursive}, but for non-homogeneous setting. They say that the batch size construction is different from that in \cite{wu2009recursive}, and that even proving convergence of the estimator is far from being trivial. Further the rate of MSE is $O(\frac{1}{n^{2/3}})$ -- they improve over \cite{ChenSGD2020}. \shubhada{Move to RL section.}

Later, in Section~\ref{sec:RL}, we present an application of the proposed estimator in the average reward RL setting. We postpone the related literature from RL to Section~\ref{sec:litRL}. 

\subsection{Setup and Asymptotic Variance}\label{sec:var}
In this subsection, we formally introduce the problem. To this end, we first define some notation. Let $\N$ denote the set of Natural numbers, $\R$ the set of Real numbers, and let $d\ge 1$. For a column vector $v\in \R^d$, let $v^T$ denote its transpose and let $\|v\|_2$ denote the $\ell_2$-norm of $v$. Let ${\bf 1}$ (or ${\bf e}$) and ${\bf 0}$ denote the all ones and all zeros vectors in appropriate dimensions, respectively.

We now describe the underlying dynamics, followed by formally defining the asymptotic variance. For $S\in\mathbb{N}$, consider an irreducible and aperiodic DTMC, $\calM$, on a finite state space $\calS = \lrset{1,\dots, S}$ with transition matrix $P$ and a unique stationary distribution, denoted by $\pi$ \cite[Chapter 1]{levin2017markov}. For $n\in\N$, let $X_n\in\calS$ denote the  state at time $n$. Consider $f:\calS\rightarrow \R$, a function defined on state space of $\calM$, and let $\bar{f}$ denote its stationary expectation, i.e., $\bar{f}:=\Exp{\pi}{f(X)}$. Since $f(X) - \bar{f}$ has stationary mean zero, there exists $V: \calS \rightarrow\R$ that solves the following Poisson equation \cite[Section 21.2]{douc2018markov}:
\[ f - \bar{f}{\bf 1} = V - PV. \numberthis\label{eq:PE} \]
Since $V$ solves~\eqref{eq:PE}, observe that $V + c{\bf 1}$ for any constant $c\in\R$ is also a solution. Let $V^*$ denote the solution normalized so that it is orthogonal to ${\bf 1}$ in $\ell_2$ norm, i.e., ${\bf 1}^TV^{*} = 0$. It is well-known that the set of solutions of~\eqref{eq:PE} takes the form $S_f:= \lrset{V^* + c{\bf 1}| c\in \R}$ \cite[Section 21.2]{douc2018markov}.

\iffalse
{\color{gray} Next, consider the partial sum of $f$
\[ S_n(f) := \sum\limits_{k=0}^{n-1}f(X_k). \]
It is well known that in this setting, Central Limit Theorem (CLT) holds \cite[Chapter 21]{douc2018markov}, and implies that for a positive constant $\kappa(f)$,
\[\frac{S_n(f)}{\sqrt{n}} \Rightarrow N(0, \kappa(f)) \quad \text{ as } \quad n\rightarrow\infty,  \]
where the above convergence is in distributions, and $N(0,\kappa(f))$ denotes the Gaussian distribution with mean $0$ and variance $\kappa(f)$. We refer the reader to \cite{billingsley2013convergence} for an exposition on convergence in distribution.} 
\fi

For a Markov chain starting in state $x_0$, the asymptotic variance is given by
\begin{equation}\label{eq:var_def}
    \lim\limits_{n\rightarrow\infty} ~ \V\lrs{\frac{1}{\sqrt{n}}\sum\limits_{k=0}^{n-1}f(X_k) \bigg| X_0 = x_0},
\end{equation}
and is independent of the starting state (Proposition~\ref{prop:var} below). This follows from geometric mixing of $\mathcal M$.  We denote this constant by $\kappa(f)$ (or $\kappa$, since $f$ is fixed). In this paper, we assume that $|f|$ is bounded. Henceforth, without loss of generality, this bound is assumed to be $1$. Below, we first present equivalent formulations for $\kappa(f)$ that will be useful in designing efficient algorithms for estimating it. 

\begin{proposition}\label{prop:var}
        Given a irreducible and aperiodic DTMC, asymptotic variance defined in~\eqref{eq:var_def} is a constant independent of the starting state, and is given by 
        \[  \kappa(f) = \lim\limits_{n\rightarrow\infty} ~ \V\lrs{ \frac{1}{\sqrt{n}}\sum\limits_{k=0}^{n-1} f(X_k) \bigg| X_0 \sim \pi }. \numberthis\label{eq:kmu0}\]
Furthermore, given any solution $V \in S_f$ to~\eqref{eq:PE}, 
    \begin{align*}
        \kappa(f) 
        & =~~ \Exp{\pi}{(f(X) - \bar{f})^2 } \numberthis\label{eq:kmu1} + 2  \lim\limits_{n\rightarrow\infty}\sum\limits_{j=1}^{n-1}\Exp{\pi}{ \lrp{f(X_0)- \bar{f}}\lrp{ f(X_j) - \bar{f} } }  \\
        &= ~~2 \Exp{\pi}{(f(X)-\bar{f})V(X)} - \Exp{\pi}{ (f(X) - \bar{f})^2} \numberthis\label{eq:kmu2}\\
        &= ~~ \Exp{\pi}{V^2(X)}  - \Exp{\pi}{ (P V)^2(X)}.\numberthis\label{eq:kmu3}
    \end{align*}
\end{proposition}

It is worth noting that the first term in~\eqref{eq:kmu1} accounts for the per-step stationary variance, while the second term encompasses the temporal correlation introduced due to the Markov chain structure. These different representations for asymptotic variance are well known in literature at least since 1970s \citep{maigret1978theoreme}. Also see \cite[Theorem 17.5.3]{meyn2012markov}, \cite[Theorem 21.2.6]{douc2018markov}, and more recently, \cite{glynn2024solution}, for similar formulations. For completeness, we give a proof of the Proposition in  Appendix~\ref{app:proof_prop:var}.

Continuing, the second representation in equation \eqref{eq:kmu2} formulates the variance of $f(\cdot)$ in terms of the corresponding solution of the Poisson equation~\eqref{eq:PE}. Note that this representation isn't affected by choice of $V$ (any constant shift in $V$ doesn't affect $\kappa$). We will use a modification of the form in~\eqref{eq:kmu2} to design a linear SA algorithm to estimate $\kappa$ for a given function $f$.  

Finally, the third equality follows by using the Poisson equation to replace $(f(X) - \bar{f})$ terms in~\eqref{eq:kmu2} by $V(X) - PV(X)$.  

\subsection{Stochastic Approximation: An Overview}
Before proceeding to design an SA-based estimator for $\kappa(f)$, in this subsection, we give a brief overview of the SA updates we will repeatedly use throughout the rest of the paper. The earliest SA algorithms, which are iterative updates that can be used to solve  for roots of certain functions, date back at least to \cite{robbins1951stochastic}. While the asymptotic convergence of SA algorithms has long been known \cite{borkar2008stochastic}, finite sample guarantees for linear SA algorithm have only recently been established \cite{srikant2019finite, zhang2021finite}. In this work, we will use SA for finding averages as well as roots of linear functions, which we discuss next.

\paragraph*{Mean Estimation with Step Sizes} An SA algorithm to estimate $\Gamma := \Exp{\pi}{ g(X) }$ for a function $g$ using samples $g(X_0), \dots, g(X_n)$, corresponds to the following recursive update with an appropriately chosen step size $\alpha_n > 0$, and an initial iterate $\Gamma_0$:
\[ \Gamma_{n+1} = (1-\alpha_n)\Gamma_{n} + \alpha_n g(X_n). \numberthis\label{eq:SAForFancyAvg}  \]
Observe that for $\Gamma_0 = 0$ and $\alpha_n = \frac{1}{n+1}$, this reduces to the classical empirical average estimator for mean. Moreover, if $X_0, \dots, X_n$ are i.i.d. according to $\pi$, then empirical mean is known to converge almost surely to $\Gamma$ (from Law of Large Numbers) under mild regularity conditions. Similar convergence results hold for averages in~\eqref{eq:SAForFancyAvg} under certain regularity conditions on $\alpha_n$ and data-generating distributions. 

\paragraph*{Linear SA} For $m> 0$ and $d > 0$, let $\A \in \R^{m\times d}$ and $\vb \in \R^d$.  We want to estimate a solution $\Gamma \in \R^d$ to the equation $\A \Gamma + \vb = 0$ using noisy observations $\A_1, \dots, \A_n$ and $\vb_1, \dots, \vb_n$. An SA algorithm for this corresponds to the following update with a well chosen step size $\alpha_n > 0$, as well as an initial iterate $\Gamma_0$:
\[ \Gamma_{n+1} = \Gamma_n + \alpha_n (\A_n \Gamma_n + \vb_n). \numberthis\label{eq:SAForLinear} \]

As we will see in the following sections, these update methods play a crucial role in arriving at a simple recursive estimator that satisfies all the desirable properties discussed earlier in Section~\ref{sec:intro}.

\section{Variance Estimation: Tabular Setting}\label{sec:varest:tabular}
With the expressions for asymptotic variance derived in Proposition~\ref{prop:var} above, in this section, we design an algorithm for estimating it. A natural first approach is to use the Monte Carlo technique for estimating $\kappa(f)$, for example, the Batch Means estimator and its variants, which proceed by appropriately truncating the infinite summation in~\eqref{eq:kmu0}. As discussed earlier in Section~\ref{sec:lit}, these approaches suffer from either being non-recursive, computationally demanding, having sub-optimal rates of convergence, or a combination of these. We refer the reader to Section~\ref{sec:lit} and Appendix~\ref{app:BM} for a discussion. 

%In this work, we design an asymptotically-unbiased SA-based algorithm that is fully recursive (adaptive) and is easy to compute.  

%Note that \cite{zhu2023online} design an adaptive version of BM estimator for variance of iterates of a specific online algorithm (stochastic gradient descent). However, their estimator exploits the specific structure of the Markov chain induced by the chosen algorithm, and cannot be directly used in our setting. Additionally, \cite{wu2009recursive} design a recursive BM estimator. However, its MSE is large, and hence, sub-optimal. 

\subsection{The Algorithm}
Since the equivalent formulation for $\kappa$ in~\eqref{eq:kmu1} involves an infinite summation, algorithms approximating it using a finite summation could possibly suffer from being non-adaptive as the vanilla Monte Carlo approaches. Instead, we design a fully-adaptive TD-type algorithm for estimating $\kappa$, using SA for a modification of the expression in~\eqref{eq:kmu2} (Equation~\eqref{lem:var_form_twotraj} below). This alternative representation helps us to design a linear SA algorithm with finite-sample guarantees on the mean (squared) estimation error. For a solution $V$ of~\eqref{eq:PE}, let $\bar{V} = \Exp{\pi}{V(X)}$. Then,  
\begin{equation}\label{lem:var_form_twotraj}
    \kappa(f) = \Exp{\pi}{2f(X)V(X) - 2f(X)\bar{V} - f^2(X) + f(X)\bar{f}}.
\end{equation}
This immediately follows from rearranging terms in~\eqref{eq:kmu2}. 

\begin{algorithm2e}[tb]
\RestyleAlgo{ruled}
\DontPrintSemicolon
\SetNoFillComment
\SetKwInOut{Input}{Input}\SetKwInOut{Output}{Output}
%\setstretch{1.05}
\caption{Estimating $\kappa(f)$: Tabular Setting}\label{alg:PE}
    %\vspace{0.4em}
    \BlankLine

    \KwIn{Time horizon $n > 0$, constants $c_1 > 0$, $c_2 > 0$, $c_3 > 0$ and step-size sequence $\lrset{\alpha_k}$.}
    \BlankLine

    {\bf Initialization: } $\bar{f}_0  = 0$, $V_0 = \vec{\bf  0} $, $\bar{V}_0 = 0$, and $\kappa_0 =0$.\\
    \BlankLine

\For{$k\leftarrow 1$ \KwTo $n$}{
\BlankLine

Observe ($X_k, f(X_k), X_{k+1}$).\;

\BlankLine

\tcp{\textbf{Average reward $\bar{f}$ estimation}} 
        \vspace{-1.5em}
        \begin{align}
            \bar{f}_{k+1} = \bar{f}_k + c_1\alpha_k (  f(X_k) -\bar{f}_k ).\label{eq:JmuUpdate}
        \end{align}
\;
        \vspace{-1.5em}

\tcp{\textbf{$V$ estimation}}
\vspace{-1.5em}
    \begin{align*}
    & \delta_k := f(X_k) - \bar{f}_{k} +  V_k(X_{k+1}) - V_k(X_k), \label{eq:deltaUpdate}\numberthis\\
    & V_{k+1}(x) = \numberthis\label{eq:QUpdate}
        \begin{cases}
            V_k(x) + \alpha_k \delta_k \lrp{1-{1}/{S}},  &\text{ for } x = X_k,\\
            V_k(x) -  {\alpha_k\delta_k}/{S}, &\text{ otherwise}.
        \end{cases}  
    \end{align*}
\;
\vspace{-1.5em}
\tcp{\textbf{$\bar{V}$ estimation}}
\vspace{-1.5em}
    \begin{align*}
    & \bar{V}_{k+1} := \bar{V}_k + c_2 \alpha_k ( V_k(X_k) - \bar{V}_k ), \label{eq:barVUpdate}\numberthis\\
    \end{align*}
\;
\vspace{-1.5em}

\tcp{\textbf{Asymptotic Variance $\kappa(f)$ estimation}}
\vspace{-1.5em}
        \begin{align}
        &\kappa_{k+1} =  (1-c_3\alpha_k) \kappa_k + c_3\alpha_k\lrp{ 2f(X_k)V_k(X_k) - 2f(X_k)\bar{V}_k - f^2(X_k) + f(X_k)\bar{f}_k }. \label{eq:kappaUpdate}
        \end{align}
        \vspace{-1.5em}
    }
 \Output{Estimates $\kappa_n$, $\bar{V}_n$, $V_n$, and $\bar{f}_n$} 
    
\end{algorithm2e}

\paragraph*{Estimating $\kappa(f)$}  For $n\ge 1$, and time steps $k = 1,2, \dots, n$, let $X_k$ represent the state of the underlying Markov chain. To estimate $\kappa(f)$, our algorithm uses SA for the  expression in~\eqref{lem:var_form_twotraj}, which corresponds to the mean estimation with step sizes, discussed in~\eqref{eq:SAForFancyAvg}. To get observations of the function being averaged, it first estimates $\bar{f}$, $V$, and the corresponding $\bar{V}$. Denote these estimates at time $k$ by $\bar{f}_k$, $V_k$, and $\bar{V}_k$. Equation~\eqref{eq:kappaUpdate} in Algorithm~\ref{alg:PE} corresponds to SA update for estimating $\kappa(f)$ using these estimates in~\eqref{lem:var_form_twotraj}.  

\paragraph*{Estimating $\bar{f}$} This corresponds to estimating $\Exp{\pi}{f(X)}$ using the latest estimate $\bar{f}_{k}$, as well as the new observed value $f(X_{k})$. We again use~\eqref{eq:SAForFancyAvg} for this step, which corresponds to update~\eqref{eq:JmuUpdate} in Algorithm~\ref{alg:PE}. 

\paragraph*{Estimating $V$} Recall that $V$ is a solution for~\eqref{eq:PE}, which is a linear equation in $V$. Hence, we will use the SA algorithm in~\eqref{eq:SAForLinear} to estimate it. Term $\delta_k$ in~\eqref{eq:deltaUpdate} in Algorithm~\ref{alg:PE} represents the SA adjustment term for estimating $V$ using the estimates $\bar{f}_k$, $V_k$ computed at the previous step and the observed function value $f(X_k)$. Notice that since~\eqref{eq:PE} may not have a unique solution, we particularly consider the one that is orthogonal to the all ones vector ${\bf 1} \in \R^S$ in $\ell_2$-norm. Thus, in each iteration of Algorithm~\ref{alg:PE}, we project $\delta_k$ to the subspace orthogonal to ${\bf 1}$. This corresponds to the adjustments made in the $V$ update in~\eqref{eq:QUpdate}, where $S$ denotes the size of state space $\abs{\calS}$ of the Markov chain $\calM$. 

\paragraph*{Estimating $\bar{V}$} This corresponds to the mean estimation update in~\eqref{eq:barVUpdate} using the estimates $\bar{V}_{k}$ and $V_{k}$ computed at the previous step (using~\eqref{eq:SAForFancyAvg}).

\begin{remark}[Justification for projection in Algorithm~\ref{alg:PE}]\label{rem:proj} Observe that $\kappa(f)$ is independent of the choice of $V$ (choice of solution of~\eqref{eq:PE} does not affect $\kappa$), since $\bar{V}$ in~\eqref{lem:var_form_twotraj} corresponds to the stationary expectation of the particular solution $V$ used in that formulation. Moreover, estimating a solution $V$ of~\eqref{eq:PE} does not require the projection step, since the SA for~\eqref{eq:PE} converges to \emph{some} solution; cf. \cite[Algorithm 1]{zhang2021finite}. To further elaborate, if in each iteration $k$ of Algorithm~\ref{alg:PE}, $V_k$ corresponds to an estimate for a different solution of~\eqref{eq:PE}, then as $k\rightarrow \infty$, $\{V_k\}$ can be still be shown to converge to \emph{some} solution of~\eqref{eq:PE}, say $V^{\dagger}$. However, to estimate $\kappa(f)$ using~\eqref{lem:var_form_twotraj}, we further need to estimate the stationary mean of $V^{\dagger}$, say $\bar{V}^{\dagger}$. However, since without projection, the estimates $V_k$ in Algorithm~\ref{alg:PE} may correspond to different solutions of~\eqref{eq:PE}, and do not necessarily all estimate $V^{\dagger}$, the iterates $\bar{V}_k$ are not guaranteed to converge to $\bar{V}^{\dagger}$. Hence, we introduce the projection step to ensure that for each $k$, $V_k$ correspond to the \emph{same estimate} for a solution of~\eqref{eq:PE} (which is the one orthogonal to $\bf 1$ in $\ell_2$ norm), and hence, $\bar{V}_k$ converge to \emph{its stationary expectation}. 
\end{remark}

\begin{remark} 
Notice that all the updates in Algorithm~\ref{alg:PE} in~\eqref{eq:JmuUpdate},~\eqref{eq:QUpdate},~\eqref{eq:barVUpdate}, and~\eqref{eq:kappaUpdate} at time $k$ use estimates from the previous time step $k-1$. Hence, the order in which these updates are performed is irrelevant. 
\end{remark}

\subsection{Convergence Rates}\label{sec:theorype}
We now bound the estimation error of the proposed algorithm. For each $k\in\N$, let $Y_k:= (X_k,X_{k+1}).$ Observe that $\mathcal M_2 := \lrset{Y_k}_{k\ge 1}$ is a Markov chain. Let its state space be denoted by $\calY$, which is a finite set. As for $\mathcal M$, $\mathcal M_2$ mixes geometrically fast \cite[Proposition 1]{bhatnagar2016multiscale} and has a unique stationary  distribution. 

Next, let $\mathcal F \in \R^{S}$ be the vector of function values, i.e., for $i\in\calS$, the $i^{th}$ coordinate of $\mathcal F$ equals $f(i)$, and $f_{\max} := \|\mathcal F\|_\infty$. Recall, without loss of generality, we assume that $f\in [-1,1]$, and hence, $f_{\max} \le 1$. Moreover,  $V^*$ was defined after~\eqref{eq:PE}, and let $D_\pi$ denote the diagonal matrix of the stationary distribution $\pi$ of $\calM$, i.e., $D_\pi(i,i) = \pi(i)$, for all $i\in\calS$. Define 
\[ \Delta_1:= \min\!\lrset{\!{\bf v}^T \! D_\pi(I\!-\!P) {\bf v} ~|~ {\bf v} \!\in\! \R^{S}, \|{\bf v}\|_2 = 1, {\bf v}^T\!{\bf 1} = 0}. \]
Clearly, feasible vectors $\bf v$ in $\Delta_1$ are non-constant vectors. Then, from \cite[Lemma 7]{tsitsiklis1997average}, we have that $ {\bf v}^T  D_\pi(I-P) {\bf v}  > 0.$ 
Since the feasible region in $\Delta_1$ is non-empty and compact, we get $\Delta_1 > 0$. 

Next, for the step-size constants $c_1$, $c_2$, and $c_3$ (inputs to Algorithm~\ref{alg:PE}), define
\begin{align}
\eta  :=  \lrp{c^2_1  + 5 + 2c^2_2 + 10c^2_3}^\frac{1}{2},\label{eq:eta} 
\end{align}
and let $\Theta_\pi := [\bar{f}~V^*~\bar{V}^*~\kappa(f)]$. The following theorem bounds the estimation error of Algorithm~\ref{alg:PE}, which immediately implies ${O}(\frac{1}{n})$ convergence rate for the mean-squared estimation error of $\kappa(f)$ to $0$. 

\begin{theorem}\label{th:finitebound}
Consider Algorithm~\ref{alg:PE} with $c_1$, $c_2$, and $c_3$ satisfying:
    \[c_1\ge \frac{1}{2\Delta_1} + \frac{\Delta_1}{2}, \qquad  \frac{5}{249}(5-2\sqrt{2})\Delta_1 \le c_3\le \frac{5}{249} (5+2\sqrt{2})\Delta_1, \]
    and
    \[ c_3 -\frac{498 c^2_3}{7\Delta_1} + \frac{7 \Delta_1}{498}   \le c_2 \le -3 c_3 + \frac{5 }{83} \sqrt{498 c_3 \Delta_1 - 17 \Delta^2_1} . \]
    Recall that $\eta$ and $\Theta_\pi$ were defined around~\eqref{eq:eta}. For a constant $B > 1$, let  $\xi_1 :=  3\lrp{1+\|\Theta_\pi\|_2}^2$, $\xi_2 := 112B(1+\|\Theta_\pi\|_2)^2$, and let $\Theta_n := [ \bar{f}_n~ V^T_n~\bar{V}_n~\kappa_n]$ denote the vector of estimates at step $n$.
    \begin{enumerate}%[label=(\alph*)]
        \item[(a)] Let $\alpha_i = \alpha$ for all $i$, such that     
        \[ \alpha < \min\lrset{ \frac{40}{\Delta_1}, \frac{1}{28B\lrp{1+ \frac{20\eta^2}{\Delta_1} }}}. \]
        Then, for all $n\ge 1$, 
        \[ \E{\|\Theta_n - \Theta_\pi\|^2_2} \le \xi_1\lrp{1-\frac{\Delta_1\alpha}{40}}^{n} +  \frac{20\xi_2 \alpha\eta^2}{\Delta_1} + \xi_2\alpha. \]

        \item[(b)]  Let $\alpha_i = \frac{\alpha}{i+h}$ for all $i \ge 0$, with $\alpha$ and $h$ chosen so that
        \[ h \ge \max\lrset{1+\frac{\alpha\Delta_1}{40}, 1+28B\lrp{ \frac{20\eta^2\alpha}{\Delta_1} + \alpha + \frac{20}{\Delta_1} }},\quad\text{and}\quad \alpha > \frac{40}{\Delta_1}.\]
        Then, for all $n\ge 0$, 
        \begin{align*}
        \E{\|\Theta_n- \Theta_\pi\|^2_2}  &\le \xi_1\lrp{\frac{h}{n+h}}^\frac{\alpha\Delta_1}{40} + \frac{5\xi_2e^2\eta^2(20+\Delta_1)\alpha^2}{(n+h)(\alpha\Delta_1 -40)} + \frac{\xi_2 \alpha}{n+h}.
        \end{align*} 
    \end{enumerate}
\end{theorem}

The above theorem follows from a more general result (Theorem~\ref{th:finiteboundlfa}) presented in the next section. We refer the reader to Appendix~\ref{app:proof_th:finitebound} for a formal reduction from the setting of Theorem~\ref{th:finiteboundlfa} to that in the above theorem. We now present some remarks about the bound in Theorem~\ref{th:finitebound} above. 

First, observe that since $\E{(\kappa_n - \kappa(f))^2} \le \E{ \|\Theta_n - \Theta_\pi\|^2_2 }$ and $\E{\|V_n - V^*\|^2_2} \le \E{\|\Theta_n - \Theta_\pi\|^2_2}$, the bound in Theorem~\ref{th:finitebound} also holds for MSE in estimation of $\kappa(f)$ as well as $V^*$.  

Theorem~\ref{th:finitebound}$(a)$ bounds the mean-squared estimation error of the proposed algorithm in case of constant step size. Though the first term in this bound decays exponentially fast as the number of iterations $n$ increase, the other two terms are constant and become a bottleneck after sufficiently large $n$. This is a well-known behavior of SA algorithms with constant step sizes. In fact, this algorithm suffers similar drawback of being non-recursive as the vanilla Monte Carlo BM estimator. In order to ensure the mean-squared error smaller than $\epsilon^2$ for $\epsilon> 0$ (or mean absolute error smaller than $\epsilon$), one needs to pick the step size $\alpha$ as a function of $\epsilon$ (also see Corollary~\ref{cor:samplecomplex} below). This drawback of being non-recursive is overcome by choosing the step size to be diminishing (as in Theorem~\ref{th:finitebound}$(b)$). 

Second, it has been shown recently that linear SA with constant step size, in presence of Markov noise, suffers from an asymptotic bias (as $n\rightarrow\infty$); see \cite{huo2023bias,nagaraj2020least}. This is in contrast to the i.i.d. noise setting, where the asymptotic bias is shown to be $0$ in \cite{mou2020linear,lakshminarayanan2018linear}. This motivates us to study SA with diminishing step-size, considered in Theorem~\ref{th:finitebound}$(b)$.

\begin{remark}
For diminishing step size, Algorithm~\ref{alg:PE} achieves ${O}(\frac{1}{n})$ rate of convergence (Theorem~\ref{th:finitebound}$(b)$). As in part $(a)$, the first term in the bound in part $(b)$ decays faster, and the other two terms are rate-determining $O(\frac{1}{n})$ terms. We believe that the rate of convergence of $\Omega(\frac{1}{n})$ for MSE is tight. In the special setting of i.i.d. noise with parametric distributions, this follows from the Cram\'er-Rao lower bound \cite{nielsen2013cramer}. We  discuss a simple estimator that achieves this rate in i.i.d. setting in Appendix~\ref{app:iidestimation}, and also propose a SA-based recursive estimator achieving the same rate in i.i.d. setting in the following section (Section~\ref{sec:statvar}). 
\end{remark}

Using bounds in Theorem~\ref{th:finitebound}, the corollary below presents the number of iterations of Algorithm~\ref{alg:PE} needed to have the mean estimation error bounded by $\epsilon$. It follows from setting the mean-squared estimation error bounds in Theorem~\ref{th:finitebound} to at most $\epsilon^2$.  %Observe that $c_1 = O(1/\Delta_1)$, $\eta = O(1/\Delta_1)$, $\xi_1 = O(\|Q^*_\mu\|^2_2)$, and $\xi_2 = O(\|Q^*_\mu\|^2_2/\Delta^2_1)$. 
We refer the reader to Section~\ref{app:samplecomplex} for a proof of Corollary~\ref{cor:samplecomplex}.

\begin{corollary}\label{cor:samplecomplex}
    To estimate $\kappa(f)$ using the iterates $\kappa_n$ generated by Algorithm~\ref{alg:PE} with diminishing step size, up to mean estimation error  $\E{\abs{\kappa_n - \kappa(f)}} \le \epsilon$, we require
    \[ n = O\lrp{\frac{1}{\epsilon^2}} O\lrp{  \frac{\|V^*\|^2_2}{\Delta^4_1}}. \]
\end{corollary}

The sample complexity in Corollary~\ref{cor:samplecomplex} depends on $|S|$ via $\Delta_1$, and $\|V^*\|_2$. Here, 
$V^* \in \R^{|S|}$, and hence, $\|V^*\|^2_2 = O(|S|^2)$. The dependence of $\Delta_1$ on  $|S|$ is more implicit, and relates to the mixing properties of the Markov chain induced by the policy. We make this dependence explicit in two specific examples. If the underlying Markov chain is a random walk on a complete graph (with $|S|$ vertices), then $\Delta_1 = O(1)$, and hence we get ${O}(|S|^2)$ dependence. On the other hand, if the underlying Markov chain is a random walk on a cycle graph \cite[Section 12.3.1]{levin2017markov}, then $\Delta_1 = O({1}/{|S|^2})$, and hence, we get a dependence of ${O}(|S|^{10})$ in the sample complexity.

\begin{remark}
Using Jensen's inequality, we have $\E{\abs{\kappa_n - \kappa(f)}} \le \sqrt{ \E{ (\kappa_n - \kappa(f))^2 } }, $ 
which then gives a bound on the mean estimation error. In particular, for the diminishing step-sizes of the form $\alpha_i = \frac{\alpha}{i+h}$ for all $i\ge 1$, Theorem~\ref{th:finitebound}(b) gives
$$\E{\abs{\kappa_n - \kappa(f)}} \le {O}\lrp{\frac{1}{\sqrt{n}}}.$$
\end{remark}

\begin{remark}
Consider the problem of estimating the standard deviation, $\sqrt{\kappa(f)}$. Mean-squared error for the estimator $\sqrt{\kappa_n}$ satisfies
$$ \E{ \lrp{\sqrt{\kappa_n} - \sqrt{\kappa(f)}}^2 } = \E{ \frac{({\kappa_n} - {\kappa(f)})^2}{\lrp{\sqrt{\kappa_n} + \sqrt{\kappa(f)}}^2} } \le \frac{1}{\kappa(f)}\E{ \lrp{{\kappa_n} - {\kappa(f)}}^2 },$$
which is at most ${O}(\frac{1}{n})$ for $\alpha_k = O(\frac{1}{n})$. \end{remark}

\subsection{Estimating Stationary Variance}\label{sec:statvar}

In the previous section, we saw an estimator for the asymptotic variance $\kappa(f)$, which captures the temporal correlations introduced due to the underlying Markov chain. Suppose you want to estimate the stationary variance, instead. Formally, this corresponds to estimating  $v(f):=  \Exp{\pi}{(f(X) - \bar{f})^2}.$ In Algorithm~\ref{alg:StationaryVariance}, we present a sequential (recursive) estimator for $v(f)$, that  also satisfies all the properties discussed in Section~\ref{sec:intro}. Our estimator  corresponds to SA update for the following equation: 
\[ v(f) = \Exp{\pi}{f^2(X) - f(X) \bar{f}}.\]

\begin{algorithm2e}[tb]
\RestyleAlgo{ruled}
\DontPrintSemicolon
\SetNoFillComment
\SetKwInOut{Input}{Input}\SetKwInOut{Output}{Output}
%\setstretch{1.05}
\caption{Estimating Stationary Variance $v(f)$}\label{alg:StationaryVariance}
    %\vspace{0.4em}
    \BlankLine

    \KwIn{Time horizon $n > 0$, constant $c > 0$, and step-size sequence $\lrset{\alpha_k}$.}
    \BlankLine

    {\bf Initialization: } $\bar{f}_0  = 0$, and $v_0 =0$.\\
    \BlankLine

\For{$k\leftarrow 1$ \KwTo $n$}{
\BlankLine

Observe $(X_k, f(X_k))$.\;

\BlankLine

\tcp{\textbf{Average reward $\bar{f}$ estimation}} 
        \vspace{-1.5em}
        \begin{align*}
            \bar{f}_{k+1} = (1-\alpha_k)\bar{f}_k + \alpha_k f(X_k).
        \end{align*}
\;
        \vspace{-1.5em}

\tcp{\textbf{Stationary Variance $v(f)$ estimation}}
\vspace{-1.5em}
    \begin{align*}
        v_{k+1} = (1-c\alpha_k) v_k + c\alpha_k (f^2(X_k) - f(X_k) \bar{f}_k)
    \end{align*}
\;
\vspace{-1.5em}
    }
 \Output{Estimates $v_n$, $\bar{f}_n$}     
\end{algorithm2e}

\begin{theorem}\label{th:finitebound_statvar}
For $\gamma > 1$, consider Algorithm~\ref{alg:StationaryVariance} with $c$ satisfying:
    \[ c  \bar{f}^2 \le 2 \lrp{  -(\gamma-1) + \sqrt{(\gamma-1)\lrp{\gamma-1 + \gamma {\bar{f}}^2}} }. \]
    Let $(\Theta^S)^T := [ \bar{f}~~v(f)]$, and let $\eta = \sqrt{1 + c^2\lrp{1+\bar{f}^2}}$. For a constant $B > 1$, let  $\xi_1 :=  3\lrp{1+\|\Theta^S\|_2}^2$, $\xi_2 := 112B(1+\|\Theta^S\|_2)^2$, and let $\Theta^S_n := [ \bar{f}_n~ v_n]$ denote the vector of estimates at step $n$.
    \begin{enumerate}%[label=(\alph*)]
        \item[(a)] Let $\alpha_i = \alpha$ for all $i$, such that     
        \[ \alpha < \min\lrset{ \frac{2}{\gamma}, \frac{1}{28B\lrp{1+ \frac{\eta^2}{\gamma} }}}. \]
        Then, for all $n\ge 1$, 
        \[ \E{\|\Theta^S_n - \Theta^S\|^2_2} \le \xi_1\lrp{1-\frac{\gamma\alpha}{2}}^{n} +  \frac{\xi_2 \alpha\eta^2}{\gamma} + \xi_2\alpha. \]

        \item[(b)]  Let $\alpha_i = \frac{\alpha}{i+h}$ for all $i \ge 1$, with $\alpha$ and $h$ chosen so that
        \[ h \ge \max\lrset{1+\frac{\alpha\gamma}{2}, 1+28B\lrp{ \frac{\eta^2\alpha}{\gamma} + \alpha + \frac{1}{\gamma} }},\quad\text{and}\quad \alpha > \frac{2}{\gamma}.\]
        Then, for all $n\ge 0$, 
        \begin{align*}
        \E{\|\Theta^S_n- \Theta^S \|^2_2}  &\le \xi_1\lrp{\frac{h}{n+h}}^\frac{\alpha\gamma}{2} + \frac{5\xi_2e^2\eta^2(1+\gamma)\alpha^2}{(n+h)(\alpha\gamma
        -2)} + \frac{\xi_2 \alpha}{n+h}.
        \end{align*} 
    \end{enumerate}
\end{theorem}
We refer the reader to Appendix~\ref{app:stationary_var} for a complete proof of the Theorem above. It follows a 2-step approach. In the first step, we view Algorithm~\ref{alg:StationaryVariance} as a linear SA update, and establish a contraction property for the average of the update matrices (Lemma~\ref{lem:negdef_statvar} in Appendix~\ref{app:step1_statvar}). In the second step, we use a general bound on MSE for iterates of a linear SA satisfying the proven contraction property, that we prove in Theorem~\ref{th:appfiniteboundlfa}. In fact, we use a similar 2-step approach to prove other results in this work. We describe this approach in details in Section~\ref{sec:proofsketch}, where we prove a MSE bound for estimating $\kappa(f)$ in a more general setting, presented in the following section.

\begin{remark}
    In Step 2 of the proof of Theorem~\ref{th:finitebound_statvar}, we could use an existing finite sample bound for MSE of a linear SA update in presence of Markov noise \cite{srikant2019finite}. This would result in $O(\frac{\log n}{n})$ MSE bound in Theorem~\ref{th:finitebound_statvar}$~(b)$.  However our tighter bound in Theorem~\ref{th:appfiniteboundlfa}, that uses Poisson equation for handling the bias introduced due to presence of Markov dependence, results in $O(\frac{1}{n})$ MSE in Theorem~\ref{th:finitebound_statvar}$~(b)$ above.
\end{remark}

\begin{remark} 
Notice that there is no dependence on the size of the state space $S$  in the bound of Theorem~\ref{th:finitebound_statvar}.
\end{remark}

\paragraph*{I.i.d. Variance Estimation} The corollary below shows that Algorithm~\ref{alg:StationaryVariance} can also be used to estimate the variance in i.i.d. setting, achieving the optimal rate of convergence of $O(\frac{1}{n})$ in MSE. Moreover, the estimator satisfies all the desirable properties discussed in Section~\ref{sec:intro} -- it is recursive, does not require storing historical samples, and does not require knowing the total run length apriori.
\begin{corollary}          
    For $n\in\N$, consider a sequence of i.i.d. observations $\lrset{X_i}_{i=1}^n$ from an unknown distribution $P$. Let $\bar{f}:= \Exp{P}{f(X)}$.  Define 
    \[v(f) := \Exp{P}{\lrp{f(X) - \Exp{P}{f(X)}}^2} = \Exp{P}{f^2(X) - f(X)\bar{f}}.\] 
    Then, the output $v_n$ of Algorithm~\ref{alg:StationaryVariance} satisfies 
    \[ \E{\lrp{v_n - v(f)}^2} = O\lrp{\frac{1}{n}}. \]
\end{corollary}

The rate of convergence of $\Omega(\frac{1}{n})$ for MSE is tight for i.i.d. setting. In the special setting of i.i.d. noise with parametric distributions, this follows from the Cram\'er-Rao lower bound \cite{nielsen2013cramer}. 

\section{Variance Estimation: Linear Function Approximation}\label{sec:lfa}

When the underlying state and action spaces are large, estimating $V$ function for each state-action pair may be computationally intractable. To address this, we consider evaluating an approximation of $V$ by its projection on a linear subspace spanned by a given fixed set of $d$ vectors $\{\tilde{\phi}_1, \dots, \tilde{\phi}_d\}$,
where $\tilde{\phi}_i \in \R^{S}$ for each $i\in (d)$. More concretely, for $s \in \calS$, and $\theta\in \R^d$, we consider the following linear function approximation (LFA) of $V(s)$: $$V_\theta(s) := \phi(s)^T\theta, \quad \text{where} \quad \phi^T(s) := [\tilde{\phi}_1(s)  \dots \tilde{\phi}_d(s)],$$ 
$\phi^T(s)$ is the feature vector for state $s$, and $\phi(s)\in\R^d$. Recall that $S = |\calS|$. Let $\Phi$ be a $S\times d$ matrix with $\tilde{\phi}_i$ being the $i^{\text{th}}$ column, and let $$W_\Phi = \{\Phi\theta : \theta \in \R^d\}$$ denote the column space of $\Phi$. Then, $V_\theta = \Phi \theta$, where $V_\theta \in \R^{S\times 1}$ is an approximation for $V$ using $\theta$.

\begin{assumption}\label{asmp:2}
    {The matrix $\Phi$ is full rank, i.e., the set of feature vectors $\{\tilde{\phi}_1, \dots, \tilde{\phi}_d\}$ are linearly independent. Additionally, $\|\phi(s)\|_2 \le 1$, for each $s\in\calS$.}
\end{assumption}

This is a standard assumption, and can be achieved by feature normalization, see \cite{tsitsiklis1999average,BertsekasT96}. We now introduce some notation that will be used in this section. Let $\theta_e\in\R^d$ be the unique vector (if it exists) such that $\Phi\theta_e = {\bf e}$. Define subspace $S_{\Phi, e}$ of $\R^d$ as 
\begin{align*}
S_{\Phi,e} := \text{span}\lrp{ \lrset{\theta | \Phi \theta = {\bf e}}} = 
    \begin{cases}
        \lrset{c\theta_e | c\in \R}, \quad & \text{ if }{\bf e} \in W_\Phi,\\
        \{0\}, \quad & \text{ otherwise. }
    \end{cases}
    \end{align*}
Let $E$ be the subspace of $\R^d$ that is orthogonal complement (in $\ell_2$-norm) of $S_{\Phi, e}$, i.e., 
\begin{align*}
    E = \begin{cases}
        \{ \theta \in \R^{d} : ~ \theta^T  \theta_e = 0 \}, \quad & \text{ if } {\bf e}\in W_\Phi,\\
        \R^d, \quad &\text{ otherwise.}
    \end{cases}
\end{align*}
Observe that for $\theta \in E$, $\Phi \theta$ is a non-constant vector. This is because there does not exist $\theta\in E$ such that $\Phi\theta = {\bf e}$. This follows since if ${\bf e}\in W_\Phi$, then the unique vector $\theta_e$ such that $\Phi \theta_e = {\bf e}$, does not belong to $E$. On the other hand, if ${\bf e}\notin W_\Phi$, then there does not exist a vector $\theta\in\R^d$, hence in $E$, such that $\Phi\theta = {\bf e}$. 

Finally, let $\Pi_{2,E}$ denote the orthogonal projection of vectors in $\R^d$ (in $\ell_2$-norm) on the subspace $E$, i.e.,
\[ \Pi_{2,E} := \begin{cases} 
    I - \theta_e (\theta^T_e \theta_e)^{-1} \theta^T_e, \quad & \text{ if } {\bf e} \in W_\Phi,\\
    I, \quad & \text{ otherwise.}
    \end{cases}
\]

\subsection{The Algorithm}\label{sec:alglfa}
We now propose an algorithm for estimating the asymptotic variance with LFA that, as in the previous section, uses the formulation of $\kappa$ in~\eqref{lem:var_form_twotraj}, with $V$ replaced by its approximation. In the setting considered here, the estimation of $V$ corresponds to an estimate for $\theta$ at each step; call it $\theta_k$. The corresponding estimate for $V$ is then $\Phi\theta_k$.

The algorithm in this setting is a modification of Algorithm~\ref{alg:PE}, with ~\eqref{eq:deltaUpdate},~\eqref{eq:QUpdate},~\eqref{eq:barVUpdate}, and~\eqref{eq:kappaUpdate} replaced by~\eqref{eq:lfadeltaUpdate},~\eqref{eq:lfaThetaUpdate},~\eqref{eq:lfabarVUpdate}, and~\eqref{eq:lfaT1Update} below, respectively. 
\begin{equation}
    \delta_k :=  f(X_k) - \bar{f}_{k} + (\phi(X_{k+1}) - \phi(X_k))^T\theta_k,\label{eq:lfadeltaUpdate}
\end{equation}
\begin{equation}
    \theta_{k+1} = \theta_k + \alpha_k \Pi_{2,E}\phi(X_k) \delta_k,\label{eq:lfaThetaUpdate}
\end{equation}
\begin{equation}
    \widetilde{V}_{k+1} = (1-c_2\alpha_k)\widetilde{V}_k + c_2 \alpha_k \phi^T(X_k)\theta_k \numberthis\label{eq:lfabarVUpdate}
\end{equation}
\begin{equation}
    \kappa_{k+1} = (1-c_3\alpha_k)\kappa_k + c_3 \alpha_k \lrp{ 2f(X_k) \phi^T(X_k)\theta_k - 2 f(X_k) \widetilde{V}_k - f^2(X_k) + f(X_k)\bar{f}_k }      \label{eq:lfaT1Update}
\end{equation}
Here, updates in~\eqref{eq:lfadeltaUpdate} and~\eqref{eq:lfaThetaUpdate} correspond to SA for estimating $\theta$, adjusted with the projection on subspace $E$. As in Section~\ref{sec:varest:tabular}, this projection step is to ensure convergence of the iterates $\widetilde{V}_k$, and is not required otherwise (Remark~\ref{rem:proj}). Note that Algorithm~\ref{alg:PE} for the tabular setting is a special case with $d = S$ and $\Phi = I$, the identity matrix in $S\times S$ dimensions.

\iffalse
\begin{algorithm2e}
\RestyleAlgo{ruled}
\DontPrintSemicolon 
%\setstretch{1.4}
\caption{Policy Evaluation for Variance with Linear Function Approximation}\label{alg:PELFA}

    \KwIn{Time horizon $T > 0$, constants $c_1 > 0$, $c_2 > 0$, and step-size sequence $\lrset{\alpha_k}$.}
    
    {\bf Initialization: } $J_0  = 0$, $\theta_0 = \vec{\bf  0} $, and $\kappa_0 =0$.\\
    
    \While{$k \le T$}{
        \begin{enumerate}
        \setlength\itemsep{0em}
        \item  On first trajectory, take actions $A_k \sim \mu(\cdot| S_k)$ \\
        and $A_{k+1}\sim \mu(\cdot| S_{k+1})$ to observe ($S_k, A_k, r(S_k, A_k), S_{k+1}, A_{k+1}$). \\
        
        \item On second trajectory, take action $A'_k \sim \mu(\cdot| S'_k)$ \\
        and observe ($S'_k, A'_k, r(S'_k, A'_k)$).\\

        \item Update estimate for $J$: 
        
        \(J_{k+1} = J_k + c_1\alpha_k ( -J_k + r(S_k, A_k) ). \)
            
        \item Update estimate for $\theta$:
        
        $\delta_k :=  r(S_k, A_k) - J_{k} + (\phi_{k+1}- \phi_k)^T\theta_k,$
        
        $\theta_{k+1} = \theta_k + \alpha_k \Pi_{2,E}\phi_k \delta_k.$ 

        \item Update estimate for $\kappa$:

        $T_1  := 2(r(S_k,A_k)-r(S'_k,A'_k))\phi^T_k\theta_{k},$
        
        $T_2 := \frac{1}{2}\lrp{r(S_k,A_k) - r(S'_k,A'_k)}^2,$

        $\kappa_{k+1} =  \kappa_{k} + c_2\alpha_k\lrp{T_1 - T_2 - \kappa_{k}} .$

        \item Increment $k$ to $k+1$.
        \end{enumerate}
    }
\end{algorithm2e}
\fi 

\subsection{Convergence Rates}\label{sec:mainresultslfa}
In this section, we present a finite-sample bound on the estimation error of the proposed algorithm with LFA. Recall the Markov chain $\mathcal M_2$ introduced in Section~\ref{sec:theorype}. Define 
\[ \Delta_2 := \min\limits_{\|\theta\|_2 = 1, \theta \in E} ~ \theta^T \Phi^T D_\pi(I-P) \Phi \theta. \]
As earlier, this can be shown to be strictly positive since for $\theta\in E$, $\Phi \theta$ is a non-constant vector. Finally, let $\Theta^{*T} := [\bar{f}~ \theta^{*T} ~ 
\widetilde{V}~ \kappa^* ]$, 
where 
\begin{align*}
    \kappa^* = \Exp{\pi}{2f(X)[\Phi\theta^*](X) - 2f(X)\widetilde{V} - f^2(X) + f(X)\bar{f}}, \numberthis\label{eq:kappastar}
\end{align*}
where $\Phi\theta^* \in \R^S$, and $[\Phi\theta^*](i)$ represents its $i^{th}$ component, $\widetilde{V} = \Exp{\pi}{[\Phi\theta^*](X)}$,
and $\theta^*$ is the unique vector in $E$ that is also a solution for \[ \Phi\theta = \Pi_{D_\pi, W_\Phi} T\Phi\theta,  \numberthis\label{eq:projBellman}\]
where
$\Pi_{D_\pi, W_\Phi}$ is the projection matrix onto $W_\Phi:= \lrset{\Phi\theta | \theta\in\R^d}$ with respect to $D_\pi$ norm, and $D_\pi$ is defined above~\eqref{eq:eta}. Specifically, 
$$\Pi_{D_\pi, W_\Phi} = \Phi(\Phi^TD_\pi\Phi)^{-1}\Phi^T D_\pi.$$ 
Further, $T$ is an operator that for a vector $V\in \R^{S}$, satisfies $T V = \mathcal F - \bar{f} {\bf e} + P V.$ 

Observe that $\kappa^*$ differs from $\kappa(f)$ in that $V$ and $\bar{V}$ in formulation~\eqref{lem:var_form_twotraj} are replaced by the estimate in the subspace spanned by $\Phi$ and $\widetilde{V}$, respectively. The following theorem bounds the mean-squared distance between the estimate at time $n$ and the limit point $\kappa^*$ under constant as well as diminishing step sizes. 

\begin{theorem}\label{th:finiteboundlfa}
    Consider the algorithm with $c_1$, $c_2$, and $c_3$ satisfying:
    \[c_1\ge \frac{1}{2\Delta_2} + \frac{\Delta_2}{2}, \qquad  \frac{5}{249}(5-2\sqrt{2})\Delta_2 \le c_3\le \frac{5}{249} (5+2\sqrt{2})\Delta_2, \]
    and
    \[ c_3 -\frac{498 c^2_3}{7\Delta_2} + \frac{7 \Delta_2}{498}   \le c_2 \le -3 c_3 + \frac{5 }{83} \sqrt{498 c_3 \Delta_2 - 17 \Delta^2_2} . \]
    Recall that $\eta$ was defined in~\eqref{eq:eta}, and $\Theta^*$ before~\eqref{eq:kappastar}. For a constant $B > 1$, let  $\xi_1 :=  3\lrp{1+\|\Theta^*\|_2}^2$, $\xi_2 := 112B(1+\|\Theta^*\|_2)^2$, and let $\Theta_n := [ \bar{f}_n~ \theta^T_n~\widetilde{V}_n~\kappa_n]$ denote the vector of estimates at step $n$.
    \begin{enumerate}%[label=(\alph*)]
        \item[(a)] Let $\alpha_i = \alpha$ for all $i$, such that     
        \[ \alpha < \min\lrset{ \frac{40}{\Delta_2}, \frac{1}{28B\lrp{1+ \frac{20\eta^2}{\Delta_2} }}}. \]
        Then, for all $n\ge 1$, 
        \[ \E{\|\Theta_n - \Theta^*\|^2_2} \le \xi_1\lrp{1-\frac{\Delta_2\alpha}{40}}^{n} +  \frac{20\xi_2 \alpha\eta^2}{\Delta_2} + \xi_2\alpha. \]

        \item[(b)]  Let $\alpha_i = \frac{\alpha}{i+h}$ for all $i \ge 0$, with $\alpha$ and $h$ chosen so that
        \[ h \ge \max\lrset{1+\frac{\alpha\Delta_2}{40}, 1+28B\lrp{ \frac{20\eta^2\alpha}{\Delta_2} + \alpha + \frac{20}{\Delta_2} }},\quad\text{and}\quad \alpha > \frac{40}{\Delta_2}.\]
        Then, for all $n\ge 0$, 
        \begin{align*}
        \E{\|\Theta_n- \Theta^*\|^2_2}  &\le \xi_1\lrp{\frac{h}{n+h}}^\frac{\alpha\Delta_2}{40} + \frac{5\xi_2e^2\eta^2(20+\Delta_2)\alpha^2}{(n+h)(\alpha\Delta_2 -40)} + \frac{\xi_2 \alpha}{n+h}.
        \end{align*} 
    \end{enumerate}
\end{theorem}
The above theorem bounds the MSE between the iterates of the algorithm, $\Theta_n$, and its limit point $\Theta^*$. In particular, since $ \E{(\kappa_n - \kappa^*)^2}  \le \E{\|\Theta_n - \Theta^*\|^2_2}$, the same bounds also hold for the MSE for $\kappa$ estimation. To prove Theorem~\ref{th:finiteboundlfa}, we view the proposed algorithm as a linear SA update, and prove appropriate contraction properties for the associated matrices (Lemma~\ref{th:negdeflfa}). We refer the reader to Section~\ref{sec:proofsketch} for a proof of the theorem.

\subsection{Approximation Error} Since we approximate the elements of the set $S_f$ of $V$ functions (defined after~\eqref{eq:PE}) by a linear combination of basis vectors, we incur an approximation error in estimating $\kappa(f)$ using~\eqref{lem:var_form_twotraj}. In particular, $\kappa^*$ may differ from $\kappa(f)$ depending on the chosen approximation architecture. 

In this section, we bound the squared approximation error: $(\kappa(f)-\kappa^*)^2$. To this end, we first define the approximation error associated with approximating $V$ using LFA. Since each element of $S_f$ is a valid $V$ function, for any $\theta\in\R^d$, we define the error of approximating a point in $S_f$ by $\Phi\theta$ as the minimum weighted distance (in $D_\pi$-norm) of $\Phi\theta$ from $S_f$ \cite{tsitsiklis1999average}, i.e., 
 $\inf\nolimits_{V\in S_f } \|\Phi\theta - V\|_{D_\pi}$.
The minimum possible approximation error in $V$, due to the chosen architecture, is given by 
\[ \mathcal E := \inf\limits_{\theta\in\R^d} \inf\limits_{V\in S_f} \|\Phi\theta -V\|_{D_\pi}. \]
This essentially captures the distance between the two sets, $S_f$ and the column space of $\Phi$. The following proposition bounds the squared difference between $\kappa^*$ and the asymptotic variance $\kappa(f)$, in terms of the approximation error for $V$.

\begin{proposition}\label{prop:approxerror} Given $d$ basis vectors represented as columns of $\Phi$, there exists $\lambda \in (0,1)$ such that 
    \[ (\kappa^* -\kappa(f))^2 \le \frac{16~{ \mathcal E }^2}{1-\lambda^2}. \numberthis\label{eq:approxerr} \]
\end{proposition}
Notice that $\mathcal E$, and hence the RHS above, equals $0$ if the chosen basis functions are such that the span of these basis functions intersects with the set $S_f$ of $V$ functions, i.e., if the approximation architecture is capable of representing some value function from $S_f$. This is particularly true for the tabular setting, where $d = S$ and the basis vectors are the standard basis in this dimension, spanning the entire space. 

To prove Proposition~\ref{prop:approxerror}, we first bound the approximation error in estimation of $V$ using the proposed Algorithm from Section~\ref{sec:alglfa}. This is analogous to that for the value function in \cite[Theorem 3]{tsitsiklis1999average}. Next, using this bound, we arrive at the bound in~\eqref{eq:approxerr}. We refer the reader to Section~\ref{app:approxerr} for a proof of the proposition.

%\shubhada{Example where this approximation error is smaller than approximating $\kappa$ using $v$, the stationary variance.}

\subsection{Proof of  Theorem~\ref{th:finiteboundlfa}}\label{sec:proofsketch}
The proof of Theorem~\ref{th:finiteboundlfa} proceeds in two steps. In the first step, we formulate the proposed algorithm as a linear SA update, and establish an appropriate contraction property of the update matrices averaged under the stationary distribution $\pi$ of the underlying Markov chain. In the second step, we use a Lyapunov drift argument to establish a MSE bound for a linear SA update satisfying the contraction property from step one. To this end, we use Poisson equation again (recall its use in the estimator design), to bound the errors associated with the noisy updates (since the algorithm does not know $\pi$) at each step. We detail these below.

\subsubsection{Step 1: Algorithm as a Linear SA}\label{sec:alglinearSA}
For $X_k\in\calS$, $X_{k+1}\in\calS$, and $k \in \N$, define 
\[f_k := f(X_k), ~~\phi_k = \phi(X_k).\numberthis\label{eq:r_k}\]
Recall that for each $x\in\calS$, $\phi(x)\in\R^d$. Further, recall that  
$Y_k := (X_k, X_{k+1})$, and let  $\vec{\bf 0}$ be the $0$ vector in $\R^{d}$. Define $A(Y_k)\in\R^{d + 3}\times\R^{d + 3}$, 
\[  A(Y_k):=\begin{bmatrix} 
                    -c_1  & \vec{\bf 0}^T & 0 & 0\\
                    -\Pi_{2,E}\phi_k & \Pi_{2,E}\phi_k (\phi^T_{k+1}  - \phi^T_k) &\vec{\bf 0} & \vec{\bf 0} \\
                    0 & c_2\phi^T_k  & -c_2 & 0\\
                    c_3 f_k & 2 c_3 f_k \phi^T_k & -2c_3 f_k & -c_3
                \end{bmatrix}, \]
and $b(Y_k) \in \R^{d+3}$ as    \[     b(Y_k)^T := \begin{bmatrix}
                    c_1 f_k ~~~~
                    f_k \lrs{\Pi_{2,E}\phi_k}^T ~~~ 0 ~~~
                    -{c_3}f^2_k
                \end{bmatrix}.
\]
Let $A = \Exp{\pi}{A(Y_k)}$ and $b = \Exp{\pi}{b(Y_k)}$ denote the stationary averages of $A(Y_k)$ and $b(Y_k)$, respectively (see Appendix~\ref{app:avgmatrices} for exact form of $A$ and $b$), and let $\Theta^T := [\theta_1~\theta^T_2~\theta_3~\theta_4]$ with $\theta_2\in E$. Observe that $\Theta = \Theta^*$ (defined before~\eqref{eq:kappastar}) is the unique solution to $A\Theta + b = 0$ with $\theta \in E$ (see Appendix~\ref{app:odedisc} for a detailed justification). However, the algorithm doesn't have access to matrices $A$ and $b$. Instead, we use SA to solve for $\Theta^*$, which corresponds to the following update rule at step $k+1$ with step size $\alpha_k$ (also, recall update from~\eqref{eq:SAForLinear}): 
\begin{equation}\label{eq:linearsa}
    {\Theta}_{k+1} = {\Theta}_k + \alpha_k \lrp{ A(Y_k) {\Theta}_k + b(Y_k) }.
\end{equation} 
Observe that the above equation with $\Theta^T_k := [ \bar{f}_k ~ \theta^T_k ~ \widetilde{V}_k ~ \kappa_k ]$ coincides with the updates in the proposed algorithm (Algorithm in Section~\ref{sec:alglfa}). %Hence, it is a linear SA update~\eqref{eq:linearsa} with the matrices $A(Y_k)$ and $b(Y_k)$ as defined earlier in this section.

Next, once we have the linear SA form in~\eqref{eq:linearsa}, to establish the convergence of iterates, in Lemma~\ref{th:negdeflfa}, we prove that $A$, the average matrix,  is contractive (under a semi-norm) for suitably chosen step-size constants $c_1$,  $c_2$, and $c_3$. This result is crucial in establishing the finite sample bounds in Theorem~\ref{th:finiteboundlfa}.

\begin{lemma}\label{th:negdeflfa}
Under the conditions on $c_1$, $c_2$, and $c_3$ from Theorem~\ref{th:finiteboundlfa}, the matrix $A$ satisfies
\[\min \limits_{\substack{\Theta \in \R\times E \times \R \times \R, \|\Theta\|^2_2 = 1}} ~ - \Theta^T A \Theta  ~ > ~ \frac{\Delta_2}{20}> 0.
\]
\end{lemma}
We refer the reader to Appendix~\ref{app:th:negdeflfa} for a proof of Lemma~\ref{th:negdeflfa}. In the special case of $E = \R$, the above property can be shown to imply that $A$ is Hurwitz.

\begin{remark}
    The constant $20$ in the above Lemma is not special, and is a result of certain simplifying bounds used in the proof of the Lemma, which helped us in arriving at explicit bounds on $c_1$, $c_2$, and $c_3$ in Theorem~\ref{th:finiteboundlfa}. Modifying this constant would only change the constraints on the step size constants. We present the most general conditions (but slightly opaque) on $c_1$, $c_2$, and $c_3$ in Equations~\eqref{eq:cond1},~\eqref{eq:cond2}, and~\eqref{eq:cond3} in Appendix~\ref{app:th:negdeflfa}.
\end{remark}

\subsubsection{Step 2: Bounding MSE for Linear SA with Markov Noise} \label{sec:MSEBound} 
In Step 1, we formulated the updates in our estimator as a linear SA. One could use existing finite sample bounds on MSE for linear SA \cite{srikant2019finite}, which require $A$ to be Hurwitz. However, Lemma~\ref{th:negdeflfa} only guarantees this when $E = \R$, i.e., when $\bf{1}$ (the all $1$s vector) does not belong to the column space of $\Phi$. Recall that we do not make any such assumption. 

Recently, \cite{zhang2021finite} extend the MSE bound of \cite{srikant2019finite} for the setting when $E \ne \R$. Using \cite[Theorem 1]{zhang2021finite}, we could get MSE bound of  $O(\frac{\log n}{n})$ at step $n$ for our estimator with $\alpha_n = O(\frac{1}{n})$. Instead, we conduct a tighter analysis of linear SA, in general. We use Poisson equation to get a tight control on the error terms  in our analysis in Appendix~\ref{app:Step2} (Theorem~\ref{th:appfiniteboundlfa}), which result in $O(\frac{1}{n})$ MSE bound, instead. Our bounds from Theorem~\ref{th:appfiniteboundlfa} immediately imply those in Theorem~\ref{th:finiteboundlfa}. We refer the reader to Section~\ref{sec:verifyconds2} for details.

\section{Vector Valued Rewards}
Our algorithm and analysis easily extend to the setting where the function ${\bf f}$ is $\frakd > 1$ dimensional, i.e., ${\bf f}: \calS\rightarrow\R^\frakd$. We illustrate this with a discussion on the extension of  Algorithm~\ref{alg:PE} for estimating the (per-step) asymptotic covariance matrix (defined next) and a brief proof sketch. 

The (per-step) asymptotic covariance matrix for a function defined on states of a Markov chain starting from state $s$ is defined as 
\[ \lim\limits_{n\rightarrow\infty}\operatorname{Cov}\lrp{ \frac{1}{\sqrt{n}} \sum\limits_{l=0}^n {\bf f}(X_l) ~\bigg|~ X_0 = s }, \]
where, for a vector-valued function ${\bf r}:\calS\rightarrow\R^\frakd$ with mean $\bar{\bf r}:= \Exp{\pi}{{\bf r}(X)}$, where ${\bar{\bf r}}\in\R^\frakd$ (represented as column vectors), $$\operatorname{Cov}\lrp{\bf  r } := \E{ (\bf r-\bar{\bf r})(\bf r-\bar{\bf r})^T }.$$ 

For $i\in [\frakd] := \{1, \dots, \frakd\}$ and $s\in\calS$, let $f^{(i)}(s)$ denote the $i^{th}$ component of the (column) vector ${\bf f}(s) \in \R^{\frakd}$ of the function value in state $s$, and let 
$$\bar{\bf f} := \Exp{\pi}{{\bf f}(X)}$$ 
be the mean vector under $\pi$, the stationary distribution of $\mathcal M$. Then, $\bar{\bf f} \in \R^\frakd$. For $i\in [\frakd]$, let $\bar{f}^{(i)}\in\R$ denote its $i^{th}$ coordinate. Then, there exists a solution $V^{(i)}: \calS \rightarrow \R$ ($V^{(i)}\in\R^{S}$) of the Poisson equation for $f^{(i)} - \bar{f}^{(i)}$ for each $i\in [\frakd]$, i.e., for each $s\in \calS$, and $i\in[\frakd]$, $V^{(i)}(s)$ solves 
\[ f^{(i)}(s) - \bar{f}^{(i)} = V^{(i)}(s) - \sum\limits_{s'\in\calS} P(s,s') V^{(i)}(s'). \]
In fact, as in the settings of previous sections ($\frakd=1$ case), one can show that for each coordinate $i \in [\frakd]$, if $V^{(i)}$ is a solution, then $V^{(i)} + c {\bf e}$ is also a solution. We denote by $V^{*(i)}$ the solution normalized so that $\mathbb{E}_{\pi}[V^{*(i)}] = 0$.

As in Proposition~\ref{prop:var}, one can argue that the asymptotic covariance matrix defined above is independent of the starting state, and is given by 
\[ \frakC({\bf f}) = \lim\limits_{n\rightarrow\infty} ~ \operatorname{Cov}\lrp{ \frac{1}{\sqrt{n}}\sum\limits_{l=0}^{n-1} {\bf f}(X_l) ~\bigg| ~X_0 \sim \pi }. \] 
For $s\in\calS$, let ${\bf V}(s) \in \R^\frakd$ be
    \[ {\bf V}^T(s) = [V^{(1)}(s) ~\dots~ V^{(\frakd)}(s)]. \]
Moreover, given (column) vectors ${\bf u}$ and ${\bf v}$ in $\R^\frakd$, let ${\bf u}\otimes {\bf v}^T$ denote their outer product, resulting in a matrix in $\R^{\frakd\times\frakd}$. Then, $\frakC(\bf f)$, which is a matrix in $\R^{\frakd\times\frakd}$, has the following equivalent representation.
\begin{align*}
    \frakC({\bf f})  
    = \Exp{\pi}{\lrp{{\bf f}(X) - \bar{\bf f}} \otimes {\bf V}^T(X)} &+ \Exp{\pi}{{\bf V}(X) \otimes \lrp{{\bf f}(X) - \bar{\bf f}}^T }   \\
    & \quad - \Exp{\pi}{ \lrp{{\bf f}(X) - \bar{\bf f} }\otimes\lrp{{\bf f}(X) - \bar{\bf f}}^T }. \numberthis\label{eq:cmu}
\end{align*}
This can be proven using arguments similar to those in the proof for the reformulation for $\kappa(f)$ in~\eqref{eq:kmu2} in Proposition~\ref{prop:var}. Notice that the expression for the diagonal entries of $\frakC({\bf f})$ in the formulation in~\eqref{eq:cmu} is exactly same as that for $\kappa(f)$ in~\eqref{eq:kmu2}. 

Next, let $\bar{\bf V}\in\R^\frakd$ be $\bar{\bf V} := \Exp{\pi}{ {\bf V}(X) }$. As in~\eqref{lem:var_form_twotraj}, we now re-express~\eqref{eq:cmu}. 
\begin{align*}
    \frakC({\bf f})  
    &= \Exp{\pi}{{\bf f}(X)\otimes {\bf V}^T(X)} + \Exp{\pi}{{\bf V}(X)\otimes {\bf f}^T(X)} - \Exp{\pi}{{\bf f}(X) \otimes \bar{\bf V}^T} \\ 
    &\qquad\qquad - \Exp{\pi}{\bar{\bf V} \otimes {\bf f}^T(X)} + \Exp{\pi}{{\bf f}(X)\otimes {\bf f}^T(X)} + \Exp{\pi}{ {\bf f}(X) \otimes \bar{\bf f}^T }. \numberthis\label{eq:cmu2traj}
\end{align*}

Again, notice that the expression for the diagonal entries of $\frakC({\bf f})$ in the above formulation is same as that for $\kappa(f)$ in~\eqref{lem:var_form_twotraj}. Our SA algorithm (Algorithm~\ref{alg:PE}) can be extended to estimate $\frakC({\bf f})$ using the formulation in~\eqref{eq:cmu2traj}. In particular, the updates in~\eqref{eq:JmuUpdate},~\eqref{eq:deltaUpdate},~\eqref{eq:QUpdate}, and~\eqref{eq:kappaUpdate} in the tabular setting can be replaced by the following update equations. Here, for  $k\in\N$ and $s\in\calS$, $\bar{\bf f}_k$, ${\bf \delta}_k$, ${\bf V}_k(s)$, and $\bar{\bf V}_k$ are vectors in $\R^\frakd$, and $\frakC_{k}$ are matrices in $\R^{\frakd\times\frakd}$.
\begin{equation*}
    \bar{\bf f}_{k+1} = \bar{\bf f}_{k} + c_1 \alpha_k ({\bf f}(X_k) - \bar{\bf f}_k),
\end{equation*}
\begin{equation*}
    {\bf \delta}_k := {\bf f}(X_k)  - \bar{\bf f}_k + ({\bf V}_k(X_{k+1}) - {\bf V}_k(X_k)),
\end{equation*}
\begin{equation*}
    {\bf V}_{k+1}(x) = \begin{cases} 
        {\bf V}_k(x) + {\alpha_k} {\delta}_k \lrp{ 1 - 1/{S} } , \quad &\text{ for } x = X_k,\\
        {\bf V}_k(x) - {\alpha_k}{\delta}_k/{S}, \quad &\text{ otherwise}, 
    \end{cases}
\end{equation*}
\begin{equation*}
    \bar{\bf V}_{k+1} = \bar{\bf V}_k + c_2\alpha_k ( {\bf V}_k(X_k) - \bar{\bf V}_k ),
\end{equation*}
\begin{align*}
    \frakC_{k+1} &=  (1-c_3\alpha_k)\frakC_k + c_3\alpha_k \bigg( {\bf f}(X_k) \otimes {\bf V}^T_k(X_k) + {\bf V}_k(X_k) \otimes {\bf f}^T(X_k) - {\bf f}(X_k) \otimes \bar{\bf V}^T_k \\
    & \qquad\qquad\qquad\qquad\qquad\qquad - \bar{\bf V}_k \otimes {\bf f}^T(X_k) + {\bf f}(X_k) \otimes {\bf f}^T(X_k) + {\bf f}(X_k)\otimes \bar{\bf f}^T_k \bigg). 
\end{align*}

The linear function approximation setting of Section~\ref{sec:lfa} can also be extended similarly to the current setup, with a set of $d$ feature vectors for each component in $[\frakd]$. Let's denote the feature vectors for approximating $V^{(i)}$ by $\{\tilde{\phi}^{(i)}_1, \dots, \tilde{\phi}^{(i)}_d\}$, and let $\Phi^{(i)}$ denote the $S \times d$ matrix with these $d$ vectors as columns. Algorithm from Section~\ref{sec:alglfa} can be appropriately modified to estimate $\frakC({\bf f})$ using SA for~\eqref{eq:cmu2traj}. The resulting algorithm can again be analyzed by formulating it as a linear SA for each pair $(i,j)$ of $\frakC({\bf f})$, for $i\in[\frakd]$, $j\in[\frakd]$, and $i\le j$,  giving similar guarantees on mean-squared estimation error for each component of $\frakC({\bf f})$ as in Theorems~\ref{th:finitebound} and~\ref{th:finiteboundlfa}.

\section{Applications in Reinforcement Learning}\label{sec:RL}
Consider an average reward reinforcement learning (RL) framework, where the transition dynamics of the underlying Markov decision process (MDP) are unknown \cite{BertsekasT96,sutton2018reinforcement}. An RL algorithm can obtain a sample of the MDP under a given policy, which specifies how actions are chosen in a given state. The traditional goal in an average reward RL problem is to find a policy that maximizes the long run average reward. While the need to optimize over average reward is well motivated, for applications in safety-critical domains like healthcare and finance, it is also crucial to control adverse outcomes. To address this, within the domain of risk-sensitive MDPs, several risk measures have been considered, namely variance \cite{Sobel82VD,filar1989variance}, conditional value-at-risk \cite{chow2014algorithms}, exponential utility \cite{whittle1990risk}, the general class of coherent risk measures \cite{tamar2015coherent}, and cumulative prospect theory \cite{prashanth2015cumulative}, which is not coherent. The choice of the risk measure largely depends on the application at hand.

We consider variance as a risk measure. Several previous works incorporate variance in a constrained setting, where the goal is to maximize the average reward (which is an expectation), while ensuring a certain bound on the variance. This is the so-called ``mean-variance tradeoff'', considered in the seminal work of Markowitz \cite{markowitz1952portfolio}, and later in MDP contexts, cf. \cite{mandl1971variance,Sobel82VD,filar1989variance}. An alternative to such a formulation is to consider the exponential utility formulation, see \cite{Arrow1971,Howard1972}, where one optimizes the exponential. The constrained formulation is preferred over the exponential utility for two reasons. First, the mean-variance tradeoff can be controlled directly through a parameter that is a bound on the variance, while this trade-off is implicit in an exponential utility formulation. Second, the algorithms for the latter formulation do not extend easily when one considers feature-based representations and function approximation, see \cite[Section 7.2]{MAL-091} for a detailed discussion.

Broadly, for a given policy, two different notions of variance are suggested for an average reward MDP in \cite{filar1989variance}. The first notion is the asymptotic variance, while the second one is the per-period stationary variance. The latter has been studied in an RL setting in \cite{prashanth2016variance}, while the former has hardly been investigated in the literature, to the best of our knowledge. We use the asymptotic variance as a measure of the risk associated with a given policy, which is the variance of the random variable whose mean is typically optimized.
This asymptotic variance can be decomposed into two additive terms,  where the first term coincides with per-step variance, while the other term involves correlations between states across time periods.
In a setting where the state sequence is independent, the second term is zero, while it is not in any non-trivial MDP. A mean-variance optimization formulation would consider maximizing the expectation of a random variable which represents the long-run average reward, and it is natural to consider the variance of this random variable; also see Section~\ref{sec:motivation_var}.

In this section, we consider the problem of policy evaluation for asymptotic variance, where the goal is to estimate the asymptotic variance associated with the Markov chain induced by a given stationary policy. This problem of estimating the asymptotic variance  is a vital sub-problem in mean-variance policy optimization, for instance, as a critic in an actor-critic framework, cf. \cite{prashanth2016variance}. For the discounted RL setting, a TD type algorithm for estimating variance has been proposed/analyzed in \cite{Tamar13TD}, while a TD algorithm with provable finite sample guarantees, which caters to the variance risk measure in an average reward RL setting, is not present in the literature to the best of our knowledge.

The rest of this section is organized as follows. We begin by formally introducing the average reward RL setting and the problem setup in Section~\ref{sec:RL_setup}. In Section~\ref{sec:motivation_var}, we motivate the choice of asymptotic variance as a natural measure of risk associated with a given policy in this framework and discuss the related literature in Section~\ref{sec:litRL}. We design a novel TD-like linear SA algorithm for policy evaluation in a tabular setting in Section~\ref{sec:RL_Alg}. Finally, we conclude in Section~\ref{sec:RL_bounds} with the finite sample error bounds for the proposed algorithm, proving ${O}(\frac{1}{n})$ rate of convergence for the mean-squared error, where  $n$ is the time step. We defer the algorithm and its finite sample error bounds for the setting of linear function approximation in RL to Appendix~\ref{sec:lfa_RL}.

\subsection{Setup}\label{sec:RL_setup}
Consider an infinite-horizon, average-reward MDP specified by $(\calS, \calA, r, p)$, where $\calS = \lrset{1, \dots, \abs{S}}$ denotes the finite state-space, and $\calA = \lrset{1, \dots, \abs{A}}$ denotes the  action space. At each time $k$, the agent is in state $S_k\in\calS$, takes an action $A_k\in\calA$, receives a reward $r(S_k,A_k)$, and transitions to state $S_{k+1}$. Here, $r:\calS\times\calA \rightarrow \R$, and the next state $S_{k+1}$ is sampled according to $p(S_k,\cdot, A_k)$, where, $p:\calS\times\calS\times\calA \rightarrow [0,1]$, is the map that for states $s, s'$ and action $a$ associates probability $p(s,s',a)$ with the transition from state $s$ to $s'$, when action $a$ is taken. 

A stationary policy $\mu: \calS \rightarrow \Sigma_{\abs{\calA}}$ is a map from state space to a probability-simplex in $\R^{\abs{\calA}}$, i.e., it maps each state to a distribution over the actions. Given a stationary policy $\mu$ and a state $x$, the associated probability of transitioning from state $x$ to $x'$ is given by 
$$P_\mu(x,x') = \sum\nolimits_{a\in\calA} \mu(a|x)p(x,x',a). $$

\begin{assumption}\label{asmp1} The sets $\calS$ and $\calA$ are finite. Under the stationary policy $\mu$, the induced Markov chain $\calM_1$ with state space in $\calS$ and transition probabilities given by $P_\mu$ is irreducible and aperiodic. \end{assumption}
This is a standard assumption in literature (see \cite{tsitsiklis1999average,bertsekas2012dynamic}), and guarantees that each state is visited infinitely often. As a consequence of Assumption~\ref{asmp1}, we have a unique stationary distribution associated with $\calM_1$, and the Markov chain starting from any initial distribution, converges to the stationary distribution geometrically-fast \cite[Section 4.3]{levin2017markov}. In particular, let $\pi_\mu$ denote the unique stationary distribution on $\calS$ that satisfies $\pi^T_\mu P_\mu = \pi^T_\mu$.%, and let $S_t$ denote the state of the Markov chain $\mathcal M_1$ at time $t$. Then, Assumption~\ref{asmp1} guarantees that there exist constants $C_1 > 0$ and $\rho_1 \in (0,1)$ such that,
%\[ \sup\limits_{s\in\calS} ~~ d_{TV}(\mathbb{P}(S_t| S_0 = s), \pi_\mu) \le C_1 \rho^t_1, \text{  for all $t\ge 0$ }, \]
%where $\mathbb{P}(\cdot | S_0 =s)$ denotes the probability of $\mathcal M_1$ being in state $S_t$ at time $t$, starting from state $s$, and for probability measures $P$ and $Q$, $d_{TV}(P,Q)$ denotes the total variation distance between $P$ and $Q$. 

Next, let $X_k := (S_k, A_k)$. Observe that the process $\calM_2:=\lrset{X_k}$ is also a Markov chain with finite states. Let its state space be denoted by $\calX \subset \calS\times\calA$. For any two states $x_1 := (s_1,a_1)$ and $x_2 := (s_2, a_2)$ in $\calX$, the probability of transitioning from $x_1$ to  $x_2$ under $\calM_2$ is given by 
\begin{equation}\label{eq:P2}
    P_2(x_1, x_2) := p(s_1,s_2,a_1)\mu(a_2|s_2).
\end{equation}
For $s\in\calS$ and $a\in\calA$,  define 
$d_\mu(s,a) := \pi_\mu(s)\mu(a|s).$

\begin{remark}\label{rem:MC2}
    Under Assumption~\ref{asmp1}, $\calM_2$ has a unique stationary distribution $d_\mu$ defined above, and it mixes geometrically-fast. This follows from \cite[Proposition 1]{bhatnagar2016multiscale}. 
\end{remark}

\paragraph*{Notation} Let $D_\mu \in \R^{\abs{\calS}\abs{\calA}} \times \R^{\abs{\calS}\abs{\calA}}$ be the diagonal array of stationary distribution $d_\mu$, i.e., %$D_\mu((s,a),(s,a)) = d_\mu(s,a) $  for all $(s, a)\in\calS\times\calA$.
\[D_\mu((s,a), (s,a)) := \begin{cases}
    d_\mu(s,a),&\text{ for } (s,a)\!\in\!\calS\times\calA,\\
    0, &\text{ otherwise}.
\end{cases}\] 
For vectors $x$ and $y$ in $\R^{\abs{\calS}\abs{\calA}}$, let 
$<x,y>_{D_\mu} := x^T D_\mu y,$
and $\|x\|_{D_\mu} := x^T D_\mu x$ denote the $D_\mu$ weighted inner product and the induced norm, respectively.

\paragraph*{Average Reward}  The long-term per-step expected reward accumulated by a stationary policy $\mu$ starting from state $s$ is given by  
\begin{equation}\label{eq:Jmu_def} \lim\limits_{n\rightarrow \infty} ~  \frac{1}{n}\E{ \sum\limits_{k=0}^{n-1} r(S_k, A_k) | S_0 = s }. \end{equation}
Here, the expectation is with respect to both the transition dynamics of the environment, as well as the randomness in the policy $\mu$. It is well known that under Assumption~\ref{asmp1}, this per-step expected reward in~\eqref{eq:Jmu_def} is a constant that is independent of the starting state \cite{bertsekas2012dynamic}. We denote this quantity by $J_\mu$. It can be shown that
\begin{equation}\label{eq:Jmu}
    J_\mu = \Exp{d_\mu}{r(S,A)} = \sum\limits_{x\in\calS, a\in\calA} \pi_\mu(x) \mu(a|x) r(x, a). 
\end{equation}
Here, $\Exp{d_\mu}{\cdot}$ denotes the expectation when $(S,A)$ is sampled according to the distribution $d_\mu$, which is defined after~\eqref{eq:P2}. 

\paragraph*{Asymptotic Variance} Observe that $r(\cdot,\cdot)$ is a function defined on the state space of $\mathcal M_2$, which is fixed by the given policy $\mu$. In this work, we consider the problem of estimating the associated asymptotic variance, 
\begin{equation}\label{eq:var_def_RL}
    \lim\limits_{T\rightarrow\infty} ~ \frac{1}{T}\V\lrs{\sum\limits_{k=0}^{T-1}r(S_k,A_k) \bigg| S_0 = s}.
\end{equation}
It can be shown to be a constant, independent of the starting state. We denote it by $\kappa_\mu$. Moreover (Proposition~\ref{prop:var}),
\[ \kappa_\mu = \lim\limits_{n\rightarrow \infty} \V\lrs{\frac{1}{\sqrt{n}} \sum\limits_{k=0}^{n-1} r(S_k, A_k) \bigg| (S_0, A_0)\sim d_\mu }.\]

\paragraph*{Q Function} Since $J_\mu = \Exp{d_\mu}{r(S,A)}$, there exists a solution of the following Poisson equation or Bellman equation \cite[Section 21.2]{douc2018markov} $Q_\mu:\calX\rightarrow\R$ such that for each $(s,a)\in\calX$, and for $P_2$ given in~\eqref{eq:P2},
\begin{align*}
r(s,\!a)\!-\!J_\mu =~ &  Q_\mu(s,a)  -\!\sum\limits_{\!(\!s'\!,a'\!)\in\calX\!}\! P_2((s,a)),\!(s',a'))Q_\mu(s',a'). \numberthis \label{eq:qmu_temp}
\end{align*}
Let $Q_\mu$ be a solution to~\eqref{eq:qmu_temp}. Then, $Q_\mu + c{\bf e}$ for any constant $c\in \R$ is also a solution. Let $Q^*_\mu$ denote the solution normalized so that it is orthogonal to ${\bf e}$, i.e., $Q^{*T}_\mu {\bf e} = 0$. 

\begin{remark} Note that the equivalent formulations for $\kappa_\mu$ from Proposition~\ref{prop:var} also hold with $r$, $J_\mu$, $Q_\mu$, and $d_\mu$ replacing $f$, $\bar{f}$, $V$, and $\pi$, respectively.
\end{remark}

In this section, our goal is to estimate the asymptotic variance associated with the Markov chain induced by a given stationary policy, i.e., estimate $\kappa_\mu$. Note that this goal is orthogonal to that of estimating the variance of iterates of any particular learning algorithm like temporal-difference (TD), or designing algorithms with minimum variance of the iterates of the algorithm, studied extensively in \cite{devraj2017fastest,yin2020asymptotically,chen2020explicit,hao2021bootstrapping}; see Section~\ref{sec:lit} for a discussion and an example highlighting this distinction.

\subsection{Motivation for Asymptotic Variance}\label{sec:motivation_var}
As discussed in the beginning of Section~\ref{sec:RL}, different notions of variance have been considered in the risk-sensitive RL literature. We now motivate the choice of asymptotic variance in a risk-sensitive setting

First, in the average reward setting, a classical goal is to optimize the long-term expected cumulative reward, which corresponds to expectation of  the random variable $\sum_k r(S_k,A_k)$. The asymptotic variance considered in this work, corresponds to the variance of the same random variable. Since both the mean and variance of $\sum_k r(S_k, A_k) $ summed up to $n$ terms are $O(n)$, we equivalently consider $\frac{1}{n}$ scaling of both the mean and variance to arrive at $J_\mu$ and $\kappa_\mu$ in~\eqref{eq:Jmu_def} and~\eqref{eq:var_def_RL}, respectively. 

Second, consider an investor receiving return $r(S_k, A_k)$ on an investment at each time $k$. The total return of the investor in $n$ steps is $\sum_{k=1}^n r(S_k, A_k)$. A risk-averse investor would aim to maximize the average cumulative return subject to its variance being small. The variance on the return at each time is $\Exp{d_\mu}{(r(S,A)-J_\mu)^2}$. If the sequence $\lrset{r(S_k,A_k)}_{k\ge 1}$ were i.i.d. according to $d_\mu$, the variance of total reward over $n$ steps is 
%constraint would reduce to a bound on 
$n \Exp{d_\mu}{(r(S,A)-J_\mu)^2}$. 
This corresponds to the first term in~\eqref{eq:kmu1}. However, since the returns are Markovian, and thus are correlated across time steps, this term captures only the variance over individual times, and does not capture the correlations across time steps. The second term in~\eqref{eq:kmu1} captures the covariance of returns across time, and so the variance of the total reward $\sum_{k=1}^n r(S_k, A_k)$ is approximately $n\kappa_{\mu}$. In other words, the per step variance can be modeled by $\lim\limits_{n\rightarrow\infty}\frac{1}{n} \V[\sum_{k=1}^n r(S_k,A_k)] = \kappa_\mu$.

Finally, it is important to emphasize that prior research on variance-constrained average reward RL, cf. \cite{prashanth2016variance}, treats the first term in \eqref{eq:kmu1} as a surrogate for variance. In a setting where the state sequence is independent, the second term is zero, while in a Markovian context, the correlation between time steps is non-negligible, and should not be overlooked. 

\subsection{Related Literature}\label{sec:litRL} In this subsection, we briefly discuss the RL literature that is relevant to our work. 

    \paragraph*{MDPs: Average Reward and Risk-sensitivity} MDPs have a long history. We refer the reader to classical books \cite{puterman1994markov, bertsekas2012dynamic} for a textbook introduction to MDPs in general, and to \cite{howard1960dynamic,blackwell1962discrete,brown1965iterative,veinott1966finding,Arapostathis1993} for an introduction to average-reward MDPs, in particular. Risk-sensitive objectives have also been well-studied in the MDP setting. For instance, see \cite{Sobel82VD,Filar95PP,mannor2013algorithmic} for variance in discounted and average reward MDPs, \cite{borkar2002risk,borkar2010risk,whittle1990risk} for the exponential utility formulation, \cite{Ruszczynski10RA} for Markov risk measures, \cite{chow2014algorithms} for conditional value-at-risk (CVaR). 
    However, in a MDP setting, algorithms require complete knowledge of the underlying model, which may not be feasible in many practical applications.

\paragraph*{Risk-neutral RL} In the risk-neutral RL setting, expected value is the sole objective. TD type algorithms have been proposed for policy evaluation in discounted as well as average reward settings, and their asymptotic convergence is shown in \cite{tsitsiklis1996analysis,tsitsiklis1999average}, respectively. TD algorithms have also been used in actor-critic style algorithms for solving the problem of control, cf. \cite{Konda2003,bhatnagar2009natural}. Asymptotic convergence of the classical Q-learning algorithm was established in \cite{borkar2000ode,tsitsiklis1994asynchronous}. In the non-asymptotic regime, finite-sample mean-square convergence bounds for classical discounted setting algorithms such as TD, TD($\lambda$), n-step TD, and Q-learning, have been developed in \cite{chen2021lyapunov}. On the other hand, in the average reward setting, finite-sample bounds for TD are derived in \cite{zhang2021finite}. 
    
\paragraph*{Risk-sensitivity in RL}    In a risk-sensitive optimization setting, the goal is either to optimize the usual expected value objective, while factoring a risk measure in a constraint, or optimize a risk measure in the objective. Tail-based risk measures such as variance, CVaR are meaningful to consider as a constraint, while risk measures such as exponential utility and prospect theory can be considered as the optimization objective, since they consider the entire distribution. In the context of RL, variance as a risk measure has been studied earlier in a discounted reward MDP setting in \cite{Mihatsch02RS,la2013actor}, in a stochastic shortest path and discounted settings \cite{Tamar13VA, Tamar13TD,tamar2012policy,tamar2016learning}, and in an average reward MDP setting in \cite{prashanth2016variance}. Exponential utility has been explored in an average reward RL setting earlier, see \cite{borkar2010learning} for a survey and \cite{moharrami2022policy} for a recent contribution. \cite{huang2021convergence,ding2022sequential,kose2021risk} study CVaR as a risk measure in the discounted setting. Other risk measures such as more general coherent risk measures and  cumulative prospect theory have been explored in an RL setting in the literature, and some representative works include \cite{prashanth2014cvar,prashanth2015cumulative,markowitz2023risk,ruszczynski2023risk}. 

\subsection{The Algorithm}\label{sec:RL_Alg}
We now present a TD-type algorithm for estimating $\kappa_\mu$. Recall that $r(\cdot,\cdot)$ is a function defined on $\calX$, states of the discrete time Markov chain $\mathcal M_2$ (introduced above~\eqref{eq:P2}), with stationary expectation $J_\mu$. Further, we have from~\eqref{eq:qmu_temp} that $Q_\mu$ is a solution for the corresponding Poisson equation. Let $\bar{Q}_\mu := \Exp{d_\mu}{Q_\mu(S,A)}$. Then,
\begin{equation*}
    \kappa_\mu = \Exp{d_\mu}{2r(S,A)Q_\mu(S,A) - 2r(S,A)\bar{Q}_\mu - r^2(S,A) + r(S,A)J_\mu}.\numberthis\label{eq:kappaRL}
\end{equation*}
This immediately follows from the formulation in~\eqref{lem:var_form_twotraj} with $r(\cdot,\cdot)$, $J_\mu$, $Q_\mu(\cdot,\cdot)$, $\bar{Q}_\mu$, and $d_\mu$ replacing $f(\cdot)$, $\bar{f}$, $V(\cdot)$, $\bar{V}$, and $\pi$, respectively. Algorithm~\ref{alg:PE_RL}, presented next, corresponds to a linear SA algorithm for estimating $\kappa_\mu$ using the above formulation.

\begin{algorithm2e}[tb]
\RestyleAlgo{ruled}
\DontPrintSemicolon
\SetNoFillComment
\SetKwInOut{Input}{Input}\SetKwInOut{Output}{Output}
%\setstretch{1.05}
\caption{Estimating $\kappa_\mu$: Tabular Setting}\label{alg:PE_RL}
    %\vspace{0.4em}
    \BlankLine

    \KwIn{Time horizon $n > 0$, constants $c_1 > 0$, $c_2 > 0$, $c_3 > 0$ and step-size sequence $\lrset{\alpha_k}$.}
    \BlankLine

    {\bf Initialization: } $\bar{r}_0  = 0$, $Q_0 = \vec{\bf  0} $, $\bar{Q}_0 = 0$, and $\kappa_0 =0$.\\
    \BlankLine

\For{$k\leftarrow 1$ \KwTo $n$}{
\BlankLine

Observe ($S_k, A_k, r(S_k,A_k), S_{k+1}, A_{k+1}$).\;

\BlankLine

\tcp{\textbf{Average reward $J_\mu$ estimation}} 
        \vspace{-1.5em}
        \begin{align}
            \bar{r}_{k+1} = \bar{r}_k + c_1\alpha_k (  r(S_k,A_k) -\bar{r}_k ).\label{eq:JmuUpdateRL}
        \end{align}
\;
        \vspace{-1.5em}

\tcp{\textbf{$Q_\mu$ estimation}}
\vspace{-1.5em}
    \begin{align*}
    & \delta_k := r(S_k,A_k) - \bar{r}_{k} +  Q_k(S_{k+1},A_{k+1}) - Q_k(S_k,A_k), \label{eq:deltaUpdateRL}\numberthis\\
    & Q_{k+1}(s,a) = \numberthis\label{eq:QUpdateRL}
        \begin{cases}
            Q_k(s,a) + \alpha_k \delta_k \lrp{1-{1}/{|\calS||\calA|}},  &\text{ for } (s,a) = (S_k,A_k),\\
            Q_k(s,a) -  {\alpha_k\delta_k}/{|\calS||\calA|}, &\text{ otherwise}.
        \end{cases}  
    \end{align*}
\;
\vspace{-1.5em}
\tcp{\textbf{$\bar{Q}_\mu$ estimation}}
\vspace{-1.5em}
    \begin{align*}
    & \bar{Q}_{k+1} := \bar{Q}_k + c_2 \alpha_k ( Q_k(S_k,A_k) - \bar{Q}_k ), \label{eq:barVUpdateRL}\numberthis\\
    \end{align*}
\;
\vspace{-1.5em}

\tcp{\textbf{Asymptotic Variance $\kappa_\mu$ estimation}}
\vspace{-1.5em}
        \begin{align*}
        &\kappa_{k+1} =  (1-c_3\alpha_k) \kappa_k + c_3\alpha_k\bigg( 2r(S_k,A_k)Q_k(S_k,A_k) - 2r(S_k,A_k)\bar{Q}_k \\
        &\qquad\qquad\qquad\qquad\qquad\qquad\qquad\qquad - r^2(S_k,A_k) + r(S_k,A_k)\bar{r}_k \bigg). \numberthis\label{eq:kappaUpdateRL}
        \end{align*}
        \vspace{-1.5em}
    }
 \Output{Estimates $\kappa_n$, $\bar{Q}_n$, $Q_n$, $\bar{r}_n$} 
    
\end{algorithm2e}

\subsection{Convergence Rates}\label{sec:RL_bounds}
We now bound the estimation error of Algorithm~\ref{alg:PE_RL}. Let $Y_k:= (S_k, A_k,S_{k+1}, A_{k+1}).$ Observe that $\mathcal M_3 := \lrset{Y_k}_{k\ge 1}$ is a Markov chain that mixes geometrically fast \cite[Proposition 1]{bhatnagar2016multiscale}, and has a unique stationary  distribution. Let $\mathcal R \in \R^{|\calS||\calA|}$ be the vector of rewards, and $r_{\max} := \|\mathcal R\|_\infty$. Without loss of generality, for simplicity of presentation, we assume that $r_{\max} = 1$. Recall that $D_\mu$ denotes the diagonal matrix of the stationary distribution $d_\mu$ of $\calM_2$. Define 
\[ \tilde{\Delta}_1:= \min\!\lrset{\!{\bf v}^T \! D_\mu(I\!-\!P_2) {\bf v} ~|~ {\bf v} \!\in\! \R^{S}, \|{\bf v}\|_2 = 1, {\bf v}^T\!{\bf 1} = 0}, \numberthis\label{eq:tildedelta1}\]
where $P_2$ is the transition matrix for $\mathcal M_2$. It can be argued that $\tilde{\Delta}_1 > 0$ \cite[Lemma 7]{tsitsiklis1997average}. 

For the step-size constants $c_1$, $c_2$, and $c_3$ (inputs to Algorithm~\ref{alg:PE_RL}), define 
$$\eta  :=  \lrp{c^2_1  + 5 + 2c^2_2 + 10c^2_3}^\frac{1}{2}.$$ 
Let $Q^*_\mu$ be the Q Function orthogonal to ${\bf 1}$ in $\ell_2$-norm, $\bar{Q}^*_\mu := \Exp{d_\mu}{Q^*_\mu(S,A)}$, and define $\Theta_\mu := [J_\mu~Q^{*T}_\mu~\bar{Q}^*_\mu~\kappa_\mu].$ The following theorem bounds the estimation error of Algorithm~\ref{alg:PE_RL}, which immediately implies ${O}(\frac{1}{n})$ convergence rate for the mean-squared estimation error. 

\begin{theorem}\label{th:finitebound_RL}
    Consider Algorithm~\ref{alg:PE_RL} with $c_1$, $c_2$, and $c_3$ satisfying:
    \[c_1\ge \frac{1}{2\tilde{\Delta}_1} + \frac{\tilde{\Delta}_1}{2}, \qquad  \frac{5}{249}(5-2\sqrt{2})\tilde{\Delta}_1 \le c_3\le \frac{5}{249} (5+2\sqrt{2})\tilde{\Delta}_1, \]
    and
    \[ c_3 -\frac{498 c^2_3}{7\tilde{\Delta}_1} + \frac{7 \tilde{\Delta}_1}{498}   \le c_2 \le -3 c_3 + \frac{5 }{83} \sqrt{498 c_3 \tilde{\Delta}_1 - 17 \tilde{\Delta}^2_1} . \]
    Recall that $\eta$ and $\Theta_\mu$ were defined below~\eqref{eq:tildedelta1}. For a constant $B > 1$, let  $\xi_1 := 3(1+\|\Theta_\mu\|_2)^2$, $\xi_2 := 112B(1+\|\Theta_\mu\|_2)^2$, and let $\Theta_n:= [\bar{r}_n ~ Q^T_{n}~ \bar{Q}_n ~ \kappa_n]$ denote the vector of iterates of Algorithm~\ref{alg:PE_RL} at step $n$.
    \begin{enumerate}%[label=(\alph*)]
        \item[(a)] Let $\alpha_i = \alpha$ for all $i$, such that     
        \[ \alpha < \min\lrset{ \frac{40}{\tilde{\Delta}_1}, \frac{1}{28B\lrp{1+ \frac{20\eta^2}{\tilde{\Delta}_1} }}}. \]
        Then, for all $n\ge 1$, 
        \[ \E{\|\Theta_n - \Theta_\mu\|^2_2} \le \xi_1\lrp{1-\frac{\tilde{\Delta}_1\alpha}{40}}^{n} +  \frac{20\xi_2 \alpha\eta^2}{\tilde{\Delta}_1} + \xi_2\alpha. \]

        \item[(b)]  Let $\alpha_i = \frac{\alpha}{i+h}$ for all $i \ge 0$, with $\alpha$ and $h$ chosen so that
        \[ h \ge \max\lrset{1+\frac{\alpha\tilde{\Delta}_1}{40}, 1+28B\lrp{ \frac{20\eta^2\alpha}{\tilde{\Delta}_1} + \alpha + \frac{20}{\tilde{\Delta}_1} }},\quad\text{and}\quad \alpha > \frac{40}{\tilde{\Delta}_1}.\]
        Then, for all $n\ge 0$, 
        \begin{align*}
        \E{\|\Theta_n- \Theta_\mu\|^2_2}  &\le \xi_1\lrp{\frac{h}{n+h}}^\frac{\alpha\tilde{\Delta}_1}{40} + \frac{5\xi_2e^2\eta^2(20+\tilde{\Delta}_1)\alpha^2}{(n+h)(\alpha\tilde{\Delta}_1 -40)} + \frac{\xi_2 \alpha}{n+h}.
        \end{align*} 
    \end{enumerate}
\end{theorem}
As earlier, since $\E{(\kappa_n - \kappa_\mu)^2} \le \E{\|\Theta_n - \Theta_\mu\|^2_2}$, the bound in the Theorem above also hold for MSE in estimating $\kappa_\mu$. Theorem~\ref{th:finitebound_RL} immediately follows from Theorem~\ref{th:finitebound} by recalling that estimating $\Theta_\mu$ corresponds to the setup of Section~\ref{sec:varest:tabular}, with $r(\cdot,\cdot)$, $Q_\mu(\cdot,\cdot)$, $d_\mu$ replacing $f(\cdot)$, $V(\cdot)$, and $\pi$. Hence, we omit its proof. 

\begin{corollary}\label{cor:MSE_PE}
    Estimates from Algorithm~\ref{alg:PE_RL} with step size as in Theorem~\ref{th:finitebound_RL} $(b)$ satisfy $\E{\|Q_n - Q_\mu\|^2_2} = O(\frac{1}{n})$ MSE bound for policy evaluation problem in the average reward RL setting.
\end{corollary}

The above Corollary follows since $\E{\|Q_n - Q_\mu\|^2_2} \le \E{ \|\Theta_n - \Theta_\mu\|^2_2 }$. Hence, the bounds in Theorem~\ref{th:finitebound_RL} also hold for the MSE in estimating $Q_\mu$. Recall that the previously best-known rate for MSE for policy evaluation in average reward RL is $O(\frac{\log n}{n})$ \cite{zhang2021finite}. A tighter analysis for bounding the MSE of linear SA in presence of Markov noise, that uses Poisson Equations (again) in the proof, leads to this improvement by a multiplicative factor of $\log n$ in Corollary~\ref{cor:MSE_PE}. We refer the reader to Appendix~\ref{app:Step2} for a proof using this approach. 

Next, to estimate $\kappa_\mu$ or $Q_\mu$ up to MSE $\epsilon^2$ using Algorithm~\ref{alg:PE_RL} with diminishing step size, requires $n = O(\frac{1}{\epsilon^2})O(\frac{\|Q_\mu\|^2_2}{\tilde{\Delta}^4_1})$. This follows from inverting the above bound at $\epsilon^2$, as in Corollary~\ref{cor:samplecomplex}. This bound depends implicitly on $|S||A|$ via $\|Q_\mu\|_2$ and $\tilde{\Delta}_1$, which depends on the mixing properties of the underlying Markov chain. Observe that compared to \cite[Corollary 1]{zhang2021finite}, here, in addition to improving the log dependence of $\epsilon$  in the sample complexity bound, we improve the dependence on the mixing-time parameters. In particular, the bound in \cite{zhang2021finite} has an additional multiplicative term $K$ that relates to the mixing properties.  

\begin{remark}
    Unlike in the discounted RL setting, we point out that the operators in the average reward RL framework are not contractive under any norm. In fact, as in \cite{zhang2021finite}, we establish semi-norm contraction (Lemma~\ref{th:negdeflfa}) and convergence in an appropriate subspace (orthogonal to all $1$s vector). This semi-norm contraction is sufficient since the formulation of asymptotic variance in~\eqref{eq:kappaRL} is unaffected by constant shifts in the estimation of $Q_\mu$.
\end{remark}

\begin{remark}
    When the underlying state and action spaces are large, we can approximate $Q_\mu$ by its projection onto a lower dimensional linear subspace spanned by a given set of (fixed) vectors. The bounds obtained above can be extended to this setting using the ideas from Section~\ref{sec:lfa}. For completeness, we present the algorithm and performance bounds for estimating $\kappa_\mu$ with this approximation architecture in Appendix~\ref{sec:lfa_RL}.
\end{remark}

\section{Conclusions and Future Work}
We proposed a novel recursive estimator for the asymptotic variance of a Markov chain that requires $O(1)$ computation at each step, does not require storing historical samples, and enjoys the optimal $O(\frac{1}{n})$ rate of convergence for MSE. We use Stochastic Approximation to design the proposed estimator.  Using Lyapunov drift arguments, and Poisson equation to handle the Markov noise in analysis, we established finite sample bounds on the performance of the proposed estimator. We generalized it in several directions, including covariance matrix estimation for vector-valued functions of a Markov chain and estimating the asymptotic variance of a given stationary policy in an average reward RL framework. Our tighter analysis using Poisson equations to handle the Markov noise improves the MSE bound for linear SA in presence of Markov noise, in general, and for policy evaluation problem in RL, in particular. We established the ${O}(\frac{1}{n})$ rate of convergence of the proposed algorithm in the tabular and linear function approximation settings. We also characterized the approximation error in the latter setting. Below, we discuss a few directions for future research. 
\begin{itemize}
\item We discuss the dependence of the sample complexity in terms of the underlying Markov chain parameters such as the size of state space $S$ and mixing time  parameter $\Delta_1$ in Corollary~\ref{cor:samplecomplex}. Deriving lower bounds on the sample complexity in terms of these problem dependent parameters is an important direction that remains open. Tight lower bounds could potentially serve as a guide to designing estimators with improved dependence on these parameters. This will be useful for arriving at efficient estimators in settings where $S$ is large, for example, in modern RL applications.

\item Our estimator with diminishing step size, for example in Algorithm~\ref{alg:PE} with $\alpha_k = O(\frac{1}{k})$, uses multiple SA updates at each time, each with a step size differing from the other by only a constant ($c_1$, $c_2$, and $c_3$). One of the limitations of this so-called single-time scale approach is the constraints on step sizes $c_1$, $c_2$, and $c_3$ in order to achieve $O(\frac{1}{n})$ MSE bound. This is a well known limitation of single-time scale SA which can be overcome by using a multi-time scale SA. In multi-time scale SA, each SA update uses a different step size (in terms of its dependence on $k$ at time $k$). However, deriving finite sample guarantees for multi-time scale SA is a challenging direction. Designing an estimator for $\kappa(f)$ that uses multiple time scales, with finite sample guarantees on MSE is an interesting open direction for future research.

%Second, it would be interesting to design a multiple-time scale SA algorithm to get around the conditions on the step size constants that are necessary for the convergence of a linear SA algorithm. 

\item As we saw in Section~\ref{sec:lfa}, when the state space is large, one approach to estimating $\kappa(f)$ is to approximate the value function $V$ using a given approximation architecture (linear in our setting), and use this approximate value function to arrive at an estimate for an approximation of $\kappa(f)$, which we call $\kappa^*$. In general, this approach and our LFA results can extend easily for nice enough infinite state Markov chains. However, a big challenge is to get an exact estimation for $\kappa(f)$ in these settings. Even in finite but large state space settings, getting an exact estimation for $\kappa(f)$ remains open.

\item In Section~\ref{sec:lit}, we mentioned that the asymptotic behavior of regenerative simulation based estimators for $\kappa(f)$ has been extensively studied in literature. CLTs for these estimators have been proven, which can be used to establish an optimal asymptotic rate of $O(\frac{1}{n})$ in MSE. Investigating the finite sample behavior of these estimators and comparing its performance to our SA-based estimator is an interesting direction for future research.

\item On the RL front, an immediate future direction is to design an actor-critic style algorithm for identifying a policy that maximizes the long-run average reward subject to an asymptotic variance constraint with provable finite sample guarantees. Our policy-evaluation algorithm (Algorithm~\ref{alg:PE_RL} or Algorithm~\ref{alg:PE_lfa_RL}) can serve as a building block for designing the critic in these algorithms.
\end{itemize}

\section*{Acknowledgements}
This work was partially supported by NSF grants EPCN-2144316 and CPS-2240982. The authors thank the International Centre for Theoretical Sciences (ICTS), TIFR, Bangalore. This work is a result of a collaboration  that initaited during the meeting -- Data Science: Probabilistic and Optimization Methods  (code:ICTS/dspom2023/7) at ICTS. S. Agrawal also thanks Shaan Ul Haque for helpful discussions. 

\bibliography{BibTex.bib}
\bibliographystyle{plainnat}

\appendix
\input{journal_appendix}
\end{document}

%% file: journal_appendix.tex
\section{Finite Sample Bounds for Linear SA in Presence of Markov Noise}\label{app:Step2}
In this section, we prove a mean squared error (MSE) bound for iterates of a linear stochastic approximation (SA) update. We begin by describing the setup. For $d > 0$ let $\iter_0 = \vec{0} \in \R^d$,  and $\alpha_k$ denote the step size at time $k$ (to be specified later). Let $\mathbb S$ be a finite dimensional subspace of $\R^d$ ($\mathbb S$ can also be $\R^d$), and let $\iter_k \in \R^d$ be generated as 
\[ \iter_{k+1}  = \iter_k + \alpha_k  ( \A(Y_k)\iter_k +  \vb(Y_k)), \numberthis\label{app:linearSAUpdate} \]
where $(Y_i)_{i\in\N}$ are the state space of the underlying Markov chain with stationary distribution $\pi$, ${\A}(\cdot)$ and $\vb(\cdot)$ are the noisy update matrices of appropriate dimensions such that for $\iter_0 = \vec{0}$, all the iterates $\iter_k$ lie in  $\mathbb S$, for all $k\in \N$. Furthermore, let 
$$\A:= \Exp{\pi}{\A(Y)}\quad\text{ and } \quad \vb := \Exp{\pi}{\vb(Y)}$$ 
denote the average matrices, and let $\iter^* \in \mathbb S$ satisfy $\A\iter^* + \vb = 0$.

Recall that for $\mathbb{S} = \R^d$, \cite{srikant2019finite} develop a finite sample MSE bound for iterates $\iter_k$ when $\A$ is Hurwitz, and \cite{zhang2021finite} extend it for the setting where $\mathbb{S}$ is a proper  subspace of $\R^d$. In this section, we improve on both these results  by a multiplicative factor of $\log k$ by using Poisson equation for handling the error terms in our analysis. Use of Poisson equations for analyzing SA algorithms has previously appeared in \cite{benveniste2012adaptive}. More recently, \cite{kaledin2020finite}
 and \cite{haque2023tight} use Poisson equation to analyse a two-time scale algorithm in presence of Markov noise. 
 
\begin{theorem}\label{th:appfiniteboundlfa}
    Consider iterates ${\iter}_k$ generated by the update in~\eqref{app:linearSAUpdate} with $\iter_0  = \vec{0}$. Suppose there exists $\gamma_2 > 0$ such that the average matrix $\A$ satisfies
    \[ \min\limits_{\iter \in \mathbb S: \|\iter\|_2 = 1}~  - \iter^T \A ~ \iter \ge \gamma_2 > 0 .\numberthis\label{eq:contractfactor} \]
    Let $\upb := \max\{1, \|\A(\cdot)\|_2, \|\vb(\cdot)\|_2\}$, and for $B > 1$, let  $\psi_1 :=  3\lrp{1+\|\iter^*\|_2}^2$, $\psi_2 := 112B(1+\|\iter^*\|_2)^2$.
    \begin{enumerate}%[label=(\alph*)]
        \item[(a)] Let $\alpha_i = \alpha$ for all $i$, such that     
        \[ \alpha < \min\lrset{ \frac{2}{\gamma_2}, \frac{1}{28B\lrp{1+ \frac{\upb^2}{\gamma_2} }}}. \]
        Then, for all $k\ge 1$, 
        \[ \E{\|\iter_k - \iter^*\|^2_2} \le \psi_1\lrp{1-\frac{\gamma_2\alpha}{2}}^{k} +  \frac{\psi_2 \alpha\upb^2}{\gamma_2} + \psi_2\alpha. \]

        \item[(b)]  Let $\alpha_i = \frac{\alpha}{i+h}$ for all $i \ge 0$, with $\alpha$ and $h$ chosen so that
        \[ h \ge \max\lrset{1+\frac{\alpha\gamma_2}{2}, 1+28B\lrp{ \frac{\upb^2\alpha}{\gamma_2} + \alpha + \frac{1}{\gamma_2} }},\quad\text{and}\quad \alpha > \frac{2}{\gamma_2}.\]
        Then, for all $k\ge 0$, 
        \begin{align*}
        \E{\|\iter_k- \iter^*\|^2_2}  &\le \psi_1\lrp{\frac{h}{k+h}}^\frac{\alpha\gamma_2}{2} + \frac{5\psi_2e^2\upb^2(1+\gamma_2)\alpha^2}{(k+h)(\alpha\gamma_2 -2)} + \frac{\psi_2 \alpha}{k+h}.
        \end{align*} 
    \end{enumerate}
\end{theorem}

\begin{proof}
Since $\iter_k \in \mathbb S$ and $\iter^* \in \mathbb S$, the vector $\iter_k - \iter^*$ lies in $\mathbb S$. Now, consider the following (for simplicity of notation, we use $\A_k := \A(Y_k)$ and $\vb_k := \vb(Y_k)$ in this proof): 
\begin{align*}
    \|\iter_{k+1}-\iter^*\|^2_2 - \|\iter_k - \iter^*\|^2_2   &= \|\iter_{k+1}-\iter_k + \iter_k - \iter^*\|^2_2 - \|\iter_k - \iter^*\|^2_2 \\
    &= \|\iter_{k+1}-\iter_k\|^2_2 + 2 (\iter_{k+1} - \iter_k)^T(\iter_k - \iter^*)\\
    &= \alpha^2_k\|\A_k \iter_k + \vb_k\|^2_2 + 2 \alpha_k (\iter_k - \iter^*)^T(\A_k\iter_{k} + \vb_k)\\
    &= \alpha^2_k\|\A_k \iter_k + \vb_k\|^2_2 + 2 \alpha_k (\iter_k - \iter^*)^T(\A_k\iter_{k} - A \iter_k + \vb_k - \vb)\\
        &\qquad\qquad\qquad+ 2 \alpha_k (\iter_k - \iter^*)^T(A\iter_k + \vb)\\
    &= \alpha^2_k\|\A_k \iter_k + \vb_k\|^2_2 + 2 \alpha_k (\iter_k - \iter^*)^T((\A_k- A) \iter_k + \vb_k - \vb)\\
        &\qquad\qquad\qquad + 2 \alpha_k (\iter_k - \iter^*)^TA(\iter_k - \iter^*)\\
    &\le \alpha^2_k\|\A_k \iter_k + \vb_k\|^2_2 + 2 \alpha_k (\iter_k - \iter^*)^T((\A_k- A) \iter_k + \vb_k - \vb) \tag{From~\eqref{eq:contractfactor}}\\
    &\qquad\qquad\qquad - \gamma_2  \alpha_k \|\iter_k - \iter^*\|^2_2. 
\end{align*}
 Next, recall that $\upb > 1$ is such that 
\[ \max\lrset{1,\|\A_k\|_2} \le \upb \qquad \text{ and }\qquad \max\lrset{1,\|\vb_k\|_2} \le \upb. \] 

Then, 
\begin{align*} 
\|\A_k \iter_k + \vb_k\|^2_2 \le ( \|\A_k \iter_k\|_2 + \|\vb_k\|_2 )^2 &\le \upb^2(\|\iter_k - \iter^*\|_2 +  1 + \|\iter^*\| )^2 \\ 
&\le 2 \upb^2\lrp{ \|\iter_k - \iter^*\|^2_2 + (1 + \|\iter^*\|_2)^2 }. 
\end{align*}

Using this, 
\begin{align*}
    \| \iter_{k+1} - \iter^* \|^2_2 - \|\iter_k - \iter^*\|^2_2 & \le \|\iter_k - \iter^*\|^2_2( -\gamma_2 \alpha_k +2\alpha^2_k \upb^2 ) + 2 \alpha^2_k\upb^2(1+\|\iter^*\|_2)^2\\
    &\qquad\qquad\qquad + 2 \alpha_k (\iter_k - \iter^*)^T((\A_k- A) \iter_k + \vb_k - \vb).\numberthis\label{eq:boundPE1}
\end{align*}

\vspace{0.75em}
\noindent{\bf Using Poisson equation. } Since the stationary mean of $\A_k = \A(Y_k)$ is $\A$, and that for  $\vb_k = \vb(Y_k)$ is $\vb$, for each state $s \in  \calS$, there exists a matrix $V_1(s)$ and a vector $V_2(s)$, that are solutions of the following Poisson equations:
\[ \A(s)-\A = V_1(s) - [PV_1](s), \quad \text{ and }\quad \vb-\vb(s) = V_2(s) - [PV_2](s). \]
Furthermore, for each $s\in\calS$, $V_x(s) := V_1(s) x - V_2(s)$ is a vector that solves 
\[ (\A(s) - \A) x + \vb(s) - \vb = V_x(s) - [PV_x](s).  \]
Moreover, there exists a constant $B:= \max\lrset{\max_s \|V_1(s)\|_2, \max_s \|V_2(s)\|_2}$ such that for each $s$ and $x$, 
$$\|V_x(s)\|_2 \le B(\|x\|_2+1),$$ 
and $V_x(s)$ is a Lipschitz function of $x$, i.e., 
\[ \|V_x(s) - V_y(s)\|_2 = \| V_1(s) (x - y) \|_2 \le \|V_1(s)\|_2 \|x-y\|_2 \le B\|x-y\|_2.  \]

Using these in the cross-term in~\eqref{eq:boundPE1},
\begin{align*}
    2 \alpha_k (\iter_k - \iter^*)^T((\A_k-\A)\iter_k + \vb_k - \vb) &=  2 \alpha_k (\iter_k - \iter^*)^T\lrp{(\A(Y_k) - \A)\iter_k + \vb(Y_k) - \vb}\\
    &= 2 \alpha_k (\iter_k - \iter^*)^T\lrp{V_{\iter_k}(Y_k) - \E{ V_{\iter_k}(Y_{k+1}) | Y_k }}\\
    &=2 \alpha_k (\iter_k - \iter^*)^T\big(V_{\iter_k}(Y_k) - \E{ V_{\iter_k}(Y_{k}) | Y_{k-1} } \\
        &\qquad+ \E{ V_{\iter_k}(Y_{k}) | Y_{k-1} }- \E{ V_{\iter_k}(Y_{k+1}) | Y_k }\big)\\
    &= 2 \alpha_k (\iter_k - \iter^*)^T\lrp{V_{\iter_k}(Y_k) - \E{ V_{\iter_k}(Y_{k}) | Y_{k-1} }} \\
        &\qquad + 2 \alpha_k (\iter_k - \iter^*)^T  \E{ V_{\iter_k}(Y_{k}) | Y_{k-1} }\\
        &\qquad - 2 \alpha_k (\iter_k - \iter^*)^T  \E{ V_{\iter_k}(Y_{k+1}) | Y_k }.
\end{align*}

Substituting back in~\eqref{eq:boundPE1}, 
\begin{align*}
    \| \iter_{k+1} - \iter^* \|^2_2 - \|\iter_k - \iter^*\|^2_2
    &\le \|\iter_k - \iter^*\|^2_2( -\gamma_2 \alpha_k +2\alpha^2_k \upb^2 ) + 2 \alpha^2_k\upb^2(1+\|\iter^*\|_2)^2 \\
    & \quad + \underbrace{ 2 \alpha_k (\iter_k - \iter^*)^T\lrp{V_{\iter_k}(Y_k) - \E{ V_{\iter_k}(Y_{k}) | Y_{k-1} }} }_{0-\text{mean term}}  \\
    &\quad+ \underbrace{ 2 \alpha_k (\iter_k - \iter^*)^T \lrp{ \E{ V_{\iter_k}(Y_{k}) | Y_{k-1} }- \E{ V_{\iter_k}(Y_{k+1}) | Y_k }}}_{T_1}. \numberthis\label{eq:boundPE2}
\end{align*}

Next, we show that $T_1$ is a small order. To this end, let 
\[ d_k := (\iter_k - \iter^*)^T \E{V_{\iter_k}(Y_k)| Y_{k-1}}. \]

Then, 
\begin{align*}
T_1 
&= 2\alpha_k d_k - 2 \alpha_k (\iter_k - \iter^*)^T \E{ V_{\iter_k}(Y_{k+1}) | Y_k} - 2\alpha_k d_{k+1} + 2\alpha_k d_{k+1}\\
&= 2\alpha_k (d_k-d_{k+1}) + 2\alpha_k (\iter_{k+1}-\iter^*)^T \E{V_{\iter_{k+1}}(Y_{k+1})|Y_k} \\
    &\qquad\qquad -  2\alpha_k(\iter_k - \iter^*)^T \E{ V_{\iter_k}(Y_{k+1}) | Y_k} \\
&= 2\alpha_k (d_k-d_{k+1}) + 2\alpha_k (\iter_{k+1}-\iter^*)^T \lrp{\E{V_{\iter_{k+1}}(Y_{k+1})|Y_k} - \E{V_{\iter_k}(Y_{k+1})| Y_{k}}} \\
&\qquad\qquad + 2\alpha_k(\iter_{k+1}-\iter_k )^T \E{ V_{\iter_k}(Y_{k+1}) | Y_k} \\
&\le 2\alpha_k (d_k-d_{k+1}) + 2\alpha_k \|\iter_{k+1}-\iter^*\|_2\E{\|V_{\iter_{k+1}}(Y_{k+1}) - V_{\iter_k}(Y_{k+1}) \|_2 | Y_{k}} \tag{Cauchy-Schwarz}\\
&\qquad\qquad + 2\alpha_k\|\iter_{k+1}-\iter_k \|_2 \E{ \|V_{\iter_k}(Y_{k+1}) \|_2| Y_k} \\
&\le 2\alpha_k (d_k-d_{k+1}) + 2B\alpha_k \|\iter_{k+1}-\iter^*\|_2\E{ \| \iter_{k+1}-\iter_k \|_2 | Y_{k}} \tag{Lipschitzness}\\
&\qquad\qquad + 2B\alpha_k\|\iter_{k+1}-\iter_k \|_2 \E{ (\|{\iter_k}\|_2 + 1) | Y_k}  \\
&\le 2\alpha_k (d_k-d_{k+1}) + 2B\alpha^2_k \|\iter_{k+1}-\iter_k\|_2\|\A_k \iter_{k} + \vb_k\|_2 \tag{Triangle inequality} \\ 
&\qquad\qquad + 2B\alpha^2_k \|\iter_{k}-\iter^*\|_2\|\A_k \iter_{k} + \vb_k \|_2 \\
&\qquad\qquad+ 2B\alpha^2_k\|\A_k\iter_k + \vb_k \|_2  \|{\iter_k} - \iter^*\|_2 + 2B\alpha^2_k\|\A_k\iter_k + \vb_k \|_2(\|\iter^*\|_2+1)  \\
&= 2\alpha_k (d_k-d_{k+1}) \!+\! 2B\alpha^3_k \|\A_k \iter_{k} \!+\! \vb_k\|^2_2 \!+\! 4B\alpha^2_k \|\iter_{k}-\iter^*\|_2\|\A_k \iter_{k} + \vb_k \|_2 \\
    &\qquad\qquad+ 2B\alpha^2_k\|\A_k\iter_k + \vb_k \|_2(1+\|\iter^*\|_2).
\end{align*}

Next, observe that 
\[ \|\A_k \iter_k + \vb_k\|^2_2 \le 2 \upb^2 (1+\|\iter^*\|_2)^2 + 2 \upb^2 \|\iter_k - \iter^*\|^2_2. \]

Using this, as well as the inequality 
\[ 2ab \le a^2 + b^2 \]
to further bound the cross terms in $T_1$, we get

\begin{align*}
    T_1 
    &\le  2\alpha_k (d_k-d_{k+1}) + 4B\upb^2\alpha^3_k (1+\|\iter^*\|_2)^2 + 4B\upb^2\alpha^3_k\|\iter_k - \iter^*\|^2_2\\
        & \qquad  + 2 B \alpha^2_k \|\iter_k - \iter^*\|^2_2 + 4 B \alpha^2_k \upb^2 (1+\|\iter^*\|_2)^2 + 4 B \alpha^2_k \upb^2 \|\iter_k -\iter^*\|^2_2 \\
        &\qquad + B \alpha^2_k (1+\|\iter^*\|_2)^2 + 2 B \alpha^2_k \upb^2 (1+\|\iter^*\|_2)^2 + 2 B \alpha^2_k \upb^2 \|\iter_k - \iter^*\|^2_2\\
    &=  2\alpha_k (d_k-d_{k+1}) + 4B\upb^2\alpha^3_k (1+\|\iter^*\|_2)^2 + 4B\upb^2\alpha^3_k\|\iter_k - \iter^*\|^2_2\\
        & \qquad  + 2 B \alpha^2_k \|\iter_k - \iter^*\|^2_2 + 6 B \alpha^2_k \upb^2 (1+\|\iter^*\|_2)^2 + 6 B \alpha^2_k \upb^2 \|\iter_k -\iter^*\|^2_2 \\
        &\qquad + B \alpha^2_k (1+\|\iter^*\|_2)^2 \\ 
    &\le 2\alpha_k (d_k-d_{k+1}) + 4B\upb^2\alpha^3_k (1+\|\iter^*\|_2)^2 + 4B\upb^2\alpha^3_k\|\iter_k - \iter^*\|^2_2\\
        & \qquad  + 8 B \alpha^2_k \upb^2 \|\iter_k - \iter^*\|^2_2 + 8 B \alpha^2_k \upb^2 (1+\|\iter^*\|_2)^2 \tag{Since $\upb \ge 1$}\\ 
    &\le 2\alpha_k (d_k-d_{k+1}) + 12 B\upb^2\alpha^2_k (1+\|\iter^*\|_2)^2 + 12 B\upb^2\alpha^2_k\|\iter_k - \iter^*\|^2_2\tag{Since $\alpha^3_k \le \alpha^2_k$}.
\end{align*}

Substituting this bound on $T_1$ back in~\eqref{eq:boundPE2}, we have
\begin{align*}
    &\| \iter_{k+1} - \iter^* \|^2_2 - \|\iter_k - \iter^*\|^2_2 \\
    &\le \|\iter_k - \iter^*\|^2_2( -\gamma_2 \alpha_k +2\alpha^2_k \upb^2) + 2 \alpha^2_k \upb^2(1+\|\iter^*\|_2)^2 \\
    &\qquad+ 2 \alpha_k (\iter_k - \iter^*)^T\lrp{V_{\iter_k}(Y_k) - \E{ V_{\iter_k}(Y_{k}) | Y_{k-1} }}\\
    & \qquad + \underbrace{2\alpha_k (d_k-d_{k+1}) }_{T_2} + 12 B \alpha^2_k \upb^2 \|\iter_k - \iter^*\|^2_2 + 12 B \alpha^2_k \upb^2 (1+ \|\iter^*\|_2)^2 \tag{Since $B \ge 1$} \\
    &\le \|\iter_k - \iter^*\|^2_2( -\gamma_2 \alpha_k + 14 B \alpha^2_k \upb^2) + 14 B \alpha^2_k \upb^2(1+\|\iter^*\|_2)^2 \\
    &\qquad+ 2 \alpha_k (\iter_k - \iter^*)^T\lrp{V_{\iter_k}(Y_k) - \E{ V_{\iter_k}(Y_{k}) | Y_{k-1} }}\\
    & \qquad + \underbrace{2\alpha_k (d_k-d_{k+1}) }_{T_2}.\numberthis\label{eq:boundPE3}
\end{align*}

Now, observe that 
\begin{align*}
    |d_k| 
    &= \abs{ (\iter_k - \iter^*)^T \E{V_{\iter_k}(Y_k)| Y_{k-1}}} \\
    &\le \|\iter_k - \iter^*\|_2 \E{\|V_{\iter_k}(Y_k)\|_2| Y_{k-1}} \tag{Cauchy-Schwarz and Jensen's}\\
    &\le B\|\iter_k - \iter^*\|_2 (\|\iter_k - \iter^*\|_2 + 1 + \|\iter^*\|_2) \tag{Lipschitzness}\\
    &= B\|\iter_k - \iter^*\|^2_2  + B \|\iter_k - \iter^*\|_2 (1+\|\iter^*\|_2)\numberthis\label{eq:bounddk1}\\
    &\le 2B\|\iter_k - \iter^*\|^2_2 + 2B(1+\|\iter^*\|_2)^2, \numberthis\label{eq:bounddk}
\end{align*}
where, for the last inequality, we used $2ab \le a^2 + b^2$.

Let us now get an upper bound on $T_2$ involving terms $2\alpha_{k-1} d_k - 2\alpha_k d_{k+1}$, instead. 
\begin{align*}
    T_2 
    &= 2 \alpha_k (d_k - d_{k+1})\\
    &= 2 \alpha_{k-1} d_k - 2 \alpha_k d_{k+1} + 2 (\alpha_k - \alpha_{k-1})d_k\\
    &\le 2 \alpha_{k-1}d_k - 2\alpha_k d_{k+1} + 2 |\alpha_k - \alpha_{k-1}| |d_k|\\
    &\le 2 \alpha_{k-1} d_k - 2\alpha_k d_{k+1} + \frac{6\alpha^2_k}{\alpha} |d_k| \tag{From Lemma~\ref{lem:diffinstepsize}}\\
    &= 2\alpha_{k-1}\lrp{1-\frac{\alpha_k\gamma_2}{2}} d_k - 2\alpha_k d_{k+1} + \gamma_2\alpha_k\alpha_{k-1}  d_k + \frac{6\alpha^2_k}{\alpha} |d_k|\\ 
    &\le 2\alpha_{k-1}\lrp{1-\frac{\alpha_k\gamma_2}{2}} d_k - 2\alpha_k d_{k+1} + 2\gamma_2\alpha^2_k|d_k| + \frac{6\alpha^2_k}{\alpha} |d_k|\tag{From  Lemma~\ref{lem:diffalpha}}\\
    &\le 2\alpha_{k-1}\lrp{1-\frac{\alpha_k\gamma_2}{2}} d_k - 2\alpha_k d_{k+1} \\
    &\qquad\qquad+ 4B\alpha^2_k\lrp{ \gamma_2+\frac{3}{\alpha}}\|\iter_k - \iter^*\|^2_2  + 4B\alpha^2_k\lrp{\gamma_2 + \frac{3}{\alpha}} (1+\|\iter^*\|_2)^2.\tag{Using~\eqref{eq:bounddk}}
\end{align*}

\begin{remark}\label{rem:const}
    Note that a more careful bound for constant step size setting would ensure that the term involving $\frac{3}{\alpha}$ is absent. This corresponds to the bound from Lemma~\ref{lem:diffinstepsize}, which is $0$ in that setting.
\end{remark}

Substituting this back in~\eqref{eq:boundPE3}, and noting that $2\alpha_k \le 2$, we have 
\begin{align*}
    \| \iter_{k+1} - \iter^* \|^2_2 - \|\iter_k - \iter^*\|^2_2 &\le \|\iter_k - \iter^*\|^2_2\lrp{-\gamma_2 \alpha_k + \alpha^2_k\lrp{14B\upb^2 + 4B\lrp{\gamma_2 + \frac{3}{\alpha}}}} \\
    &+ \alpha^2_k(1+\|\iter^*\|_2)^2 \lrp{ 14\upb^2 B + 4 B \lrp{\gamma_2 + \frac{3}{\alpha}} }\\
    &+ 2 \alpha_k (\iter_k - \iter^*)^T (V_{\iter_k}(Y_k) - \E{V_{\iter_k}(Y_k) | Y_{k-1}})\\
    &+ 2 \alpha_{k-1} \lrp{1-\frac{\alpha_k \gamma_2}{2}} d_k - 2\alpha_k d_{k+1}.
\end{align*}

Taking conditional expectation, conditioned on randomness till time $k$, 
\begin{align*}
    \E{\| \iter_{k+1} - \iter^* \|^2_2 | \iter_k, Y_{k-1}} &\le \|\iter_k - \iter^*\|^2_2\lrp{1 -\gamma_2 \alpha_k + \alpha^2_k(14B\upb^2 + 4B\lrp{\gamma_2 + \frac{3}{\alpha}})} \\
    &+ \alpha^2_k(1+\|\iter^*\|_2)^2 \lrp{ 14\upb^2 B + 4 B \lrp{\gamma_2 + \frac{3}{\alpha}} }\\
    &+ 2 \alpha_{k-1} \lrp{1-\frac{\alpha_k \gamma_2}{2}} d_k - 2\alpha_k d_{k+1},
\end{align*}
which further gives 
\begin{align*}
    \E{\| \iter_{k+1} - \iter^* \|^2_2 } &\le \E{\|\iter_k - \iter^*\|^2_2}\lrp{1 -\gamma_2 \alpha_k + \alpha^2_k\lrp{14B\upb^2 + 4B\lrp{\gamma_2 + \frac{3}{\alpha}}}} \\
    & \qquad + \alpha^2_k(1+\|\iter^*\|_2)^2 \lrp{ 14\upb^2 B + 4 B \lrp{\gamma_2 + \frac{3}{\alpha}} }\\
    &\qquad + 2 \alpha_{k-1} \lrp{1-\frac{\alpha_k \gamma_2}{2}} \E{d_k} - 2\alpha_k \E{d_{k+1}}\\
    &\le \E{\|\iter_k - \iter^*\|^2_2}\lrp{1 -\gamma_2 \alpha_k + 14B\alpha^2_k\lrp{\upb^2 + \gamma_2 + \frac{3}{\alpha}}} \\
    &\qquad + 14B\alpha^2_k(1+\|\iter^*\|_2)^2 \lrp{ \upb^2 + \gamma_2 + \frac{3}{\alpha}} \\
    & \qquad + 2 \alpha_{k-1} \lrp{1-\frac{\alpha_k \gamma_2}{2}} \E{d_k} - 2\alpha_k \E{d_{k+1}}.
\end{align*}

Choose $\alpha_k$ small enough so that 
\[14B\alpha_k\lrp{\upb^2 + \gamma_2 + \frac{3}{\alpha}} \le \frac{\gamma_2}{2}, \tag{Condition 1}\label{eq:condition1}\]
so that 
\begin{align*}
    \E{\| \iter_{k+1} - \iter^* \|^2_2}&\le \E{\|\iter_k - \iter^*\|^2_2}\lrp{1 -\frac{\gamma_2 \alpha_k}{2}} \\
    &\qquad + 2\alpha_{k-1} \lrp{1-\frac{\alpha_k \gamma_2}{2}}\E{d_k}-2\alpha_k\E{d_{k+1}}\\
    &\qquad + 14B\alpha^2_k(1+\|\iter^*\|_2)^2\lrp{\upb^2 + \gamma_2 + \frac{3}{\alpha}}.
\end{align*}
Let 
\[ B_1 := 14B(1+\|\iter^*\|_2)^2\lrp{\upb^2 + \gamma_2 + \frac{3}{\alpha}}.\tag{Constant $B_1$}\label{eq:constantB1}\]
Then, the above recursion becomes 
\begin{align*}
    \E{\| \iter_{k+1} - \iter^* \|^2_2}&\le \E{\|\iter_k - \iter^*\|^2_2}\lrp{1 -\frac{\alpha_k\gamma_2 }{2}} \\ 
        &\qquad \qquad + 2\alpha_{k-1} \lrp{1-\frac{\alpha_k \gamma_2}{2}}\E{d_k}-2\alpha_k\E{d_{k+1}} + \alpha^2_k B_1. 
\end{align*}

On opening the recursion above, we get
\begin{align*}
    \E{\| \iter_{k+1} - \iter^* \|^2_2}
    &\le \lrp{1 -\frac{\alpha_k\gamma_2 }{2}}\lrp{1 -\frac{\alpha_{k-1}\gamma_2 }{2}}\E{\|\iter_{k-1} - \iter^*\|^2_2} \\ 
    &\qquad +B_1\lrp{\alpha^2_k + \lrp{1 -\frac{\alpha_k\gamma_2 }{2}} \alpha^2_{k-1}}\\
    &\qquad + 2\alpha_{k-2}\lrp{1 -\frac{\alpha_k\gamma_2 }{2}} \lrp{1-\frac{\alpha_{k-1} \gamma_2}{2}}\E{d_{k-1}}-2\alpha_k\E{d_{k+1}}\\
    &\le \E{\|\iter_{0} - \iter^*\|^2_2}\prod\limits_{i=0}^{k}\lrp{1 -\frac{\alpha_i\gamma_2 }{2}}+B_1\sum\limits_{i=0}^k \alpha^2_i \prod\limits_{j=i+1}^k \lrp{1 -\frac{\alpha_j\gamma_2 }{2}}\\
    &\qquad + 2\alpha_{-1}\prod\limits_{i=0}^k\lrp{1 -\frac{\alpha_i\gamma_2 }{2}} \E{d_{0}}-2\alpha_k\E{d_{k+1}}.
\end{align*}
Next, upper-bounding the last term, 
\begin{align*}
    \E{\| \iter_{k+1} - \iter^* \|^2_2} &\le \E{\|\iter_{0} - \iter^*\|^2_2}\prod\limits_{i=0}^{k}\lrp{1 -\frac{\alpha_i\gamma_2 }{2}}+B_1\sum\limits_{i=0}^k \alpha^2_i \prod\limits_{j=i+1}^k \lrp{1 -\frac{\alpha_j\gamma_2 }{2}}\\
    &\qquad + 2\alpha_{-1}\prod\limits_{i=0}^k\lrp{1 -\frac{\alpha_i\gamma_2 }{2}} \E{d_{0}}+2\alpha_k\E{|d_{k+1}|}\\
    &\le \E{\|\iter_{0} - \iter^*\|^2_2}\prod\limits_{i=0}^{k}\lrp{1 -\frac{\alpha_i\gamma_2 }{2}}+B_1\sum\limits_{i=0}^k \alpha^2_i \prod\limits_{j=i+1}^k \lrp{1 -\frac{\alpha_j\gamma_2 }{2}}\tag{From~\eqref{eq:bounddk}}\\
    &~~~+ 2\alpha_{-1}\prod\limits_{i=0}^k\lrp{1 -\frac{\alpha_i\gamma_2 }{2}} \E{d_{0}}+ 4\alpha_k B \E{\|\iter_{k+1}-\iter^*\|^2_2} \\ 
    &\qquad + 4 \alpha_k B (1+\|\iter^*\|_2)^2.
\end{align*}

On rearranging, we get
\begin{align*}
    \E{\| \iter_{k+1} - \iter^* \|^2_2} &\lrp{1-4\alpha_kB }\\
    &\le \E{\|\iter_{0} - \iter^*\|^2_2}\prod\limits_{i=0}^{k}\lrp{1 -\frac{\alpha_i\gamma_2 }{2}}+B_1\sum\limits_{i=0}^k \alpha^2_i \prod\limits_{j=i+1}^k \lrp{1 -\frac{\alpha_j\gamma_2 }{2}}\\
    &\qquad+ 2\alpha_{-1}\prod\limits_{i=0}^k\lrp{1 -\frac{\alpha_i\gamma_2 }{2}} \E{d_{0}}+ 4B\alpha_k(1+\|\iter^*\|_2)^2\\
    &\le \E{\|\iter_{0} - \iter^*\|^2_2}\prod\limits_{i=0}^{k}\lrp{1 -\frac{\alpha_i\gamma_2 }{2}}+B_1\sum\limits_{i=0}^k \alpha^2_i \prod\limits_{j=i+1}^k \lrp{1 -\frac{\alpha_j\gamma_2 }{2}}\\
    &\qquad+ 4\alpha_{-1}B(1+\|\iter^*\|_2)^2\prod\limits_{i=0}^k\lrp{1 -\frac{\alpha_i\gamma_2 }{2}} + 4B\alpha_k(1+\|\iter^*\|_2)^2 \tag{Using~\eqref{eq:bounddk1}}. 
\end{align*}

Choosing $\alpha_k$ small enough so that 
\[ 1-4\alpha_kB \ge \frac{1}{2}, \tag{Condition 2}\label{eq:condition2} \]
we get
\begin{align*}
    \E{\| \iter_{k+1} - \iter^* \|^2_2}&\le 2\E{\|\iter_{0} - \iter^*\|^2_2}\prod\limits_{i=0}^{k}\lrp{1 -\frac{\alpha_i\gamma_2 }{2}}+2B_1\sum\limits_{j=0}^k \alpha^2_j \prod\limits_{i=j+1}^k \lrp{1 -\frac{\alpha_i\gamma_2 }{2}}\\
    &\qquad\qquad+ 8\alpha_{-1}B(1+\|\iter^*\|_2)^2\prod\limits_{i=0}^k\lrp{1 -\frac{\alpha_i\gamma_2 }{2}} + 8B\alpha_k(1+\|\iter^*\|_2)^2. \numberthis\label{eq:finalrecursion}
\end{align*}

We now make the above bound more explicit in two cases: diminishing step size and constant step size. 

\subsection*{Diminishing Step Size}
In this subsection, we consider step sizes of the form $\alpha_k = \frac{\alpha}{k+h}$, for appropriate choices for $\alpha$ and $h\ge 2$. Let us bound each term in the  recursion in~\eqref{eq:finalrecursion} separately. First observe that for any $i_0\ge 0$,
\begin{align*}
    \prod\limits_{j=i_0}^k\lrp{ 1-\frac{\gamma_2\alpha_j}{2} } 
    &= \prod\limits_{j=i_0}^k\lrp{ 1-\frac{\gamma_2\alpha}{2(j+h)} }\\
    &\le \prod\limits_{j=i_0}^k e^{ -\frac{\gamma_2\alpha}{2(j+h)} }\\
    &=  e^{ -\frac{\gamma_2\alpha}{2}\sum\limits_{j=i_0}^k\frac{1}{j+h}}\\
    &\le e^{ -\frac{\gamma_2\alpha}{2}\ln\frac{k+h}{i_0+ h}}\\
    &= \lrp{\frac{i_0+h}{k+h}}^{\frac{\alpha\gamma_2}{2}}. \numberthis\label{eq:boundonprodofnegdrift}
\end{align*}

Using this in~\eqref{eq:finalrecursion} with $i_0 = 0$, we get
\begin{align*}
    \E{\| \iter_{k+1} - \iter^* \|^2_2}&\le 2\E{\|\iter_{0} - \iter^*\|^2_2}\lrp{\frac{h}{k+h}}^\frac{\alpha\gamma_2}{2}+\underbrace{2B_1\sum\limits_{j=0}^k \alpha^2_j \prod\limits_{i=j+1}^k \lrp{1 -\frac{\alpha_i\gamma_2 }{2}}}_{:=T^{(d)}_1}\\
    &\qquad\qquad+ 8\alpha_{-1}B(1+\|\iter^*\|_2)^2 \lrp{\frac{h}{k+h}}^{\frac{\alpha\gamma_2}{2}} + 8B\alpha_k(1+\|\iter^*\|_2)^2. 
\end{align*}

Let us now bound $T^{(d)}_1$. 
\begin{align*}
    T^{(d)}_1 
    &= 2B_1 \sum\limits_{j=0}^k \prod_{i=j+1}^k \lrp{ 1-\frac{\gamma_2 \alpha_i}{2} }\alpha^2_j \\
    &\le 2B_1 \sum\limits_{j=0}^k \lrp{\frac{j+1+h}{k+h}}^{\frac{\alpha\gamma_2}{2}}\frac{\alpha^2}{(j+h)^2} \tag{From~\eqref{eq:boundonprodofnegdrift}}\\
    &= \frac{2B_1\alpha^2}{(k+h)^{\frac{\alpha\gamma_2}{2}}} \sum\limits_{j=0}^k \lrp{\frac{j+1+h}{j+h}}^{2}\lrp{j+1+h}^{\frac{\alpha\gamma_2}{2}-2}\\
    &\le \frac{8B_1\alpha^2}{(k+h)^{\frac{\alpha\gamma_2}{2}}} \sum\limits_{j=0}^k \lrp{j+1+h}^{\frac{\alpha\gamma_2}{2}-2}. \tag{Since $h\ge 2$}
\end{align*}

Choose $\alpha$ large enough so that
\begin{equation} \frac{\alpha\gamma_2}{2} > 1, \label{eq:condition3}\tag{Condition 3}
\end{equation}
and hence, 
\begin{align*}
    \sum\limits_{j=0}^k (j+1+h)^{\frac{\gamma_2\alpha}{2}-2}&\le\int\limits_{0}^{k+1} (x+1+h)^{\frac{\gamma_2\alpha}{2}-2}dx\\
    &=\frac{1}{\frac{\gamma_2\alpha}{2}-1} \lrp{ (k+2+h)^{\frac{\gamma_2\alpha}{2}-1} - (1+h)^{\frac{\gamma_2\alpha}{2} - 1}}\\
    &\le \frac{2(k+2+h)^{\frac{\gamma_2\alpha}{2}-1}}{\gamma_2\alpha-2}.
\end{align*}

Using this,
\begin{align*} 
    T^{(d)}_1 
    &\le \lrp{\frac{8B_1\alpha^2}{(k+h)^{\frac{\alpha\gamma_2}{2}}}}\frac{2(k+2+h)^{\frac{\gamma_2\alpha}{2}-1}}{\gamma_2\alpha-2}\\
    &= \frac{16B_1\alpha^2}{\gamma_2\alpha-2} \lrp{\frac{1}{k+2+h}}\lrp{\frac{k+2+h}{k+h}}^{\frac{\alpha\gamma_2}{2}}\\
    &\le \frac{16B_1\alpha^2}{\gamma_2\alpha-2} \lrp{\frac{e^{\frac{\alpha\gamma_2}{k+h}}}{k+2+h}} \tag{$1+x \le e^x$}\\
    &= \frac{16B_1\alpha^2}{\gamma_2\alpha-2} \lrp{\frac{e^{\alpha_k\gamma_2}}{k+2+h}}.
\end{align*}
Choosing $h$ large enough to ensure 
\[\alpha_k\gamma_2 \le 2, \label{eq:condition4}\tag{Condition 4} \] 
we have
\[ T^{(d)}_1 \le \frac{16B_1\alpha^2e^2}{(\gamma_2\alpha -2)(k+h+1)}.\]

Using this bound, the MSE bound becomes: 
\begin{align*}
    \E{\| \iter_{k+1} - \iter^* \|^2_2}&\le 2\E{\|\iter_{0} - \iter^*\|^2_2}\lrp{\frac{h}{k+h}}^\frac{\alpha\gamma_2}{2}+\frac{16B_1\alpha^2e^2}{(\gamma_2\alpha -2)(k+h+1)}\\
        &\qquad\qquad+ 8\alpha_{-1}B(1+\|\iter^*\|_2)^2 \lrp{\frac{h}{k+h}}^{\frac{\alpha\gamma_2}{2}} + 8B\alpha_k(1+\|\iter^*\|_2)^2\\
    &\le 2(1+\|\iter^*\|_2)^2\lrp{1+4\alpha_{-1}B}\lrp{\frac{h}{k+h}}^\frac{\alpha\gamma_2}{2} \\
        &\qquad\qquad+\frac{16B_1\alpha^2e^2}{(\gamma_2\alpha -2)(k+h)} +  8B\alpha_{k}(1+\|\iter^*\|_2)^2\\
    &= 2(1+\|\iter^*\|_2)^2\lrp{1+4\alpha_{-1}B}\lrp{\frac{h}{k+h}}^\frac{\alpha\gamma_2}{2} \numberthis\label{eq:semifbounddiminishing} \\
        &\qquad\qquad+ \frac{2\alpha}{k+h} \lrp{\frac{8B_1\alpha e^2}{(\gamma_2\alpha -2)} +   4B(1+\|\iter^*\|_2)^2}.
\end{align*}

\vspace{0.75em}
\noindent{\bf Conditions on step size parameters. } Recall that $h\ge 2$. In addition, ~\eqref{eq:condition1} is satisfied if 
\[ h \ge 1 + 28 B\lrp{\frac{\upb^2\alpha}{\gamma_2} + \alpha + \frac{3}{\gamma_2} }, \]
which also automatically satisfies~\eqref{eq:condition2}.~\eqref{eq:condition4} is  satisfied for 
\[h\ge 1 + \frac{\alpha \gamma_2}{2}.\] 
Moreover, ~\eqref{eq:condition3} requires that 
\[ \alpha > \frac{2}{\gamma_2}. \]

\vspace{0.75em}
\noindent{\bf Final conditions. } Thus, 
\[ h \ge \max\lrset{2, 1 + \frac{\alpha\gamma_2}{2},
1 + 28 B\lrp{ \frac{\upb^2\alpha}{\gamma_2} + \alpha + \frac{1}{\gamma_2} }},\quad\text{and}\quad \alpha > \frac{2}{\gamma_2}. \numberthis\label{eq:conditions_dimss}\]
Recall from~\ref{eq:constantB1},
\begin{align*}
    B_1 = 14B(1+\|\iter^*\|_2)^2\lrp{\upb^2 + \gamma_2 + \frac{3}{\alpha}}
\end{align*}

\vspace{0.75em}
\noindent{\bf Final bound. }Using this in~\eqref{eq:semifbounddiminishing}, we get
\begin{align*}
    \E{ \|\iter_{k+1} - \iter^*\|^2_2 } &\le 2 (1+\|\iter^*\|_2)^2(1+4\alpha_{-1}B)\lrp{\frac{h}{k+h}}^{\frac{\alpha\gamma_2}{2}} \\
    &\qquad + \frac{8B\alpha(1+\|\iter^*\|_2)^2}{k+h}\lrp{ \frac{28 \alpha e^2}{\gamma_2 \alpha - 2}\lrp{\upb^2 + \gamma_2 + \frac{3}{\alpha}} + 1 }\\
    &\le 2 (1+\|\iter^*\|_2)^2(1+4\alpha_{-1}B)\lrp{\frac{h}{k+h}}^{\frac{\alpha\gamma_2}{2}} \\
    &\qquad + \frac{112B\alpha(1+\|\iter^*\|_2)^2}{k+h}\lrp{ \frac{2 \alpha e^2}{\gamma_2 \alpha - 2}\lrp{\upb^2 + \gamma_2 + \frac{3}{\alpha}} + 1 }.
\end{align*}

Further, from~\eqref{eq:condition2}, we have for all $k$, 
\[ \alpha_{-1} \le \frac{1}{8B}.  \]

Hence, the bound becomes
\begin{align*}
    \E{ \|\iter_{k+1} - \iter^*\|^2_2 } & \le 3 (1+\|\iter^*\|_2)^2\lrp{\frac{h}{k+h}}^{\frac{\alpha\gamma_2}{2}} \\
    &\qquad + \frac{112B\alpha(1+\|\iter^*\|_2)^2}{k+h}\lrp{ \frac{2 \alpha e^2}{\gamma_2 \alpha - 2}\lrp{\upb^2 + \gamma_2 + \frac{3}{\alpha}} + 1 }.
\end{align*}

Now, finally, since $\alpha > \frac{2}{\gamma_2}$, we have $\frac{1}{\alpha} < \frac{\gamma_2}{2}$. Using this in the last term above, 
\begin{align*}
    \E{ \|\iter_{k+1} - \iter^*\|^2_2 } & \le 3 (1+\|\iter^*\|_2)^2\lrp{\frac{h}{k+h}}^{\frac{\alpha\gamma_2}{2}} \\
    &\qquad + \frac{112B\alpha(1+\|\iter^*\|_2)^2}{k+h}\lrp{ \frac{2 \alpha e^2}{\gamma_2 \alpha - 2}\lrp{\upb^2 + \frac{5\gamma_2}{2}} + 1 }\\
    &\le 3 (1+\|\iter^*\|_2)^2\lrp{\frac{h}{k+h}}^{\frac{\alpha\gamma_2}{2}} \\
    &\qquad + \frac{112B\alpha(1+\|\iter^*\|_2)^2}{k+h}\lrp{ \frac{5 \alpha e^2 \upb^2}{\gamma_2 \alpha - 2}\lrp{1 + \gamma_2} + 1 }
\end{align*}

\subsection*{Constant Step Size}
In this section, we bound the recursion in~\eqref{eq:finalrecursion} for the setting of constant step size, i.e., $\alpha_k = \alpha$. In this setting, the recursion in~\eqref{eq:finalrecursion} becomes

\begin{align*}
    \E{\| \iter_{k+1} - \iter^* \|^2_2}&\le 2\|\iter^*\|^2_2\lrp{1 -\frac{\alpha\gamma_2 }{2}}^{k+1}+2B_1\alpha^2\sum\limits_{j=0}^k  \lrp{1 -\frac{\alpha\gamma_2 }{2}}^{j}\\
    &\qquad\qquad+ 8\alpha_{-1} B(1+\|\iter^*\|_2)^2\lrp{1 -\frac{\alpha\gamma_2 }{2}}^{k+1} + 8B\alpha(1+\|\iter^*\|_2)^2\\
    &=2(1+\|\iter^*\|_2)^2\lrp{1+4\alpha_{-1} B}\lrp{1 -\frac{\alpha\gamma_2 }{2}}^{k+1} \numberthis\label{eq:semifconstant} \\
        &\qquad\qquad + 4\alpha\lrp{\frac{B_1}{\gamma_2}+   2B(1+\|\iter^*\|_2)^2}.
\end{align*}

\vspace{0.75em}
\noindent{\bf Conditions on step size. } Using the observation in Remark~\ref{rem:const}, and choosing 
\[ \alpha \le \frac{1}{28B\lrp{1+\frac{\upb^2}{\gamma_2}}}, \]
ensures~\eqref{eq:condition1}, which also automatically satisfies~\eqref{eq:condition2}. Moreover, $\alpha$ also needs to satisfy 
\[ \alpha  < \frac{2}{\gamma_2}. \]

\vspace{0.75em}
\noindent{\bf Final conditions. } Thus, 
\[ \alpha < \min\lrset{ \frac{2}{\gamma_2}, \frac{1}{28B\lrp{1+\frac{\upb^2}{\gamma_2}}} }. \]

Under these conditions, and accounting for the observation from Remark~\ref{rem:const}, recall from~\eqref{eq:constantB1} that
\begin{align*}
    B_1 := 14B\gamma_2(1+\|\iter^*\|_2)^2\lrp{1+\frac{\upb^2}{\gamma_2}}.
\end{align*}

\vspace{0.75em}
\noindent{\bf Final bound. }Substituting this in~\eqref{eq:semifconstant}, along with~\eqref{eq:condition2}, which ensures that $\alpha \le \frac{1}{8B}$, we have
\begin{align*}
    \E{\| \iter_{k+1} - \iter^* \|^2_2} &\le 2(1+\|\iter^*\|_2)^2\lrp{1+4\alpha B}\lrp{1 -\frac{\alpha\gamma_2 }{2}}^{k+1} \\
    &\qquad \qquad + 56B\alpha(1+\|\iter^*\|_2)^2\lrp{\frac{\upb^2}{\gamma_2} + 2}\\
    &\le 3(1+\|\iter^*\|_2)^2\lrp{1 -\frac{\alpha\gamma_2 }{2}}^{k+1} + 112B\alpha(1+\|\iter^*\|_2)^2\lrp{\frac{\upb^2}{\gamma_2} + 1}.
\end{align*}
\end{proof}

\section{Proofs for Results in Section~\ref{sec:var}}\label{app:sec3}

\subsection{Proof of Proposition~\ref{prop:var}}\label{app:proof_prop:var}
From Lemma~\ref{lem:asympvar_pi0}, we get that the asymptotic variance $\kappa(f)$ is constant independent of the starting state or starting distribution, and is given by the expression in~\eqref{eq:kmu0}. Further, from Lemma~\ref{lem:asymp_var}, we get the first two formulations for $\kappa(f)$ (Equations~\eqref{eq:kmu1} and~\eqref{eq:kmu2}). Finally, the third formulation follows by replacing $(f(X) - \bar{f})$ terms in~\eqref{eq:kmu2} by $V(X) - P V (X)$, which follows from the Poisson equation~\eqref{eq:PE}.\qed

\subsection{Variance of Functions of a Discrete Time Markov Chain (DTMC)}\label{app:var_dtmc}
    Recall that $X = (X_k: k \ge 0)$ is an irreducible and aperiodic DTMC on a finite-state space $\calS$ with transition probability matrix  $P$ and a unique stationary distribution $\pi$. Let $f: \calS \rightarrow \R$ be any function, and let $\bar{f} := \sum_{x\in\calS} \pi(x) f(x)$ denote its stationary expectation. Let $V$ denote the solution to the Poisson equation for $f-\bar{f}$, i.e., for each $x\in\calS$, $V$ satisfies 
        $$V(x) - \sum\nolimits_{x'\in\calS}P(x,x')V(x') = f(x)-\bar{f}.$$

\begin{lemma}\label{lem:asympvar_pi0}
    Consider the Markov Chain starting in distribution $\pi_0$. The asymptotic variance of $f$, $\kappa(f)$, is a constant independent of the starting distribution of the Markov Chain $\pi_0$, i.e., 
    \begin{align*}
        \kappa(f) := \lim\limits_{n\rightarrow\infty}\frac{1}{n} ~ \V\lrs{ \sum\limits_{k=0}^{n-1} f(X_k) | X_0 \sim \pi_0 } = \lim\limits_{n\rightarrow\infty}\frac{1}{n} ~ \V\lrs{ \sum\limits_{k=0}^{n-1} f(X_k) | X_0 \sim \pi}.
    \end{align*}
\end{lemma}
\begin{proof}
    For $k\in \N_+$, let $\bar{f}_k := \Exp{\pi_0}{f(X_k)} = \E{f(X_k) | X_0 \sim \pi_0}$. Then,
    \begin{align*}
        \kappa(f)
        &= \frac{1}{n} ~ \Exp{\pi_0}{ \lrp{ \sum\limits_{k=0}^{n-1} \lrp{ f(X_k) - \bar{f}_k }}^2 }\\
        &= \frac{1}{n} ~ \sum\limits_{k=0}^{n-1}\Exp{\pi_0}{ (f(X_k) - \bar{f}_k)^2 } + \frac{2}{n} \sum\limits_{k=0}^{n-1}\sum\limits_{i < k} \Exp{\pi_0}{\lrp{f(X_i)-\bar{f}_i}\lrp{f(X_k)-\bar{f}_k} }. \numberthis\label{eqn:var_pi0}
    \end{align*}
    Clearly, 
    \[ \lim\limits_{n\rightarrow\infty} \frac{1}{n} ~ \sum\limits_{k=0}^{n-1}\Exp{\pi_0}{ (f(X_k) - \bar{f}_k)^2  } = \lim\limits_{n\rightarrow \infty} \frac{1}{n} ~ \sum\limits_{t=0}^{n-1}\Exp{\pi}{ (f(X_t)- \bar{f})^2}. \numberthis\label{eq:conv1} \]
    This follows from the following:
        \begin{align*}
        &\lim\limits_{n\rightarrow\infty}\frac{1}{n} ~ \sum\limits_{k=0}^{n-1}\Exp{\pi_0}{ (f(X_k) - \bar{f}_k)^2  } 
        \\ 
        &= \lim\limits_{n\rightarrow\infty} ~ \frac{1}{n} ~ \sum\limits_{k=0}^{n-1}\Exp{\pi_0}{ (f(X_k) - \bar{f})^2}- \lim\limits_{n\rightarrow\infty}~ \frac{1}{n} ~ \sum\limits_{k=0}^{n-1} (\bar{f} - \bar{f}_k)^2\\
        &= \lim\limits_{n\rightarrow\infty} ~ \frac{1}{n} ~ \sum\limits_{k=0}^{n-1}\Exp{\pi}{ (f(X_k) - \bar{f})^2},
    \end{align*}
    where we used that 
    \[ \lim\limits_{n\rightarrow \infty} \frac{1}{n} \sum\limits_{k=0}^{n-1} (\bar{f} - \bar{f}_k)^2 = 0, \]
    which follows from geometric mixing of the underlying Markov chain.
    
    Let us now show that the second term in~\eqref{eqn:var_pi0} converges to the right limit. Towards this, we first re-write it as 
    \begin{align*}
       \mathcal B = \frac{2}{n} ~ \sum\limits_{k=0}^{n-1}\sum\limits_{i<k} \lrp{\Exp{\pi_0}{ (f(X_i) - \bar{f})(f(X_k) - \bar{f}) } + (\bar{f}_i - \bar{f})(\bar{f} - \bar{f}_k)},
    \end{align*}
    and let the corresponding term under $\pi$ be 
    \[ \mathcal B'= \frac{2}{n} ~ \sum\limits_{k=0}^{n-1}\sum\limits_{i<k} \Exp{\pi}{ (f(X_i) - \bar{f})(f(X_k) - \bar{f}) } . \]
    Let 
    \[T_1 := \Exp{\pi_0}{ (f(X_i) - \bar{f})(f(X_k) - \bar{f}) }\quad \text{ and } \quad T_2 := \Exp{\pi}{ (f(X_i) - \bar{f})(f(X_k) - \bar{f}) }.\]
    Then,
    \begin{align*}
        \abs{\mathcal B - \mathcal B'} 
        &= \abs{ \frac{2}{n} ~ \sum\limits_{k=0}^{n-1}\sum\limits_{i < k} (T_1 - T_2) + \frac{2}{n} ~ \sum\limits_{k=0}^{n-1} \sum\limits_{i < k} (\bar{f}_i - \bar{f})(\bar{f} - \bar{f}_k) }\\
        &\le  \frac{2}{n} ~ \sum\limits_{k=0}^{n-1}\sum\limits_{i < k} \abs{T_1 - T_2} + \frac{2}{n} ~ \sum\limits_{k=0}^{n-1} \sum\limits_{i < k} \abs{\bar{f}_i - \bar{f}}\abs{\bar{f} - \bar{f}_k }\\
        &\overset{(a)}{\le} \frac{2}{n} ~ \sum\limits_{k=0}^{n-1}\sum\limits_{i < k} \abs{T_1 - T_2} + \frac{2}{n} ~ \sum\limits_{k=0}^{n-1} \sum\limits_{i < k} 4f^2_{\max}C\alpha^k,\numberthis\label{eq:bound1}
    \end{align*}    
    where we used Lemma~\ref{lem:geometric_conv_bound} to get the bound in $(a)$ above, and recall that $f_{\max}:= \max_x f(x)$. Clearly, the second term converges to $0$ as $n\rightarrow \infty$. Let us now bound the first term.
    
    \begin{align*}
        \abs{ T_1 - T_2 }
        &=\abs{ \sum\limits_{x\in\calS}\sum\limits_{y\in\calS}\sum\limits_{z\in\calS} \lrp{\pi_0(x) - \pi(x)}P^i(x,y)P^{k-i}(y,z) (f(y)-\bar{f})(f(z) - \bar{f})}\\
        &\le 4f^2_{\max}\sum\limits_{z\in\calS}\abs{\sum\limits_{x\in\calS}\sum\limits_{y\in\calS} \lrp{\pi_0(x) - \pi(x)}P^i(x,y)P^{k-i}(y,z)} \\
        &=4f^2_{\max}\sum\limits_{z\in\calS}\abs{\sum\limits_{x\in\calS} \lrp{\pi_0(x)P^k(x,z) - \pi(x)P^k(x,z)}} \\
        &=\sum\limits_{z\in\calS}\abs{\sum\limits_{x\in\calS} \pi_0(x)P^k(x,z) - \pi(z)} f^2_{\max}\\
        &\le \sum\limits_{z\in\calS}\sum\limits_{x\in\calS}\pi_0(x)\abs{P^k(x,z) - \pi(z)} f^2_{\max}\\
        &= 2L^2\sum\limits_{x\in\calS}\pi_0(x)\|P^k(x,\cdot) - \pi\|_{\text{T.V.}} \\
        &\le 2f^2_{\max} C\alpha^k.
    \end{align*}
    Using the above bound in~\eqref{eq:bound1}, we get that $\lim\limits_{n\rightarrow\infty} \abs{\mathcal{B} - \mathcal{B}'} \le  0$, proving the desired result.
\end{proof}

\begin{lemma}\label{lem:asymp_var}
    Consider the Markov Chain starting in distribution $\pi_0$. For $j\in\N$ and 
         $\gamma(k) := \Exp{\pi}{ \lrp{ f(X_k) - \bar{f}} \lrp{ f(X_0) - \bar{f} }  },$ 
    the asymptotic variance of $f$, $\kappa(f)$, is given by
    \begin{align*}
        \kappa(f) 
        &= \gamma(0) + 2\lim\limits_{n\rightarrow\infty} \sum\limits_{k=1}^{n-1}\gamma(k)=2 \sum\limits_{x\in\calS} \pi(x)(f(x)-\bar{f})V(x) - \sum\limits_{x\in\calS}\pi(x)(f(x)-\bar{f})^2. 
    \end{align*}
\end{lemma}

\begin{proof}
    From Lemma~\ref{lem:asympvar_pi0}, we have
    \begin{align*}
        \kappa(f):=\lim\limits_{n\rightarrow \infty}\frac{1}{n} \V\lrs{ \sum\limits_{k=0}^{n-1} f(X_k) | X_0 \sim \pi_0 }        &= \lim\limits_{n\rightarrow \infty}\frac{1}{n} \V\lrs{ \sum\limits_{k=0}^{n-1} f(X_k) | X_0 \sim \pi }.
    \end{align*}
    We show that the r.h.s. in the above equation is same as the expression in the lemma. Recall that for any $k\ge 1$,  $\E{ f(X_k) | X_0 \sim \pi } = \bar{f} $. For simplicity of notation, we denote $\Exp{ X_0 \sim \pi_0 }{\cdot}$ by $\Exp{\pi_0}{\cdot}$. Then,
    \begin{align*}
        \V\lrs{ \sum\limits_{k=0}^{n-1} f(X_k) | X_0 \sim \pi } 
        &= \Exp{\pi}{\lrp{\sum\limits_{k=0}^{n-1} \lrp{f(X_k) - \bar{f}}}^2 }\\
        &= \Exp{\pi}{ \sum\limits_{k=0}^{n-1}(f(X_k) - \bar{f})^2 } \\
            &\qquad\qquad + 2\sum\limits_{k=0}^{n-1}\sum\limits_{i < k } \Exp{\pi}{ \lrp{f(X_i)- \bar{f}} \lrp{ f(X_k) - \bar{f} } }\\
        &\overset{(a)}{=} \Exp{\pi}{ \sum\limits_{k=0}^{n-1}(f(X_k) - \bar{f})^2 } \\
            &\qquad\qquad + 2\sum\limits_{k=1}^{n-1} (n-k) \Exp{\pi}{ \lrp{f(X_k)- \bar{f}} \lrp{ f(X_0) - \bar{f} } }\\
        &= n \gamma(0)  + 2\sum\limits_{k=1}^{n-1} (n-k) \gamma(k), \numberthis\label{eq:var_n}
    \end{align*}
    where $(a)$ follows from Markov property. Dividing by $n$ and taking limits, we get
    \begin{align*}
        \lim\limits_{n\rightarrow \infty } \frac{1}{n} \V\lrs{ \sum\limits_{k=0}^{n-1} f(X_k) | X_0 \sim \pi } &= \gamma(0) + \lim\limits_{n\rightarrow \infty}2\sum\limits_{k=1}^{n-1} \frac{n-k}{n} \gamma(k) \\
        &\overset{(a)}{=} \gamma(0) + \lim\limits_{n\rightarrow \infty}2\sum\limits_{k=1}^{n-1}  \gamma(k) \numberthis\label{eq:exp1} \\
        &= \lim\limits_{n\rightarrow \infty} 2\sum\limits_{k=0}^{n-1} \frac{n-k}{n} \gamma(k) - \gamma(0)\\
        &\overset{(b)}{=} \lim\limits_{n\rightarrow\infty} 2 \sum\limits_{k=0}^{n-1}\gamma(k) - \gamma(0),\numberthis\label{eq:asymp_var_form_3}
    \end{align*}
    where $(a)$ and $(b)$ follow from Lemma~\ref{lem:simplifygamma}. Expression in~\eqref{eq:exp1} corresponds to the first expression in the lemma. Next, we show that the expression in~\eqref{eq:asymp_var_form_3} is same as the second expression in the lemma. 
    
    Since $V$ is the solution of the Poisson equation, we have the following form for $V$ \cite[Proposition 21.2.3, Lemma 21.2.2]{douc2018markov}
    \[ V(x) = \sum\limits_{k=0}^\infty \E{f(X_k)-\bar{f} | X_0 = x} + c, \]
    for any constant $c$. Substituting this in the expression in the lemma statement gives:  
    \begin{align*}
        &2 \sum\limits_{x\in\calS} \pi(x)(f(x)-\bar{f})V(x) - \gamma(0) \\
        &= 2 \sum\limits_{x\in\calS} \pi(x)(f(x)-\bar{f})\lrp{\sum\limits_{k=0}^\infty \E{f(X_k)-\bar{f} | X_0 = x}}   - \gamma(0)\\
        & \overset{(a)}{=} 2 \sum\limits_{k=0}^\infty \Exp{\pi}{(f(X_0)-\bar{f})\lrp{f(X_k)-\bar{f} }}  - \gamma(0)\\
        &= 2 \lim\limits_{n\rightarrow \infty} \sum\limits_{k=1}^{n-1} \gamma(k) - \gamma(0),
    \end{align*}
    which is same as~\eqref{eq:asymp_var_form_3}. Here, to change the limits and expectation in $(a)$, we used the bounded convergence theorem since $f$ is bounded (underlying Markov chain is on a finite state space).
\end{proof}

\begin{lemma}\label{lem:simplifygamma}
    For $k\in\N$ and  $\gamma(k)$ below,
    \[ \lim\limits_{n\rightarrow \infty} \sum\limits_{k=0}^{n-1} \frac{n-k}{n}\gamma(k) = \lim\limits_{n\rightarrow\infty} \sum\limits_{k=0}^{n-1} \gamma(k),\]  where  
    \[\gamma(k) := \Exp{\pi}{ (f(X_k) - \bar{f})(f(X_0) - \bar{f}) }. \]
\end{lemma}
\begin{proof}
    The proof of this lemma follows along the lines of the proof of \cite[Lemma 3]{mou2020heavy}. However, we give the proof for completeness. 
    Let $f_{\max} := \max\limits_{s\in\calS} f(x)$. Since $X$ is an irreducible and aperiodic Markov chain on a finite state space, it mixes geometrically fast \cite{levin2017markov}, i.e., there exist constants $C > 0$ and $\alpha\in (0,1)$ such that for all $k\in\N_+$, 
    \[ \sup\limits_{x\in\calX}~ d_{TV}(P( X_k |X_0 = x), \pi) \le C\alpha^k, \]
    where $d_{TV}(\cdot,\cdot)$ represents the total variation distance between the two input distributions. Below, we first show that $\gamma(k)$ is bounded. 
    \begin{align*}
        \abs{\gamma(k)} 
        &= \abs{ \Exp{\pi}{ \lrp{f(X_k) - \bar{f}}\lrp{ f(X_0) - \bar{f} } } }\\
        &= \abs{ \Exp{\pi}{ \E{ \lrp{f(X_k) - \bar{f}}\lrp{ f(X_0) - \bar{f} } }  | X_0 = x } }\\
        &= \abs{ \Exp{x\sim\pi}{ \E{ \lrp{f(X_k) - \bar{f}}| X_0 = x }\lrp{ f(x) - \bar{f} } } }\\
        &\le  \Exp{x\sim\pi}{ \abs{\E{ \lrp{f(X_k) - \bar{f}}| X_0 = x }}\abs{ f(x) - \bar{f} } } \\
        &\le 4f^2_{\max}C\alpha^k, \numberthis\label{eq:boundedgamma}
    \end{align*}
    where we used Lemma~\ref{lem:geometric_conv_bound} to get the last inequality above. 

    Next, define
    \[ V_1 := \lim \limits_{n\rightarrow \infty} \sum\limits_{k=0}^{n-1} \gamma(k), \quad\text{ and }\quad V_2 := \lim\limits_{n\rightarrow \infty} \sum\limits_{k=0}^{n-1} \frac{n-k}{n} \gamma(k). \]
    Consider the following:
    \begin{align*}
        \abs{ V_1 - V_2 } 
        &=  \lim\limits_{n\rightarrow\infty} \abs{ \sum\limits_{k=0}^{n-1} \frac{k}{n}\gamma(k) }\le \lim\limits_{n\rightarrow\infty} \sum\limits_{k=0}^{n-1}\frac{k}{n}\abs{\gamma(k)}\overset{(a)}{\le} 4f^2_{\max}C\lim\limits_{n\rightarrow\infty} \frac{1}{n}\sum\limits_{k=1}^{n-1} k\alpha^k\le 0,
    \end{align*}
    where in $(a)$ we used Lemma~\ref{lem:geometric_conv_bound}, and the last inequality follows since since $\alpha < 1$. This implies $V_1 = V_2$.
\end{proof} 

\begin{lemma}\label{lem:geometric_conv_bound}
    Let $X = (X_k: k\ge 0)$ be DTMC on a finite-state space $\calS$ with a unique stationary distribution $\pi$. Let $f:\calS\rightarrow\R$ be any function, and let $\bar{f}:= \sum_{x\in\calS} \pi(x)f(x)$ denote its stationary expectation. Then, for any $k\in\N_+$, there exist constants $\alpha\in (0,1)$ and $C > 0$, such that for any initial distribution $X_0\sim \pi_0$, and $k > j$,
    \[ \abs{ \E{ f(X_k) - \bar{f}  | X_0 \sim \pi_0} } \le 2f_{\max}C\alpha^k,\numberthis\label{eq:boundfromstatfirst} \]
    and
    \[\abs{ \E{ (f(X_k) - \bar{f})(f(X_j) - \bar{f})  | X_0 \sim \pi_0} } \le 4f^2_{\max} C^2\alpha^{k-j}, \numberthis\label{eq:boundfromstat}\]
    where $f_{\max}:= \max_{x \in \calS} f(x)$. 
    \begin{comment}
    Further, for $j\ge 1$, let $\bar{f}_j := \E{f(X_j) | X_0 \sim \pi_0}$. Then for $k > j$, there exists a constant $C' > 1$ such that
    \[ \operatorname{Cov}_{\pi_0}(f(X_k), f(X_j)) := \E{ (f(X_k) - \bar{f}_k) (f(X_j) - \bar{f}_j) } \le f^2_{\max} C'\alpha^{k-j}, \numberthis\label{eq:cov} \]
    where $\operatorname{Cov}_{\pi_0}(.,.)$ denotes the covariance when the Markov chain is started with initial distribution $\pi_0$. Additionally, for $k >  j$, there exists a constant $C_1 > 1$ such that 
    \[ \abs{\Exp{\pi_0}{ (f(X_k) - \bar{f})(f(X_j) - \bar{f}) }  - \Exp{\pi}{ (f(X_k) - \bar{f})(f(X_j) - \bar{f}) }} \le f^2_{\max}C_1 \alpha^k. \numberthis\label{eq:diffcov} \]
\end{comment}

\end{lemma}

\begin{proof}
    This lemma follows from the geometric mixing of the underlying Markov chain. Parts of this lemma are borrowed from~\cite{mou2020heavy}. 
    
\noindent{\bf Proving~\eqref{eq:boundfromstatfirst}. } Consider the following inequalities: 
    \begin{align*}
         \abs{ \E{ f(X_k) - \bar{f}  | X_0 \sim \pi_0} } 
         &= \abs{ \sum\limits_{x'\in\calS}\pi_0(x') \sum\limits_{x\in\calS} f(x)\lrp{ \mathbb{P}\lrp{ X_k = x | X_0 = x' }  - \pi(x) } } \\
         &\le \sum\limits_{x'\in\calS}\pi_0(x') \sum\limits_{x\in\calS} \abs{ f(x)}\abs{ \mathbb{P}\lrp{ X_k = x | X_0 = x' }  - \pi(x) } \\
        &\le f_{\max}\sum\limits_{x'\in\calS}\pi_0(x') \sum\limits_{x\in\calS}  \abs{ \mathbb{P}\lrp{ X_k = x  | X_0 = x' }  - \pi(x) } \\
        &\le f_{\max}~ \sup\limits_{x'\in\calS} ~\sum\limits_{x\in\calS}  \abs{ \mathbb{P}\lrp{ X_k = x | X_0 = x' }  - \pi(x) }\\
        &\le 2 f_{\max} C\alpha^{k},
    \end{align*}
    where the last inequality follows from the definition of total variation distance and geometric mixing of the underlying Markov chain.

    \noindent{\bf Proving~\eqref{eq:boundfromstat}. } We now prove the other inequality using the geometric mixing of the Markov chain. Let $\Exp{\pi_0}{\cdot}$ represent $\E{\cdot | X_0 \sim \pi_0}$.
    \begin{align*}
         \abs{ \Exp{\pi_0}{ (f(X_k) - \bar{f})(f(X_j) - \bar{f})  }} 
         &=\abs{ \Exp{\pi_0}{ (f(X_j) - \bar{f}) \E{(f(X_k) - \bar{f}) | X_j}}}\\
         &\le  \Exp{\pi_0}{ |f(X_j) - \bar{f}| } \abs{ \sup\limits_{x\in\calS} \E{(f(X_k) - \bar{f}) | X_j=x} }. \numberthis\label{eq:int_dist_stat_1}
    \end{align*}
    Since~\eqref{eq:boundfromstatfirst} holds for any initial distribution $\pi_0$, we have 
    \begin{align*}
    \abs{\sup_{x} \E{ (f(X_k) - \bar{f}) | X_j = x}} = \sup_{x} \E{ (f(X_{k-j}) - \bar{f}) | X_0 = x} \le 2f_{\max} C \alpha^{k-j}.
    \end{align*}
Using the above bound in~\eqref{eq:int_dist_stat_1}, 
    \begin{align*}
        \abs{ \Exp{\pi_0}{ (f(X_k) - \bar{f})(f(X_j) - \bar{f})  }} &\le 2f_{\max} C \alpha^{k-j}   \Exp{\pi_0}{ \abs{(f(X_j) - \bar{f}) }}  \le  4f^2_{\max} C \alpha^{k-j},
    \end{align*}
proving the desired result.

\end{proof}

\subsection{Discussion on Batched Mean Estimators}\label{app:BM}
A natural first approach to estimate $\kappa(f)$ is to use the Monte Carlo (MC) technique. For this, observe that \[\kappa =  \lim\limits_{m\rightarrow \infty} \kappa_m, \quad \text{where} \quad \kappa_m:= \V[Z_m | X_0 \sim \pi], \quad \text{and} \quad Z_m := \frac{1}{\sqrt{m}}\sum_{k=0}^{m-1} f(X_k). \numberthis\label{eq:ZTmain}\]
A vanilla MC estimator, popularly known as Batched Means (BM) estimator,  proceeds by generating a single large trajectory of the underlying Markov chain $\mathcal M$ starting from any initial distribution $\pi_0$ on $\calS$. It then considers samples in non-overlapping buckets of a fixed length, say $m$. For $i\in \N$, let $$Z^{(i)}_m:=  \frac{1}{\sqrt{m}} \sum\limits_{k= (i-1)m}^{i(m-1)} f(X_k).$$
BM method estimates the first and second moments of $Z_m$, defined in~\eqref{eq:ZTmain}, using sample averages of $Z^{(i)}_m$ across $b$ buckets, say, and combines them to get an estimate for $\kappa_m$. %Call this $\hat{\kappa}^M_T$. 

The asymptotic convergence of BM (as $b\rightarrow\infty$) to $\kappa_m$ follows from the Ergodic Theorem for Markov chains \cite[Theorem 5.2.1]{douc2018markov}. However, a drawback of BM is that approximating the limit in $\kappa$ by the variance of a truncated random  variable $Z_m$ introduces a bias, which corresponds to $\abs{\kappa - \kappa_m}$. \cite{chien1997large} established that the mean-squared estimation error for estimating $\kappa$ using BM is $O(\frac{1}{m^2} + \frac{1}{b})$, where the $O(\frac{1}{m^2})$ term corresponds to the bias.  Thus, estimating $\kappa$ with mean-squared estimation error at most $\epsilon^2$ requires $m = O(\frac{1}{\epsilon})$, and $b= O(\frac{1}{\epsilon^2})$. This corresponds to a total of $mb = O(\frac{1}{\epsilon^3})$ samples. While this bound can be improved by designing better MC estimators, see \cite{flegal2010batch}, MC estimators typically also suffer from being non-adaptive (or non-recursive), i.e., $m$ is a function of $\epsilon$ - this necessitates restarting the procedure for different choices of $\epsilon$.

\section{Proof of Results in Section~\ref{sec:varest:tabular}}

\subsection{Proof of Theorem~\ref{th:finitebound}}\label{app:proof_th:finitebound}
For $s\in\calS$, let $\ind(s) \in \lrset{0,1}^{\abs{\calS}}$ be the indicator for $s$, i.e., it equals $1$ at the $s^\text{th}$ coordinate, and $0$ otherwise. Furthermore, for $n \in \lrset{1,2, \dots}$, define $\ind_n  := \ind(S_n)$, indicator for $S_n$. Then, setting $d=\abs{\calS}$,  basis vectors in the setup in Section~\ref{sec:lfa} as standard basis, and for $k\in\lrset{1,2,\dots}$, $\phi_k = \ind_k$, we recover Algorithm~\ref{alg:PE} from updates in Section~\ref{sec:alglfa}. 

Further, since in this setting, the set of fixed points $V$ of~\eqref{eq:PE} lies in the span of the basis vectors, $\kappa^*$ from~\eqref{eq:kappastar} is same as $\kappa(f)$. Hence, guarantees in Theorem~\ref{th:finiteboundlfa} with these adjustments, reduce to those in Theorem~\ref{th:finitebound}, proving the result. 

\subsection{Proof of Corollary~\ref{cor:samplecomplex}}\label{app:samplecomplex}
We prove the sample complexity bound separately for the algorithm with a constant step size, and that with the diminishing step size. But first, observe that $c_1 = O({1}/{\Delta_1})$, $c_2 = O(\Delta_1)$, and $c_3 = O(\Delta_1)$, implying that $\eta = O(1/\Delta_1)$. Furthermore, $\xi_1 = O(\|V^*\|^2_2)$, and $\xi_2 = O(\|V^*\|^2_2)$.

\begin{enumerate}
    \item[(a)] For Algorithm~\ref{alg:PE} with a constant step size $\alpha$, to estimate $\kappa(f)$ up to a mean estimation error of at most  $\epsilon$, we require
    \[ n \ge O\lrp{\frac{1}{\Delta_1 \alpha} \log\frac{\xi_1}{\epsilon^2} } =  O\lrp{\frac{\log\frac{1}{\epsilon^2}}{\epsilon^2}} \tilde{O}\lrp{ \frac{\eta^2\|V^*\|^2_2}{\Delta^2_1} }=  O\lrp{\frac{\log\frac{1}{\epsilon^2}}{\epsilon^2}} \tilde{O}\lrp{ \frac{\|V^*\|^2_2}{\Delta^4_1} }. \]
    This follows from Theorem~\ref{th:finitebound}(a), where setting the second term to at most $\epsilon^2$ gives 
    \[ \alpha = O\lrp{\epsilon^2}O\lrp{ \frac{\Delta_1}{\|V^*\|^2_2 \eta^2} }.
    \]
    %Here, $\tilde{O}(\cdot)$ hides poly-logarithmic factors. 
    Choosing $\alpha$ to satisfy this condition, and setting the first term in the error bound to $\epsilon^2$, we get the required sample complexity bound.
    
    \item[(b)] Next, for Algorithm~\ref{alg:PE} with diminishing step size $\alpha_i = \frac{\alpha}{i+h}$ for all $i\ge 1$, for estimating $\kappa$ up to a mean estimation error at most $\epsilon$, we require
    \[ n + h = O\lrp{ \frac{\alpha^2 \xi_2 \eta^2}{\epsilon^2} }  = O\lrp{\frac{1}{\epsilon^2}} O\lrp{  \frac{\eta^2\|V^*\|^2_2}{\Delta^2_1}}= O\lrp{\frac{1}{\epsilon^2}} O\lrp{  \frac{\|V^*\|^2_2}{\Delta^4_1}}. \]
    This follows from the mean-square estimation error bound in Theorem~\ref{th:finitebound}(b) by setting the rate-determining second term to at most $\epsilon^2$, after optimizing the term $$ \lrp{ \frac{\alpha^2}{\Delta_1 \alpha - 40} }$$ over $\alpha$. Note that $\alpha = O(1/\Delta_1)$ is the optimal choice. Further, with this choice of $\alpha$,  since $h = O(1/\Delta^4)$, we get
    \[ n = O\lrp{\frac{1}{\epsilon^2}} O\lrp{  \frac{\|V^*\|^2_2}{\Delta^4_1}}. \]
    \end{enumerate}

\subsection{MSE Bound for Stationary Variance Estimation from Section~\ref{sec:statvar}}\label{app:stationary_var}

In this section, we prove Theorem~\ref{th:finitebound_statvar} in $2$ steps, that we discuss below. 
\subsubsection{Step 1: Algorithm~\ref{alg:StationaryVariance} as a Linear SA}\label{app:step1_statvar}
    Algorithm~\ref{alg:StationaryVariance} for estimating the stationary variance of $f$, $v(f)$, can be formulated as a linear SA update, as discussed in this section. Consider the following matrices: 
\[
A^S_k := A^S(X_k)=  
\begin{bmatrix}
    -1 & 0 \\
    -cf(X_k) & -c
\end{bmatrix}, \quad\text{and}\quad b^S_k := b^S(X_k) = \begin{bmatrix}
    f(X_k)\\
    cf^2(X_k)
\end{bmatrix}.
\]
For $(\Theta^S_k)^T := [\bar{f}_{k}~v_k]$, Algorithm~\ref{alg:StationaryVariance} corresponds to the following update
\[ \Theta^S_{k+1} = \Theta^S_k + \alpha_k (A^S_k \Theta^S_k + b^S_k).\]
Define $A^S := \Exp{\pi}{A_k}$ and $b^S:= \Exp{\pi}{b_k}$. Then,
\[ 
A^S = \begin{bmatrix}
    -1 & 0\\
    -c \bar{f} & -c
\end{bmatrix}, \quad\text{and}\quad b^S = \begin{bmatrix}
    \bar{f}\\
    c\Exp{\pi}{f^2}
\end{bmatrix}.
\]

We now prove a contraction property for the average matrix $A^S$.

\begin{lemma}\label{lem:negdef_statvar}
    For any $\gamma > 1$ and $c > 0$ such that 
    \[ c {\bar{f}}^2 \le 2 \lrp{ - (\gamma-1) + \sqrt{(\gamma-1)\lrp{\gamma-1 + \gamma {\bar{f}}^2}} }, \]
    matrix $A^S$ satisfies 
    \[ \min\limits_{\Theta\in\R^2, \|\Theta\|_2 = 1}-  \Theta^T A^S \Theta > \gamma > 0. \]
\end{lemma}
\begin{proof}
    Consider $\Theta^T = [\theta_1~\theta_2]$. Then, 
    \begin{align*}
        -\Theta^T A^S \Theta = \theta^2_1 + c \bar{f} \theta_1\theta_2 + c\theta^2_2.
    \end{align*}
    Now, consider the following:
    
    \begin{align*}
    \min\limits_{\substack{\theta_1\in\R, \theta_2 \in \R\\ \theta^2_1 + \theta^2_2 = 1}}    -\Theta^T A^S \Theta &\ge \min\limits_{\substack{\theta_1 \in \R, \theta_2\in\R, \\ \theta^2_1 + \theta^2_2 = 1}} ~ \theta^2_1 - c |\bar{f}| |\theta_1| |\theta_2| + c \theta^2_2\\
    &= \min\limits_{\substack{|\theta_1| \le 1}} ~ \theta^2_1 - c \abs{\bar{f}} |\theta_1| \sqrt{1-\theta^2_1} + c (1-\theta^2_1)\\
    &= \min\limits_{\substack{|\theta_1| \le 1}} ~ (1-c)\theta^2_1 - c \abs{\bar{f}} |\theta_1| \sqrt{1-\theta^2_1} + c \\
    &= \frac{1+c}{2} - \frac{1}{2}\sqrt{(1-c)^2 + c^2(\bar{f})^2}. \tag{Lemma~\ref{lem:aux1}}
    \end{align*}

    This lower bound greater than $\gamma$ is equivalent to 
    \[ \frac{c^2\bar{f}^2}{4} - c(1-\gamma) + \gamma - \gamma^2 \le 0, \]
    which has a solution (for $c$) if $\gamma > 1$ or $\gamma < \frac{1}{1 + \bar{f}^2}$. For $ \gamma > 1$, the upper bound on $c$ in the condition in the statement of the lemma is the positive root of the above equation. 
\end{proof}
To make the above  condition explicit for specific choices for $\gamma$, 
    \[ \gamma = 1 + \bar{f}^2 \implies c \le 2 \lrs{ - 1 + \sqrt{2 + \bar{f}^2}};\]
    \[ \gamma =2 \implies c {\bar{f}^2} \le 2\lrs{ -1 + \sqrt{1 + 2 \bar{f}^2} }. \]

\subsubsection{Step 2: Proving the MSE Bound}\label{app:step2_statvar}
Let $\Theta^s := [\bar{f} ~ v(f)]$ and $\Theta_k = [\bar{f}_k~v_k]$. Define
 \[ \eta = \max\lrset{1, \|A^S\|_2, \|b^S\|_2}. \]

The MSE bound in Theorem~\ref{th:finitebound_statvar} follows from Theorem~\ref{th:appfiniteboundlfa}. To see that the conditions of Theorem~\ref{th:appfiniteboundlfa} are satisfied, first observe that the iterates $\Theta^S$ are generated by a linear SA, as formulated in the previous section (Section~\ref{app:step1_statvar}). Furthermore, for $d=2$ and subspace $\mathbb S = \R^2$, Lemma~\ref{lem:negdef_statvar} establishes the required contraction property for the averaged matrix $A^S$, and $\eta$ in this section corresponds to the constant $\upb$. Then,  the iterates $\Theta^S_k$ satisfy the bounds in Theorem~\ref{th:appfiniteboundlfa} with $\Theta^* = [\bar{f}~v(f)]$,  and $\gamma_2 = \gamma$.

\subsection{Estimating Variance: IID Setting}\label{app:iidestimation}
Consider  $X_1, X_2, \dots$, independent samples from a distribution with mean $0$ and an unknown variance $ \sigma^2$. Further assume that the sampling distribution has bounded $4^{th}$ moment. The goal is to estimate $\sigma^2$ using these samples. 

Since, in this setting, $\V\lrs{X} = \E{X^2}$, let 
\[\hat{\sigma}^2_n := \frac{1}{n} \sum\limits_{j=1}^n X^2_j\]
denote the estimate for $\sigma^2$ using $n$ samples. Clearly, $\E{\hat{\sigma}^2_n} = \sigma^2$, i.e., $\hat{\sigma}^2_n$ is an unbiased estimator.

Next, consider the following
\begin{align*}
    \E{(\hat{\sigma}^2_n - \sigma^2)^2} 
    &= \E{ (\hat{\sigma}^2_n)^2 + \sigma^4 - 2\sigma^2 \hat{\sigma}^2_n }\\
    &= \E{ (\hat{\sigma}^2_n)^2} - \sigma^4 \\
    &= \frac{1}{n^2}\E{ \sum\limits_{j=1}^n X^4_j + \sum\limits_{i\ne j} X^2_iX^2_j } - \sigma^4\\
    &= \frac{\E{X^4}}{n} + \frac{n(n-1)\sigma^4}{n^2} - \sigma^4\\
    &= \frac{1}{n}\lrp{\E{X^4} - \sigma^4}\\
    &= \frac{c_1}{n},
\end{align*}
where $c_1 = \E{X^4}-(\E{X^2})^2$. Thus, we have that the mean-squared estimation error in this setting is exactly $O(\frac{1}{n})$. In fact, Cram\'er-Rao lower bound \cite{nielsen2013cramer} for the mean-squared error  in estimating $\sigma^2$ using an unbiased estimators is $O(\frac{1}{n})$, establishing that this rate cannot be improved in certain settings.

For $\sigma > 0$, let $\G(0,\sigma^2)$ denote the Gaussian distribution with variance $\sigma^2$. If $X \sim \G(0,\sigma^2)$, $c_1$ equals $3\sigma^2$, and in this case we have that 
\[ \E{ \lrp{\hat{\sigma}^2_n - \sigma^2}^2 } = \frac{3\sigma^2}{n}, \quad\text{ for $c > 0$}.\]

\section{Details of Proof for MSE Bound in Theorem~\ref{th:finiteboundlfa}}\label{app:mainresultslfa}
In this section, we give the missing details and proofs of results from Section~\ref{sec:proofsketch} (proof of Theorem~\ref{th:finiteboundlfa}). As discussed in the main text, the proof of Theorem~\ref{th:finiteboundlfa} proceeds in two steps. We give details of Step 1 in Appendix~\ref{app:Step1}, and that for the Step 2 in Appendix~\ref{sec:verifyconds2}. 

\subsection{Details of Step 1: Algorithm as a Linear SA}\label{app:Step1}
In this section, we add details to the discussion in Section~\ref{sec:alglinearSA} about formulating the proposed algorithm as a linear SA update. Below, we first present the exact form of the average matrices (Section~\ref{app:avgmatrices}), followed by a discussion on the limit point (if convergence happens) of the proposed updates (Section~\ref{app:odedisc}). We then conclude with a proof for Lemma~\ref{th:negdeflfa} in Section~\ref{app:th:negdeflfa}. 

\subsubsection{Stationary-Average Matrices}\label{app:avgmatrices}
The average matrices $A $ and $b$ are given by   
\[
A = \begin{bmatrix}
    -c_1 & \vec{\bf 0}^T & 0 & 0\\
    - \Pi_{2,E}\Phi^T D_\pi \vec{\bf 1}  & 
        {\Pi_{2,E}\Phi^T D_\pi (P - I)\Phi} &  \vec{\bf 0} & \vec{\bf 0}\\
    0 & c_2 \vec{\bf 1}^T D_\pi \Phi & -c_2 &  0\\
    c_3 \bar{f} & 2c_3\mathcal F^T  D_\pi \Phi & -2c_3\bar{f} & -c_3
\end{bmatrix}, \quad \text{and}\quad b = 
\begin{bmatrix}
     c_1 \bar{f}\\
     \Pi_{2,E}\Phi^T D_\pi \mathcal F \\
     0\\
     -{c_3}\Exp{\pi}{f^2(X)}
\end{bmatrix},
 \]
where $P$ is the transition matrix for the Markov chain $\mathcal M$, and $I$ is the identity matrix. This follows from linearity of the expectation (projection is a linear operator) and the observation that the stationary expectation of each entry of matrix $A(\cdot)$ and $b(\cdot)$ is given by the corresponding entry of $A$ and $b$, respectively. 

\subsubsection{Algorithm's Limit: A Discussion}\label{app:odedisc} 
Let $\Theta^{*T} := [\bar{f}~ \theta^{*T} ~ 
\widetilde{V}~ \kappa^* ]$, 
where recall that
\begin{align*}
    \kappa^* = \Exp{\pi}{2f(X)[\Phi\theta^*](X) - 2f(X)\widetilde{V} - f^2(X) + f(X)\bar{f}}, 
\end{align*}
$\Phi\theta^* \in \R^S$, and $[\Phi\theta^*](i)$ represents its $i^{th}$ component, $\widetilde{V} = \Exp{\pi}{[\Phi\theta^*](X)}$,
and $\theta^*$ is the unique vector in $E$ that is also a solution for \[ \Phi\theta = \Pi_{D_\pi, W_\Phi} T\Phi\theta, \]
where
$\Pi_{D_\pi, W_\Phi}$ is the projection matrix onto $W_\Phi:= \lrset{\Phi\theta | \theta\in\R^d}$ with respect to $D_\pi$ norm. Here, $T$ is an operator that for a vector $V\in \R^S$, satisfies $T V = \mathcal F - \bar{f} {\bf e} + P V$.

Let $\Theta_\infty := [{\bar f}_\infty~ \theta_\infty~\widetilde{V}_\infty ~ \kappa_\infty]$, where $\theta_\infty \in E$. We now show that $\Theta_\infty = \Theta^*$ is the unique solution to $A\Theta_\infty + b = 0$. This  equation corresponds to the following:
\begin{align} -c_1 \bar{f}_\infty + c_1 \bar{f}  &= 0, \label{eq:j_mu} \\
-\Pi_{2,E}\Phi^TD_\pi {\bf 1} \bar{f}_\infty + \Pi_{2,E}\Phi^TD_\pi(P - I)\Phi \theta_\infty + \Pi_{2,E}\Phi^T D_\pi \mathcal F &= 0, \label{eq:thetastar} \\
c_2 \vec{\bf 1}^T D_\pi \Phi \theta_\infty - c_2 \widetilde{V}_\infty &=0, \label{eq:tildeV} \\
c_3 \bar{f}\bar{f}_\infty + 2c_3 \mathcal F^T D_\pi \Phi \theta_\infty - 2c_3 \bar{f}\widetilde{V}_\infty - c_3 \kappa_\infty - c_3 \Exp{\pi}{f^2} &= 0.\label{eq:kappamu}
\end{align}

Clearly,~\eqref{eq:j_mu} implies that  $\bar{f}_\infty = \bar{f}$. Moreover, if $\theta_\infty = \theta^* + a \theta_e$, for $a\in\R $, where $\theta_e$ is the unique vector in $\R^d$ such that 
$\Phi\theta_e = {\bf e}$, then~\eqref{eq:tildeV} implies that $\widetilde{V}_\infty = \Exp{\pi}{\Phi(\theta^* + a\theta_e)} = \widetilde{V} + a{\bf e}$ and then,~\eqref{eq:kappamu} implies $\kappa_\infty = \kappa^*$. It remains to show that $\theta_\infty = \theta^* + a \theta_e$, for some $a\in\R$. In fact, since $\theta_\infty \in E$, we show below that $\theta_\infty = \theta^*$. 

Let's use $\bar{f}_\infty = \bar{f}$, and re-write the LHS of~\eqref{eq:thetastar} as below: 
\[ \Pi_{2,E}  \Phi^T D_\pi \lrp{  (\mathcal F - \bar{f} {\bf e}) +(P - I)\Phi \theta_\infty  }   . \]
Clearly, it equals $0$ for $\theta_\infty = \theta^*$. This follows from the definition of $\theta^*$. Hence, $\theta^*$ is a solution for~\eqref{eq:thetastar}. Furthermore, $\theta = \theta^* + c \theta_e$ for any $c\in \R$ cannot be a solution, since $\theta^* + c\theta_e \notin E$ and $\theta_\infty \in E$. We now show that there does not exist any $\theta\in E$ different from $\theta^*$ that satisfies~\eqref{eq:thetastar}. To this end, suppose such a $\theta' \in E$ exists. Then,~\eqref{eq:thetastar} evaluated at $\theta'$ re-writes as
\[  \Pi_{2,E}\Phi^TD_\pi \lrp{\underbrace{ -  \bar{f}{ \bf e } + (P - I)\Phi \theta^* +  \mathcal F }_{=0}} +  \Pi_{2,E} \Phi^TD_\pi  (P - I)\Phi (\theta' - \theta^*),\]
where the first term equals $0$ from the definition of $\theta^*$. Now, recall that $\Pi_{2,E} = I - \theta_e \theta^T_e$, with the convention that $\theta_e = 0$ if ${\bf e}\notin W_\Phi$, the column space of $\Phi$. Further, $\theta^T_e \Phi^T = {\bf e}^T \in \R^{S}$, which implies 
\[  \theta^T_e \Phi^T D_\pi (P - I) = 0. \]
Using these, the second term above equals
\[ \Phi^TD_\pi  (P - I)\Phi (\theta' - \theta^*) - \theta_e \underbrace{\theta^T_e\Phi^TD_\pi (P - I)\Phi (\theta' - \theta^*)}_{=0}. \]
Now, the first term above is non-zero since $\theta' - \theta^* \in E$, and hence, $\Phi(\theta' - \theta^*)$ is non-constant vector that does not belong to the null space of $P - I$. Thus, $\nexists \theta' \in E$ different from $\theta^*$ that also satisfies~\eqref{eq:thetastar}.

\subsubsection{Proof of Lemma~\ref{th:negdeflfa}}\label{app:th:negdeflfa}
In this appendix, we will show  the following:
\[ \min\limits_{\substack{\Theta\in\R\times E\times \R\times \R, \\ \|\Theta\|^2_2 =1} } ~  - \Theta^T A \Theta  > \frac{\Delta_2}{20} > 0, \numberthis\label{eq:minev} \]
where 
\[\Delta_2 := \min\limits_{\|\theta\|_2 = 1, \theta \in E} ~ \theta^T \Phi^T D_\pi(I-P) \Phi \theta.   \]

Recall that for ${\theta} \in E$, ${\Phi \theta}$ is a non-constant vector in $\R^S$. Thus,   $\theta^T \Phi^T D_\pi (I-P) \Phi{\theta} > 0,$
for ${\theta} \in E$ \cite[Lemma 7]{tsitsiklis1997average}. Since the set $\lrset{ \theta \in E : \|\theta\|^2_2 = 1}$ is non-empty and compact, by extreme value theorem, we have $\Delta_2 > 0$. 

Next, we re-write the above minimization problem as (set $\gamma = 1$ in the following)   
\begin{align*}
   &\min\limits_{\substack{ \|\Theta\|^2_2 = 1, \\ \vec{\theta}_2 \in E }} ~ -\Theta^T A \Theta \\
    &=    \min\limits_{\substack{ \|\Theta\|^2_2 = 1, \\ \vec{\theta}_2 \in E }} ~ ~ c_1 \theta^2_1 + \gamma \vec{\theta}^T_2 \Pi_{2,E}\Phi^T D_\pi \vec{1} \theta_1 + \gamma \vec{\theta}^T_2 \Pi_{2,E}\Phi^T D_\pi (I - P) \Phi\vec{\theta}_2 - c_2 \vec{1}^T D_\pi \Phi \vec{\theta}_2 \theta_3 \\ 
    &\qquad\qquad\qquad + c_2 \theta^2_3 - c_3 \bar{f} \theta_1 \theta_4 - 2 c_3 \mathcal F^T D_\pi \Phi\vec{\theta}_2 \theta_4 + 2c_3 \bar{f} \theta_3 \theta_4 + c_3 \theta^2_4\\
    &\overset{(a)}{=}    \min\limits_{\substack{ \|\Theta\|^2_2 = 1, \\ \vec{\theta}_2 \in E }} ~ ~ c_1 \theta^2_1 + \gamma \vec{\theta}^T_2 \Phi^T D_\pi \vec{1} \theta_1 + \gamma \vec{\theta}^T_2 \Phi^T D_\pi (I - P) \Phi\vec{\theta}_2 - c_2 \vec{1}^T D_\pi\Phi \vec{\theta}_2 \theta_3 \\ 
    &\qquad\qquad\qquad + c_2 \theta^2_3 - c_3 \bar{f} \theta_1 \theta_4 - 2 c_3 \mathcal F^T D_\pi \Phi\vec{\theta}_2 \theta_4 + 2c_3 \bar{f} \theta_3 \theta_4 + c_3 \theta^2_4\\
    &\overset{(b)}{\ge}\min\limits_{\substack{ \|\Theta\|^2_2 = 1, \\ \vec{\theta}_2 \in E }} ~ ~ c_1 \theta^2_1 + \gamma \vec{\theta}^T_2 \Phi^T D_\pi \vec{1} \theta_1 + \gamma \Delta_2 \|\vec{\theta}_2\|^2_2 - c_2 \vec{1}^T D_\pi \Phi \vec{\theta}_2 \theta_3 \\ 
    &\qquad\qquad\qquad + c_2 \theta^2_3 - c_3 \bar{f} \theta_1 \theta_4 - 2 c_3 \mathcal F^T D_\pi \Phi\vec{\theta}_2 \theta_4 + 2c_3 \bar{f} \theta_3 \theta_4 + c_3 \theta^2_4\\
    &=\min\limits_{\substack{ \|\Theta\|^2_2 = 1, \\ \vec{\theta}_2 \in E }} ~ ~ c_1 \theta^2_1 + \gamma \Delta_2 \|\vec{\theta}_2\|^2_2 + c_2 \theta^2_3 + c_3 \theta^2_4 \\ 
    &\qquad\qquad\qquad + \vec{\theta}^T_2 \Phi^T D_\pi \lrp{  \gamma \vec{1} \theta_1 - c_2 \vec{1} \theta_3  - 2 c_3 \theta_4 \mathcal F} + 2c_3 \bar{f} \theta_3 \theta_4 - c_3 \bar{f} \theta_1 \theta_4.
\end{align*}
where, $(a)$ follows since $\theta_2\in E$, and $(b)$ follows from the definition of $\Delta_2$. Now, 
\begin{align*}
    &\abs{ \vec{\theta}^T_2 \Phi^T D_\pi\lrp{ \gamma \vec{1}\theta_1 - c_2 \vec{1}\theta_3 - 2c_3 \theta_4 \mathcal F}}\\
    &\le \abs{ \gamma \theta_1\vec{\theta}^T_2 \Phi^T D_\pi \vec{1} } + \abs{ c_2\theta_3\vec{\theta}^T_2 \Phi^T D_\pi \vec{1} } + \abs{2c_3\theta_4\vec{\theta}^T_2 \Phi^T D_\pi \mathcal F}\\
    &\le \gamma \abs{\theta_1} \| \Phi\vec{\theta}_2\|_\infty\| D_\pi \vec{1}\|_1 + c_2\abs{ \theta_3}\|\Phi\vec{\theta}_2\|_\infty \| D_\pi \vec{1} \|_1 + 2c_3\abs{\theta_4}\|\Phi\vec{\theta}_2\|_\infty \| D_\pi \mathcal F\|_1\\
    &\le \gamma \abs{\theta_1} \|\Phi\vec{\theta}_2\|_\infty + c_2\abs{ \theta_3}\|\Phi\vec{\theta}_2\|_\infty  + 2c_3\abs{\theta_4}\|\Phi\vec{\theta}_2\|_\infty f_{\max}\\
    &\overset{(a)}{\le} \gamma \abs{\theta_1} \|\vec{\theta}_2\|_2 + c_2\abs{ \theta_3}\|\vec{\theta}_2\|_2  + 2c_3\abs{\theta_4}\|\vec{\theta}_2\|_2 f_{\max},
\end{align*}
where $f_{\max} := \max_s f(s)$, and $(a)$ follows since $$\|\Phi\vec{\theta}_2\|_\infty \le \max\limits_{s\in\calS} \|\phi(s)\|_2 \|\vec{\theta}_2\|_2 \le \|\vec{\theta}_2\|_2,$$ 
where the first inequality follows from the definition of $\|\cdot\|_\infty$ and Holder's inequality, and the second from Assumption~\ref{asmp:2}. 

Next, using the previous bound, 
\begin{align*}
    &\min\limits_{\substack{ \|\Theta\|^2_2 = 1, \\ \vec{\theta}_2 \in E }} ~ -\Theta^T A \Theta \\
    \ge &\min\limits_{\substack{ \|\Theta\|^2_2 = 1, \\ \vec{\theta}_2 \in E }} ~ ~ c_1 \theta^2_1 + \gamma \Delta_2 \|\vec{\theta}_2\|^2_2 + c_2 \theta^2_3 + c_3 \theta^2_4 + 2c_3 \bar{f} \theta_3 \theta_4 - c_3 \bar{f} \theta_1 \theta_4\\ 
    &\qquad\qquad\qquad - \gamma \abs{\theta_1} \|\vec{\theta}_2\|_2 - c_2\abs{ \theta_3}\|\vec{\theta}_2\|_2  - 2c_3\abs{\theta_4}\|\vec{\theta}_2\|_2 f_{\max}  \\
    = &\min\limits_{\substack{ \theta^2_3 + \theta^2_4 \in[0,1] }} ~~ \min\limits_{\substack{\theta^2_1 + \|\vec{\theta}_2\|^2_2 = 1 - \theta^2_3 - \theta^2_4\\
    \vec{\theta}_2 \in E, \theta_1 \in \R}}  ~ ~ c_1 \theta^2_1 + \gamma \Delta_2 \|\vec{\theta}_2\|^2_2 + c_2 \theta^2_3 + c_3 \theta^2_4  + 2c_3 \bar{f} \theta_3 \theta_4 - c_3 \bar{f} \theta_1 \theta_4\\ 
    &\qquad\qquad\qquad\qquad\qquad\qquad\qquad - \gamma \abs{\theta_1} \|\vec{\theta}_2\|_2 - c_2\abs{ \theta_3}\|\vec{\theta}_2\|_2  - 2c_3\abs{\theta_4}\|\vec{\theta}_2\|_2 f_{\max}\\
    = &\min\limits_{\substack{ \theta^2_3 + \theta^2_4 \in[0,1] }} ~ c_2 \theta^2_3 + c_3 \theta^2_4  + 2c_3 \bar{f} \theta_3 \theta_4 + (P1)\\
    \ge & \min\limits_{\substack{ \theta^2_3 + \theta^2_4 \in[0,1] }} ~ c_2 \theta^2_3 + c_3 \theta^2_4  - 2c_3 f_{\max} \abs{\theta_3}\abs{ \theta_4} + (P1) \numberthis\label{eq:original},
\end{align*}
where,
\begin{align*} (P1) := \min\limits_{\substack{\theta^2_1 + \|\vec{\theta}_2\|^2_2 = 1 - \theta^2_3 - \theta^2_4\\
    \vec{\theta}_2 \in E, \theta_1 \in \R}}   &~~\big\{ c_1 \theta^2_1 + \gamma \Delta_2 \|\vec{\theta}_2\|^2_2 - c_3 \bar{f} \theta_1 \theta_4 - \gamma \abs{\theta_1} \|\vec{\theta}_2\|_2 \\
    &\qquad\qquad\qquad\qquad - c_2\abs{ \theta_3}\|\vec{\theta}_2\|_2  - 2c_3\abs{\theta_4}\|\vec{\theta}_2\|_2 f_{\max}\big\}. 
\end{align*}

    Suppose for $y_1 \in \R$, $y_2 \in \R$, and $y_3 \in \R$, 
    \begin{align*}
        c_1 - \gamma \Delta_2 \ge y_1, & \qquad c_2 - c_3 \ge y_2, \\
        c_3-\gamma\Delta_2 - \frac{{ y_1 - \sqrt{y^2_1 + c^2_2} }}{2} &+ \frac{ y_2 - \sqrt{y^2_2 + 4c^2_3 f^2_{\max} } }{2} \ge y_3.\numberthis \label{eq:conds}.\end{align*} 
    
    Let us first simplify $(P1)$ as below:
    \begin{align*}
    (P1) &= \min\limits_{\substack{\theta^2_1 \in [0, 1 - \theta^2_3 - \theta^2_4]\\
    \theta_1 \in \R}}   \left\{ c_1 \theta^2_1 + \gamma \Delta_2 (1-\theta^2_1 - \theta^2_3 - \theta^2_4) - c_3 \bar{f} \theta_1 \theta_4 \right.\\ &\qquad\qquad\qquad\qquad\qquad\left. -  \lrp{\gamma \abs{\theta_1} + c_2\abs{ \theta_3}  + 2c_3\abs{\theta_4}f_{\max}} \sqrt{1-\theta^2_1 -\theta^2_3 - \theta^2_4} \right\}\\
    &\ge  \min\limits_{\theta_1 \in \left[0, \sqrt{1 - \theta^2_3 - \theta^2_4}\right]}   \left\{ c_1 \theta^2_1 + \gamma \Delta_2 (1-\theta^2_1 - \theta^2_3 - \theta^2_4) - c_3 {f}_{\max} \abs{\theta_1} \abs{\theta_4} \right.\\ &\qquad\qquad\qquad\qquad\qquad\left. -  \lrp{\gamma \abs{\theta_1} + c_2\abs{ \theta_3}  + 2c_3\abs{\theta_4}f_{\max}} \sqrt{1-\theta^2_1 -\theta^2_3 - \theta^2_4} \right\}\\
    &\ge \gamma \Delta_2 (1 - \theta^2_3 - \theta^2_4)  - c_3 {f}_{\max} \abs{\theta_4} \sqrt{ 1-\theta^2_3 - \theta^2_4 } \\ 
    & \qquad \qquad + \min\limits_{\theta_1 \in \left[0, \sqrt{1 - \theta^2_3 - \theta^2_4}\right]}   \bigg\{ (c_1 - \gamma \Delta_2)\theta^2_1  \\
    &\qquad\qquad\qquad\qquad\qquad\quad\quad-  \lrp{\gamma \abs{\theta_1} + c_2\abs{ \theta_3}  + 2c_3\abs{\theta_4}f_{\max}} \sqrt{1-\theta^2_1 -\theta^2_3 - \theta^2_4} \bigg\}\\
    &\ge \gamma \Delta_2 (1 - \theta^2_3 - \theta^2_4)  - c_3 {f}_{\max} \abs{\theta_4} \sqrt{ 1-\theta^2_3 - \theta^2_4 }  \\ 
        &\qquad\qquad\qquad-  \lrp{c_2\abs{ \theta_3}  + 2c_3\abs{\theta_4}f_{\max}} \sqrt{1-\theta^2_3 - \theta^2_4}\\ 
        & \qquad \qquad \qquad + \min\limits_{\theta_1 \in \left[0, \sqrt{1 - \theta^2_3 - \theta^2_4}\right]}   \left\{ (c_1 - \gamma \Delta_2)\theta^2_1   -  \gamma  \sqrt{\theta^2_1  \lrp{1-\theta^2_1 -\theta^2_3} - \theta^4_4} \right\}\\
    &\ge \gamma \Delta_2 (1 - \theta^2_3 - \theta^2_4)  - c_3 {f}_{\max} \abs{\theta_4} \sqrt{ 1-\theta^2_3 - \theta^2_4 }  \\
        &\qquad\qquad\qquad -  \lrp{c_2\abs{ \theta_3}  + 2c_3\abs{\theta_4}f_{\max}} \sqrt{1-\theta^2_3 - \theta^2_4}\\ & \qquad \qquad \qquad + \min\limits_{\theta_1 \in \left[0, \sqrt{1 - \theta^2_3 - \theta^2_4}\right]}   \left\{ y_1\theta^2_1   -  \gamma  \sqrt{\theta^2_1  \lrp{1-\theta^2_1 -\theta^2_3} - \theta^4_4} \right\}\\
    &\overset{(a)}{=} \gamma \Delta_2 (1 - \theta^2_3 - \theta^2_4)  - c_3 {f}_{\max} \abs{\theta_4} \sqrt{ 1-\theta^2_3 - \theta^2_4 } \\
        &\qquad\qquad\qquad -  \lrp{c_2\abs{ \theta_3}  + 2c_3\abs{\theta_4}f_{\max}} \sqrt{1-\theta^2_3 - \theta^2_4}\\ & \qquad \qquad \qquad + \frac{1-\theta^2_3 - \theta^2_4}{2} \lrp{ y_1 - \sqrt{y^2_1 + c^2_2} },
\end{align*}
where $(a)$ follows from Lemma~\ref{lem:aux1}. Now, substituting $(P1)$ in~\eqref{eq:original}, we get
\begin{align*}
    &\min\limits_{\substack{ \|\Theta\|^2_2 = 1, \\ \vec{\theta}_2 \in E }} ~ -\Theta^T A \Theta \\
    \ge &\min\limits_{\substack{ \theta^2_3 + \theta^2_4 \in[0,1] }} ~ c_2 \theta^2_3 + c_3 \theta^2_4  - 2c_3 f_{\max} \abs{\theta_3}\abs{ \theta_4} + \gamma \Delta_2 (1 - \theta^2_3 - \theta^2_4) \\
        &\qquad\qquad\qquad- c_3 {f}_{\max} \abs{\theta_4} \sqrt{ 1-\theta^2_3 - \theta^2_4 } -  \lrp{c_2\abs{ \theta_3}  + 2c_3\abs{\theta_4}f_{\max}} \sqrt{1-\theta^2_3 - \theta^2_4} \\
        &\qquad\qquad\qquad+ \frac{1-\theta^2_3 - \theta^2_4}{2} \lrp{ y_1 - \sqrt{y^2_1 + c^2_2} }\\
    = & \min\limits_{x^2\in [0,1]} ~\min\limits_{\substack{ \theta^2_3 + \theta^2_4 = x^2 }} ~ c_2 \theta^2_3 + c_3 \theta^2_4  - 2c_3 f_{\max} \abs{\theta_3}\abs{ \theta_4} + \gamma \Delta_2 (1 - \theta^2_3 - \theta^2_4) \\
    & \qquad \qquad \qquad -  \lrp{c_2\abs{ \theta_3}  + 3c_3\abs{\theta_4}f_{\max}} \sqrt{1-\theta^2_3 - \theta^2_4} \\
    &\qquad\qquad\qquad + \frac{1-\theta^2_3 - \theta^2_4}{2} \lrp{ y_1 - \sqrt{y^2_1 + c^2_2} }\\
    = & \gamma\Delta_2 + \frac{ y_1 - \sqrt{y^2_1 + c^2_2}}{2} + ~\min\limits_{x^2\in [0,1]} ~\lrp{c_3-\gamma\Delta_2 - \frac{{ y_1 - \sqrt{y^2_1 + c^2_2} }}{2}}x^2 \\  
    & \qquad \qquad \qquad +~\min\limits_{\substack{ \theta^2_3 \in [0, x^2] }} ~~~ \bigg\{  (c_2-c_3) \theta^2_3  \\ 
    &\qquad\qquad\qquad\qquad- 2c_3 f_{\max} \sqrt{\theta^2_3 x^2 - \theta^4_3} 
    -\lrp{c_2\abs{ \theta_3}  + 3c_3f_{\max}\sqrt{x^2 - \theta^2_3}} \sqrt{1-x^2} \bigg\} \\
    \ge & \gamma\Delta_2 + \frac{ y_1 - \sqrt{y^2_1 + c^2_2}}{2} + ~\min\limits_{x\in [0,1]} ~\lrset{c_3-\gamma\Delta_2 - \frac{{ y_1 - \sqrt{y^2_1 + c^2_2} }}{2}}x^2 \\ 
    & \qquad\qquad\qquad -\lrp{c_2  + 3c_3f_{\max}} \sqrt{x^2-x^4} \\
    &\qquad\qquad\qquad ~+~\min\limits_{\substack{ \theta_3 \in [0, x] }} ~~~ \lrset{(c_2-c_3) \theta^2_3  - 2c_3 f_{\max} \sqrt{\theta^2_3 x^2 - \theta^4_3}}  \\
    \ge & \gamma\Delta_2 + \frac{ y_1 - \sqrt{y^2_1 + c^2_2}}{2} + ~\min\limits_{x\in [0,1]} ~\lrset{c_3-\gamma\Delta_2 - \frac{{ y_1 - \sqrt{y^2_1 + c^2_2} }}{2}}x^2 \\
    &\qquad\qquad\qquad -\lrp{c_2  + 3c_3f_{\max}} \sqrt{x^2-x^4}+~\min\limits_{\substack{ \theta_3 \in [0, x] }} ~~\lrset{ y_2 \theta^2_3  - 2c_3 f_{\max} \sqrt{\theta^2_3 x^2 - \theta^4_3} } \\
    \overset{(a)}{=} & \gamma\Delta_2 + \frac{ y_1 - \sqrt{y^2_1 + c^2_2}}{2} -\lrp{c_2  + 3c_3f_{\max}} \sqrt{x^2-x^4}\\ 
    & \qquad\qquad\qquad+ ~\min\limits_{x^2\in [0,1]} ~\lrset{c_3-\gamma\Delta_2 - \frac{{ y_1 - \sqrt{y^2_1 + c^2_2} }}{2} + \frac{ y_2 - \sqrt{y^2_2 + 4c^2_3 f^2_{\max} } }{2}}x^2 \\
    \ge & \gamma\Delta_2 + \frac{ y_1 - \sqrt{y^2_1 + c^2_2}}{2} + ~\min\limits_{x^2\in [0,1]} ~y_3x^2 -\lrp{c_2  + 3c_3f_{\max}} \sqrt{x^2-x^4}\\
    \overset{(b)}{=} & \gamma\Delta_2 + \frac{ y_1 - \sqrt{y^2_1 + c^2_2}}{2} + \frac{ y_3 - \sqrt{y^2_3 + (c_2 + 3c_3 f_{\max})^2} }{2},
\end{align*}
where, $(a)$ and $(b)$ follow from Lemma~\ref{lem:aux1}.

\vspace{0.75em}
\noindent{\bf Conditions on $c_1$, $c_2$, and $c_3$ to ensure a strictly positive lower bound. } For $d_1\in [-1,0]$, $d_2 < 0$, and $d_3 \in [-1,0]$, let
    \[\frac{ y_1 - \sqrt{y^2_1 + c^2_2}}{2} = d_1 \Delta_2,\]
    \[\frac{ y_2 - \sqrt{y^2_2 + 4c^2_3 f^2_{\max} } }{2} = d_2 \Delta_2,\]
    and 
    \[\frac{ y_3 - \sqrt{y^2_3 + (c_2 + 3c_3 f_{\max})^2} }{2} = d_3 \Delta_2.\]
Then, 
\[ y_1 = -\frac{c^2_2}{4 d_1 \Delta_2} + d_1 \Delta_2, \quad y_2 = - \frac{c^2_3 f^2_{\max}}{d_2 \Delta_2} + d_2 \Delta_2, \quad \text{ and }\quad y_3 = - \frac{(c_2 + 3 c_3 f_{\max})^2}{4 d_3 \Delta_2} + d_3 \Delta_2. \]
On choosing $\gamma = 1$, the lower bound becomes 
\[ (1+d_1 + d_3) \Delta_2, \]
and for $c>0$, setting 
\[d_2 = -c(1+d_1 + d_3),\]
the conditions in~\eqref{eq:conds} reduce to 
\begin{align*}
    &c_1 \ge - \frac{1}{4 d_1 \Delta_2} + \Delta_2 (1+d_1),\\
    &c_2 - c_3 - \frac{c^2_3 f^2_{\max}}{c (1+d_1 + d_3) \Delta_2} + c (1+d_1 + d_3) \Delta_2 \ge 0\\
    &c_3 + \frac{(c_2 + 3  c_3 f_{\max})^2}{4 d_3 \Delta_2} - (1+c)(1+d_1+d_3)\Delta_2 \ge 0,
\end{align*}
which for $f_{\max} = 1$, reduce to \begin{align*}
    &c_1 \ge - \frac{1}{4 d_1 \Delta_2} + \Delta_2 (1+d_1),\numberthis\label{eq:cond1}\\
    &c_2 - c_3 - \frac{c^2_3}{c (1+d_1 + d_3) \Delta_2} + c (1+d_1 + d_3) \Delta_2 \ge 0\numberthis\label{eq:cond2}\\
    &c_3 + \frac{(c_2 + 3  c_3 )^2}{4 d_3 \Delta_2} - (1+c)(1+d_1+d_3)\Delta_2 \ge 0 \numberthis\label{eq:cond3}.
\end{align*}

\vspace{0.75em}
\noindent{\bf Specific choices for the parameters. } The conditions~\eqref{eq:cond2} and~\eqref{eq:cond3} imply
\[ d_3 \le -\frac{9(1+c)(1+d_1)}{1+9(1+c)}, \] 
\[ \max\lrset{(1+c)(1 + d_1 + d_3) \Delta_2, -\frac{2 \Delta_2}{9 } \lrp{ d_3 + \sqrt{d^2_3 + 9 d_3 (1+c) (1+d_1+d_3)  } }} \le c_3,\]
\[c_3\le \frac{2 \Delta_2}{9 } \lrp{ -d_3 + \sqrt{d^2_3 + 9 d_3 (1+c) (1+d_1+d_3)  } }, \]
\[ c_3 + \frac{c^2_3}{c(1+d_1+d_3) \Delta_2} - c(1+d_1+d_3)\Delta_2 \le c_2,\]
and 
\[ c_2 \le -3 c_3  + 2 \sqrt{ d_3 \Delta_2 (-c_3 + (1+c)(1+d_1+d_3) \Delta_2)}.  \]

Choose $c$, $d_1$, $d_3$, and $c_1$ as below:  
    \[ c = -\frac{7}{24}, \quad d_1 = -\frac{1}{2}, \quad d_3 = - \frac{75}{166}, \quad c_1\ge \frac{1}{2\Delta_2} + \frac{\Delta_2}{2}. \]
Further, let $c_2$ and $c_3$ satisfy
\[ \frac{5}{249}(5-2\sqrt{2})\Delta_2 \le c_3\le \frac{5}{249} (5+2\sqrt{2})\Delta_2, \]
\[ c_3 -\frac{498 c^2_3}{7\Delta_2} + \frac{7 \Delta_2}{498}   \le c_2 \le -3 c_3 + \frac{5 }{83} \sqrt{498 c_3 \Delta_2 - 17 \Delta^2_2} . \]
We remark that the range of $c_2$ above is non-trivial for the feasible choices of $c_3$, and with these choices, 
\[ \min\limits_{\substack{ \|\Theta\|^2_2 = 1, \\ \vec{\theta}_2 \in E }} ~ -\Theta^T A \Theta \ge (1+d_1+d_3)\Delta_2 \ge \frac{4\Delta_2}{83}> \frac{\Delta_2}{20} > 0. \]
This proves the required contraction property of the update matrix $A$.

\subsection{Details of Step 2: MSE Bound}\label{sec:verifyconds2}

In this section, we add details to the discussion in Section~\ref{sec:MSEBound}. In this appendix, we verify that the  conditions in Theorem~\ref{th:appfiniteboundlfa} are met by the iterates and update matrices of our estimator, which will enable us to use Theorem~\ref{th:appfiniteboundlfa} to arrive at the desired bounds.

Consider the sequence of iterates $ \Theta_k  := [\bar{f}_k~ \theta^T_k~ \widetilde{V}_k  ~\kappa_k]$ generated by the proposed algorithm in Section~\ref{sec:alglfa} with $\Theta_0 = {\vec{\bf 0}}$ and $\theta_k \in E$. Recall from Section~\ref{sec:proofsketch} that these iterates can be rewritten as a linear stochastic-approximation update given below: 
\[ \Theta_{k+1} = \Theta_k + \alpha_k \lrp{  A(Y_k)\Theta_k  + b(Y_k)}. \]

Moreover, in our setting, recall that $\mathbb S = E$, and $\Theta_k \in E$, $\Theta^* \in E$. Lemma~\ref{th:negdeflfa} implies that the condition~\ref{eq:contractfactor} on update matrix $\A$ in Theorem~\ref{th:appfiniteboundlfa} is also satisfied in our setting with $\gamma_2 = \frac{\Delta_2}{20}$.  It only remains to verify the bound $\upb$ on update matrices. In what follows, we will show that $\upb$ in Theorem~\ref{th:appfiniteboundlfa} corresponds to $\eta$ in our setting. Theorem~\ref{th:finiteboundlfa} then corresponds to Theorem~\ref{th:appfiniteboundlfa} with these choices. 

\vspace{0.75em}
\noindent{\bf Bounds on ${A}(Y_t)$ and ${b}(Y_t)$. }
Define $\tilde{A}(Y_k)$ and $\tilde{b}(Y_k)$ so that $A(Y_k) = \Pi \tilde{A}(Y_k)$ and $b(Y_k) = \Pi \tilde{b}(Y_k)$, where
\[ \Pi := 
\begin{bmatrix}
    1  & \vec{\bf 0}^T & 0 & 0\\
    \vec{\bf 0} & \Pi_{2,E}  & \vec{\bf 0} & \vec{\bf 0}\\
    0  & \vec{\bf 0}^T  & 1 & 0\\
    0  & \vec{\bf 0}^T  & 0 & 1
\end{bmatrix}.\]
Similarly, define $\tilde{A}$ and $\tilde{b}$ so that $A = \Pi \tilde{A} $ and $b = \Pi \tilde{b}$, i.e.,
\[
\tilde{A} = \begin{bmatrix}
    -c_1 & \vec{\bf 0}^T & 0 & 0\\
    - \Phi^T D_\pi \vec{\bf 1}  & 
        {\Phi^T D_\pi (P - I)\Phi} &  \vec{\bf 0} & \vec{\bf 0}\\
    0 & c_2 \vec{\bf 1}^T D_\pi \Phi & -c_2 &  0\\
    c_3 \bar{f} & 2c_3\mathcal F^T  D_\pi \Phi & -2c_3\bar{f} & -c_3
\end{bmatrix}, \quad \text{and}\quad \tilde{b} = 
\begin{bmatrix}
     c_1 \bar{f}\\
     \Phi^T D_\pi \mathcal F \\
     0\\
     -{c_3}\Exp{\pi}{f^2(X)}
\end{bmatrix},
 \]
For a matrix $M$, let $\|M\|_2$  and $\|M\|_F$ denote its induced $2$-norm and Frobenius norm, respectively. Then, recall that $\|M\|_2 \le \|M\|_F$.
    \begin{align*}
        \|{A}(Y_t)\|_2 
        &= \|\Pi \tilde{A}(Y_t)\|_2\le \|\Pi\|_2\|\tilde{A(Y_t)}\|_2,
        \end{align*}
        where the above inequality follows from sub-multiplicativity of the induced $2$-norm. Now, since $\Pi$ is an orthogonal projection matrix, $\|\Pi\|_2 = 1$. Thus, we have 
        \begin{align*}
        \|A(Y_t)\|_2 &\le \|\tilde{A}(Y_t)\|_2\\
        &\stackrel{(a)}{\le} \|\tilde{A}(Y_t)\|_F\\
        &\le \sqrt{ c^2_1  + 1 + \| \phi_k(\phi^T_{k+1} - \phi^T_k) \|^2_F  + 2c^2_2 + c^2_3 f^2_{\max} + 8c^2_3 f^2_{\max} + c^2_3 }\\
        &\stackrel{(b)}{\le} \sqrt{ c^2_1  + 1 + (\| \phi_k\phi^T_{k+1}\|_F + \|\phi_k\phi^T_k\|_F)^2  + 2c^2_2 + 10c^2_3 }\\
        %&= \sqrt{ c^2_1  + 1 + (\| \phi_k\phi^T_{k+1}\|_2 + \|\phi_k\phi^T_k\|_2)^2  + 16c^2_2 r^2_{\max} + c^2_2 }\\
        %&\le \sqrt{ c^2_1  + 1 + (\| \phi_k\phi^T_{k+1}\|_F + \|\phi_k\phi^T_k\|_F)^2  + 16c^2_2 r^2_{\max} + c^2_2 }\\
        &\stackrel{(c)}{=} \sqrt{ c^2_1  + 1 + (\|\phi_k\|_2\| \phi_{k+1}\|_2 + \|\phi_k\|_2\|\phi_k\|_2 )^2 + 2c^2_2 + 10 c^2_3 }\\
        &\stackrel{(d)}{\le} \sqrt{ c^2_1  + 5 + 2c^2_2 + 10c^2_3 },
    \end{align*}
    where $(a)$ follows since for a matrix $M$, $\|M\|_2 \le \|M\|_F$, $(b)$ uses triangle inequality for $\|\cdot\|_F$ and that $f_{\max} = 1$, and $(d)$ uses Assumption~\ref{asmp:2}. To conclude $(c)$, let $x_1 \in \R^d$ and $x_2\in\R^d$. Observe the following equalities for $\|x_1 x^T_2\|_F$.
    \[ \|x_1 x^T_2\|_F = \sqrt{\operatorname{Trace}(x_1 x^T_2 x_2 x^T_1)} = \sqrt{\operatorname{Trace}(x_1x^T_1 \|x_2\|^2_2)} = \|x_2\|_2 (x^T_1 x_1) = \|x_1\|_2\|x_2\|_2 . \]
    Next, to bound ${b}(Y_t)$,
    \begin{align*}
        \|{b}(Y_t) \|_2
        = \|\Pi~ \tilde{b}(Y_t) \|_2
        \le \|\tilde{b}(Y_t)\|_2
        = \sqrt{c^2_1  + 1  + {c^2_3}},
    \end{align*}
    where we again used that $f_{\max} = 1$.    Thus, \begin{equation*}
        \eta :=  \sqrt{ c^2_1  + 5 + 2c^2_2 + 10c^2_3}
    \end{equation*}
    is the required upper bound $\upb$.

\section{Approximation Error due to LFA: Proof of Proposition~\ref{prop:approxerror}}\label{app:approxerr}

\vspace{0.75em}
\noindent{\bf Approximation error in $V$ estimation. } Since any element from $S_f$ is a valid estimator for $V$ function, and since from Theorem~\ref{th:finiteboundlfa}, the proposed algorithm's estimate for $V$ converges to $\Phi\theta^*$, where $\theta^*$ is the unique vector in the subspace $E$ that satisfies~\eqref{eq:projBellman}, the approximation error incurred in estimating $V$ using the proposed TD-like algorithm from Section~\ref{sec:alglfa} is the minimum weighted distance between the estimate using limit of the algorithm, i.e. $\Phi\theta^*$, and the set $S_f$, given by
$${\mathcal E}_{D_\pi, S_f} \lrs{\Phi\theta^*} := \inf\limits_{V\in S_f} \| \Phi\theta^* -  V\|^2_{D_\pi}.$$ 

Following arguments similar to those in \cite{tsitsiklis1999average}, one can relate this error to the minimum possible approximation error (using any algorithm) incurred due to the chosen architecture, 
$$\mathcal E := \inf\limits_{\theta \in \R^d}\inf\limits_{V\in S_f}\|\Phi \theta - V\|_{D_\pi}.$$ 
In particular, it can be shown that for a constant $\lambda \in (0,1)$,   
\begin{align*} 
{\mathcal E}_{D_\pi, S_f} \lrs{\Phi\theta^*} \le \frac{\mathcal E}{\sqrt{1-\lambda^2}} .
\end{align*}
Note that the bound above is a blow-up of the minimum error possible due to the chosen architecture for approximation.  

\vspace{0.75em}
\noindent{\bf Approximation error in variance estimation. } Using this, we now bound the approximation error for the variance estimation using the proposed algorithm. Recall from~\eqref{eq:kmu2} that for any $V \in S_f$,
\[ \kappa(f) = 2\Exp{\pi}{(f(X) - \bar{f})V(X)} - \Exp{\pi}{(f(X) - \bar{f})^2}.\]
Moreover, from~\eqref{eq:kappastar}, one can see that 
\[ \kappa^*= 2\Exp{\pi}{(f(X) - \bar{f})[\Phi\theta^*](X)} - \Exp{\pi}{(f(X) -\bar{f})^2}. \]
Thus, for any $c\in \R$, 
\begin{align*}
    \abs{\kappa(f) - \kappa^*} 
    &= 2 \abs{\Exp{\pi}{\lrp{f(X) \!-\! \bar{f}}\lrp{V^*(X) + c \! - \! \lrs{\Phi\theta^*}(X)}}}\\
    &= 2 \abs{\left<\mathcal F - \bar{f} {\bf e} , V^* + c{\bf e}- \Phi\theta^* \right>_{D_\pi }}\\
    &= 2 \abs{\left< {D^\frac{1}{2}_\pi} \lrp{\mathcal F - \bar{f} {\bf e}} , {D^\frac{1}{2}_\pi }\lrp{V^* + c{\bf e} - \Phi\theta^*} \right>}\\
    &\le 2 \| D^\frac{1}{2}_{\pi}\lrp{\mathcal F - \bar{f} {\bf e}} \|_2~ \| D^\frac{1}{2}_\pi\lrp{V^* + c {\bf e}- \Phi  \theta^*} \|_2\\
    &\le 4 f_{\max}~\| V^* + c{\bf e} - \Phi \theta^* \|_{D_\pi}\\
    &= 4~\| V^* + c{\bf e} - \Phi \theta^* \|_{D_\pi}.
\end{align*}
Since the above inequality is true for all $c\in\R$, the approximation error for $\kappa(f)$ is bounded as below: 
\begin{align*} 
    (\kappa(f) - \kappa^*)^2 
    &\le 16  ~\inf\limits_{c\in\R}~\| \Phi\theta^* - V^* - c{\bf e} \|^2_{D_\pi} \\
    &= 16  ~\inf\limits_{V\in S_f}~\| \Phi\theta^* - V\|^2_{D_\pi}\\
    &=  16  ~\lrp{{\mathcal E}_{D_\pi,S_f}\lrs{\Phi\theta^*}}^2\\
    &\le\frac{16 ~{\mathcal E}^2}{1-\lambda^2}. \qed 
\end{align*}

%\section{Details from Section~\ref{sec:RL} on Applications in RL}\label{sec:RL_app}

\section{Policy Evaluation with Linear Function Approximation in RL}\label{sec:lfa_RL}
When the state and action spaces of the underlying MDP are large, estimating the $Q_\mu$ function for each state-action pair (as in Algorithm~\ref{alg:PE_RL}) requires a lot of memory and may be intractable. To address this, we consider estimating an approximation of $\kappa_\mu$ by extending the setup of Section~\ref{sec:lfa} to RL, by approximating $Q_\mu$ with its projection on to a linear subspace spanned by a given fixed set of $d$ vectors. We discuss this framework in the current section.

\subsection{Setup}
Let us again denote the given set of $d$ vectors by $\{\tilde{\phi}_1, \dots, \tilde{\phi}_d\}$, where $\tilde{\phi}_i \in \R^{\abs{\calS}\abs{\calA}}$ for $i\in [d]$. In particular, for $(s,a) \in \calS\times\calA$, and $\theta\in \R^d$, we consider a linear function approximation 
$$Q_\theta(s,a) = \phi(s,a)^T\theta$$ 
of $Q_\mu(s,a)$, where 
$$\phi^T(s,a) := [\tilde{\phi}_1(s,a)  \dots \tilde{\phi}_d(s,a)]$$ 
is the feature vector for state-action pair $(s,a)$ and $\phi(s,a)\in\R^d$. With this notation, let $\Phi$ be a $\abs{\calS}\abs{\calA} \times d$ matrix with $\tilde{\phi}_i$ being the $i^{\text{th}}$ column. Let $W_\Phi = \{\Phi\theta : \theta \in \R^d\}$ denote the column space of $\Phi$. Then, $Q_\theta = \Phi \theta$, where $Q_\theta \in \R^{\abs{\calS}\abs{\calA}\times 1}$ is an approximation for $Q_\mu$ using $\theta$. We make the following assumption on the feature matrix $\Phi$, which can be achieved by feature normalization, and is standard in literature; see \cite{tsitsiklis1999average,BertsekasT96}.

\begin{assumption}\label{asmp:2_RL}
    \emph{The matrix $\Phi$ is full rank, i.e., the set of feature vectors $\{\tilde{\phi}_1, \dots, \tilde{\phi}_d\}$ are linearly independent. Additionally, $\|\phi(s,a)\|_2 \le 1$, for each $(s,a)\in\calS\times\calA$.}
\end{assumption}

Define subspace $S_{\Phi, e}$ of $\R^d$ as $S_{\Phi,e} := \text{span}\lrp{ \lrset{\theta | \Phi \theta = {\bf e}}}$. It equals $\lrset{c\theta_e | c\in \R}$ if ${\bf e} \in W_\Phi,$ and  $\theta_e$ is the unique vector (if it exists) such that $\Phi\theta_e = {\bf e}$. Otherwise, $S_{\Phi,e} = \{0\}$. Let $E$ be the subspace of $\R^d$ that is orthogonal complement (in $\ell_2$-norm) of $S_{\Phi, e}$, i.e., $ E = \{ \theta \in \R^{d} : ~ \theta^T  \theta_e = 0 \}$ if ${\bf e}\in W_\Phi$. It equals $\R^d$, otherwise. Additionally, let $\Pi_{2,E}$ denote the orthogonal projection of vectors in $\R^d$ (in $2$-norm) on the subspace $E$. 

Observe that for $\theta \in E$, $\Phi \theta$ is a non-constant vector. This follows since there does not exist $\theta\in E$ such that $\Phi\theta = {\bf e}$. In particular, if ${\bf e}\in W_\Phi$, then the unique vector $\theta_e$ such that $\Phi \theta_e = {\bf e}$ does not belong to $E$. Similarly, if ${\bf e}\notin W_\Phi$, then there does not exist a vector $\theta\in\R^d$, hence in $E$, such that $\Phi\theta = {\bf e}$. 

In this setup, let $\theta^*_\mu$ denote the unique vector in $E$ that is a solution for 
\[ \Phi\theta^*_\mu = \Pi_{D_\mu, W_\Phi} T\Phi\theta^*_\mu, \numberthis\label{eq:thetamu_RL}\]
where recall that $D_\mu$ is the diagonal matrix of stationary distribution of the underlying Markov chain (defined around~\eqref{eq:P2}), $\Pi_{D_\mu, W_\Phi}$ is the projection matrix onto the subspace $W_\Phi$ in $D_\mu$ norm, and $T$ is the operator for which $Q_\mu$ is a fixed point, i.e., it acts on vectors $Q \in \R^{|\calS||\calA|}$, and satisfies
\[ TQ = \mathcal R - J_\mu {\bf e} + P_2Q, \]
where $P_2$ is the transition matrix for the underlying Markov chain, defined in~\eqref{eq:P2}. 

Algorithm~\ref{alg:PE_lfa_RL} in the next section, estimates $\theta^*_\mu$ using stochastic approximation for~\eqref{eq:thetamu_RL}, and estimates 
\[ \kappa^*_\mu =  \Exp{d_\mu}{2r(S,A)[\Phi\theta^*_\mu](S,A) - 2 r(S,A)\widetilde{Q}_\mu - r^2(S,A) + r(S,A)J_\mu},\]
where $\widetilde{Q}_\mu := \Exp{d_\mu}{[\Phi\theta^*_\mu](S,A)}$ as a proxy for $\kappa_\mu$. Observe that $\kappa^*_\mu$ differs from $\kappa_\mu$ in that $Q_\mu$ and $\bar{Q}_\mu$ in formulation~\eqref{eq:kappaRL} are replaced by the estimate in the subspace spanned by $\Phi$ and $\widetilde{Q}$, respectively. 

\subsection{Algorithm}\label{app:alglfa_RL}
We now adapt the Algorithm from Section~\ref{sec:alglfa} to RL for estimating the asymptotic variance of a given stationary  policy $\mu$ with linear function approximation. Here, we estimate a vector $\theta \in \R^d$ at each step; call this estimate at step $k$ as $\theta_k$. The corresponding estimate for $Q_\mu$ is then $\Phi\theta_k$. 

\begin{algorithm2e}[bt]
\RestyleAlgo{ruled}
\DontPrintSemicolon
\SetNoFillComment
\SetKwInOut{Input}{Input}\SetKwInOut{Output}{Output}
%\setstretch{1.05}
\caption{Estimating $\kappa_\mu$: Linear Function Approximation Setting}\label{alg:PE_lfa_RL}
    %\vspace{0.4em}
    \BlankLine

    \KwIn{Time horizon $n > 0$, constants $c_1 > 0$, $c_2 > 0$, $c_3 > 0$ and step-size sequence $\lrset{\alpha_k}$.}
    \BlankLine

    {\bf Initialization: } $\bar{r}_0  = 0$, $\theta_0 = \vec{\bf  0} $, $\widetilde{Q}_0 = 0$, and $\kappa_0 =0$.\\
    \BlankLine

\For{$k\leftarrow 1$ \KwTo $n$}{
\BlankLine

Observe ($S_k, A_k, r(S_k,X_k), S_{k+1}, A_{k+1}$).\;

\BlankLine

\tcp{\textbf{Average reward $J_\mu$ estimation}} 
        \vspace{-1.5em}
        \begin{align*}
            \bar{r}_{k+1} = \bar{r}_k + c_1\alpha_k (  r(S_k,A_k) -\bar{r}_k ).
        \end{align*}
\;
        \vspace{-1.5em}

\tcp{\textbf{$\theta^*_\mu$ estimation}}
\vspace{-1.5em}
    \begin{align*}
    & \delta_k :=  r(S_k,A_k) - \bar{r}_{k} + (\phi(S_{k+1},A_{k+1}) - \phi(S_k, A_k))^T\theta_k,\\
    &\theta_{k+1} = \theta_k + \alpha_k \Pi_{2,E}\phi(S_k,A_k) \delta_k.
    \end{align*}
\;
\vspace{-1.5em}
\tcp{\textbf{$\widetilde{Q}_\mu$ estimation}}
\vspace{-1.5em}
    \begin{align*}
        &\widetilde{Q}_{k+1} = (1-c_2\alpha_k)\widetilde{Q}_k + c_2 \alpha_k \phi^T(S_k,A_k)\theta_k.
    \end{align*}
\;
\vspace{-1.5em}

\tcp{\textbf{Asymptotic Variance $\kappa_\mu$ estimation}}
\vspace{-1.5em}
        \begin{align*}
        \kappa_{k+1} =  (1-c_3\alpha_k) \kappa_k & + c_3\alpha_k\big( 2r(S_k,A_k)\phi^T(S_k,A_k)\theta_k \\
        &\qquad\qquad \qquad - 2r(S_k,A_k)\widetilde{Q}_k - r^2(S_k,A_k) + r(S_k,A_k)\bar{r}_k \big).
        \end{align*}
        \vspace{-1.5em}
    }
 \Output{Estimates $\kappa_n$, $\widetilde{Q}_n$, $\theta_n$, $\bar{r}_n$} 
    
\end{algorithm2e}
%Here, updates in~\eqref{eq:lfadeltaUpdate} and~\eqref{eq:lfaThetaUpdate} correspond to SA for estimating $\theta$, adjusted with the projection on subspace $E$. As in Section~\ref{sec:varest:tabular}, this projection step is to ensure convergence of the iterates $\widetilde{V}_k$, and is not required otherwise (Remark~\ref{rem:proj}). Note that Algorithm~\ref{alg:PE} for the tabular setting is a special case with $d = S$ and $\Phi = I$, the identity matrix in $S\times S$ dimensions.

Note that Algorithm~\ref{alg:PE_RL} for the tabular setting is a special case with $d = \abs{\mathcal S}\abs{\mathcal A}$ and $\Phi = I$, the identity matrix in $\abs{\mathcal S}\abs{\mathcal A}$ dimensions. We conclude this section with a finite-sample bound on the estimation error of the proposed algorithm with linear function approximation. 

\subsection{Convergence Rates}
Let $\Theta^T_k := [r_k ~ \theta^T_k ~ \widetilde{Q}_k~\kappa_k]$ denote the vector of estimates of Algortithm~\ref{alg:PE_lfa_RL} at time $k$. Recall the Markov chain $\mathcal M_3$ introduced in Section~\ref{sec:RL_bounds}. Define 
\[ \tilde{\Delta}_2 := \min\limits_{\|\theta\|_2 = 1, \theta \in E} ~ \theta^T \Phi^T D_\mu(I-P_2) \Phi \theta. \numberthis\label{eq:tildedelta2} \]
As earlier, this can be shown to be strictly positive since for $\theta\in E$, $\Phi \theta$ is a non-constant vector. Finally, let $\eta:= (c^2_1+5+2c^2_2+10c^2_3)^\frac{1}{2}$, and let $\Theta^{*T}_\mu := [J_\mu~ \theta^{*T}_\mu ~ 
\widetilde{Q}_\mu~ \kappa^*_\mu ]$, 
where 
each of these terms is defined around~\eqref{eq:thetamu_RL}. The following theorem bounds the mean-squared distance between $\Theta_k$, the vector of estimates of  Algorthm~\ref{alg:PE_lfa_RL}  at time $k$,  and its limit point $\Theta^*_\mu$, under constant as well as diminishing step sizes. 

\begin{theorem}\label{th:finiteboundlfa_RL}
    Consider estimates ${\Theta}_k$ generated by Algorithm~\ref{alg:PE_lfa_RL} with $c_1$, $c_3$, and $c_2$ satisfying:
    \[c_1\ge \frac{1}{2\tilde{\Delta}_2} + \frac{\tilde{\Delta}_2}{2}, \qquad  \frac{5}{249}(5-2\sqrt{2})\tilde{\Delta}_2 \le c_3\le \frac{5}{249} (5+2\sqrt{2})\tilde{\Delta}_2, \]
    and
    \[ c_3 -\frac{498 c^2_3}{7\tilde{\Delta}_2} + \frac{7 \tilde{\Delta}_2}{498}   \le c_2 \le -3 c_3 + \frac{5 }{83} \sqrt{498 c_3 \tilde{\Delta}_2 - 17 \tilde{\Delta}^2_2} . \]
    For a constant $B > 1$, let  $\xi_1 :=  3\lrp{1+\|\Theta^*_\mu\|_2}^2$, $\xi_2 := 112B(1+\|\Theta^*_\mu\|_2)^2$, where recall that $\eta$ and $\Theta^*_\mu$ were defined below~\eqref{eq:tildedelta2}.
    \begin{enumerate}%[label=(\alph*)]
        \item[(a)] Let $\alpha_i = \alpha$ for all $i$, such that     
        \[ \alpha < \min\lrset{ \frac{40}{\tilde{\Delta}_2}, \frac{1}{28B\lrp{1+ \frac{20\eta^2}{\tilde{\Delta}_2} }}}. \]
        Then, for all $k\ge 1$, 
        \[ \E{\|\Theta_k - \Theta^*_\mu\|^2_2} \le \xi_1\lrp{1-\frac{\tilde{\Delta}_2\alpha}{40}}^{k} +  \frac{20\xi_2 \alpha\eta^2}{\tilde{\Delta}_2} + \xi_2\alpha. \]
        \item[(b)]  Let $\alpha_i = \frac{\alpha}{i+h}$ for all $i \ge 1$, with $\alpha$ and $h$ chosen so that
        \[ h \ge \max\lrset{1+\frac{\alpha\tilde{\Delta}_2}{40}, 1+28B\lrp{ \frac{20\eta^2\alpha}{\tilde{\Delta}_2} + \alpha + \frac{20}{\tilde{\Delta}_2} }},\quad\text{and}\quad \alpha > \frac{40}{\tilde{\Delta}_2}.\]
        Then, for all $k\ge 0$, 
        \begin{align*}
        \E{\|\Theta_k- \Theta^*_\mu\|^2_2}  &\le \xi_1\lrp{\frac{h}{k+h}}^\frac{\alpha\tilde{\Delta}_2}{40} + \frac{5\xi_2e^2\eta^2(20+\tilde{\Delta}_2)\alpha^2}{(k+h)(\alpha\tilde{\Delta}_2 -40)} + \frac{\xi_2 \alpha}{k+h}.
        \end{align*} 
    \end{enumerate}
\end{theorem}

\noindent\emph{Proof sketch. }
    The bounds in this theorem follow immediately from those in Theorem~\ref{th:finiteboundlfa} by observing that the setting considered here corresponds to that of Theorem~\ref{th:finiteboundlfa} with $r(\cdot,\cdot)$, $d_\mu$, $\theta^*_\mu$, $\widetilde{Q}_\mu$, and $\kappa^*_\mu$ replacing $f(\cdot)$, $\pi$, $\theta^*$, $\widetilde{V}$,  and $\kappa^*$, respectively.
\qed

\begin{remark}
As in Corollary~\ref{cor:MSE_PE}, Theorem~\ref{th:appfiniteboundlfa} $(b)$ gives a MSE bound of $O(\frac{1}{n})$ for policy evaluation problem in average reward RL setting with LFA, improving the result of \cite{zhang2021finite} by a multiplicative factor of $\log n$.
\end{remark}

\section{Auxiliary Technical Lemmas}
\begin{lemma}\label{lem:diffinstepsize}
    For any $\alpha > 0$ and $h\ge 2$, for all $k \ge 1$, and $\alpha_k = \frac{\alpha}{k+h}$ or $\alpha_k = \alpha$, we have
    \[ \alpha_{k-1} - \alpha_k \le \frac{3\alpha^2_k}{\alpha}. \]
\end{lemma}
\begin{proof}
    The result is trivially true for $\alpha_k$ being a constant ($\alpha_k = \alpha$). Hence, we only provide a proof for the setting of a diminishing step-size ($\alpha_k = \frac{\alpha}{k+h}$). To this end, consider the following: 
    \begin{align*}
        \alpha_{k-1} - \alpha_k 
        &= \frac{\alpha}{k+h-1} - \frac{\alpha}{k+h} \\
        &= \frac{\alpha}{(k+h)(k+h-1)}\\
        &= \frac{\alpha^2_k}{\alpha\lrp{1-\frac{1}{k+h}}}\\
        &\le \frac{\alpha^2_k}{\alpha}\lrp{1+\frac{2}{k+h}} \\
        &\le \frac{3\alpha^2_k}{\alpha},
    \end{align*}
    where we used $\frac{1}{1-x}\le 1+2x$ for $x\le \frac{1}{2}$ in the first inequality above. 
\end{proof}

\begin{lemma}\label{lem:diffalpha} For any $\alpha > 0$, and either $\alpha_k = \alpha$ for all $k\ge 0$, or for $h\ge 2$, $\alpha_k = \frac{\alpha}{k+h}$, for all $k\ge 0$, we have
\[\alpha_k \alpha_{k-1} \le 2 \alpha^2_k.\]\end{lemma}
\begin{proof}
    This is trivially true for $\alpha_k = \alpha$. Below, we show for diminishing step sizes of the form $\alpha_k = \frac{\alpha}{k+h}$, for $h\ge 2$. Consider the following:
    \begin{align*}
        \alpha_k \alpha_{k-1} &= \frac{\alpha^2_k}{1-\frac{1}{k+h}}\le \frac{\alpha^2_k}{1-\frac{1}{h}}\le 2\alpha^2_k,
    \end{align*}
    proving the lemma.
\end{proof}

\begin{lemma}\label{lem:aux1}
    For $b\ge 0$, and $c \ge 0$,
    \[ \min\limits_{x\in [0,c]}\lrset{ax - b\sqrt{cx-x^2}} = \frac{c}{2}\lrp{a-\sqrt{a^2 + b^2}}. \]
\end{lemma}
\begin{proof}
    Let $f(x) := ax - b\sqrt{cx - x^2}$. Let $f'(x)$ denote the derivative of $f(x)$ and $f^{''}(x)$ denote the corresponding second derivative, both evaluated at $x$. Then, 
    \[ f'(x) = a - \frac{b(c-2x)}{2\sqrt{cx - x^2}}, \]
    and $f^{''}(x) > 0$ for $b\ge 0$. Solving for x such that $f'(x) = 0$, we get $x_1$ defined below as the optimizer. 
    \[ x_1 := \frac{c}{2}\lrp{1-\frac{a}{\sqrt{a^2 + b^2}}}, \]
    and $f(x_1)$ is the desired optimal value.
\end{proof}